%% file: multidim-lipschitz-minlp-preprint.tex
\documentclass[a4paper,american,reqno]{amsart}

\newif\ifAppendix \Appendixtrue
\newif\ifPreprint \Preprinttrue
\newif\ifSubmission \Submissionfalse

\input{setup-preprint}
\bibliography{multidim-lipschitz-minlp}

\begin{document}

\title[An SLR Method for MINLPs with Multivariate Lipschitz
Constraints]%
{A Successive Linear Relaxation Method for MINLPs with Multivariate
  Lipschitz Continuous Nonlinearities}
% with Applications to Bilevel Optimization and Gas Transport

\author[J. Grübel, R. Krug, M. Schmidt, W. Wollner]%
{Julia Grübel, Richard Krug, Martin Schmidt, Winnifried Wollner}

\address[J. Gr\"ubel]{
  (a) Friedrich-Alexander-Universit\"at Erlangen-Nürn\-berg (FAU),
  Chair of Economic Theory,
  Lange Gasse~20, 90403 N\"urnberg, Germany,
  (b) Energie Campus N\"urnberg,
  F\"urther Str.~250, 90429~N\"urnberg, Germany}
\email{julia.gruebel@fau.de}

\address[R. Krug]{%
  Friedrich-Alexander-Universit\"at Erlangen-Nürn\-berg (FAU),
  Department of Data Science,
  Cauerstr.~11, 91058~Erlangen,
  Germany
}
\email{richard.krug@fau.de}

\address[M. Schmidt]{%
  Trier University,
  Department of Mathematics,
  Universitätsring~15,
  54296~Trier,
  Germany
}
\email{martin.schmidt@uni-trier.de}

\address[W. Wollner]{%
  Universit\"at Hamburg,
  MIN Fakul\"at,
  Fachbereich Mathematik,
  Bundesstr.~55,
  20146~Hamburg,
  Germany
}
\email{winnifried.wollner@uni-hamburg.de}

\date{\today}

\begin{abstract}
  \input{abstract}
\end{abstract}

\keywords{\input{keywords}}
\subjclass[2020]{\input{msc2020}}

\maketitle

\input{introduction}
\input{problem-definition}
\input{algorithm}
\input{case-study-bilevel}
\input{case-study-gas}
\input{conclusion}
\input{acknowledgements}

\ifAppendix
\appendix
\input{appendix}
\fi

\printbibliography

\end{document}

%% file: setup-preprint.tex
\usepackage{babel}
\usepackage[utf8]{inputenc}
\usepackage[T1]{fontenc}
\usepackage{relsize}
\usepackage[binary-units=true]{siunitx}
\usepackage{booktabs}
\usepackage[left=3.75cm,right=3.75cm]{geometry}
\usepackage{xspace}
\usepackage{accents}
\usepackage{algorithm,algorithmic}
\usepackage{todonotes}
\usepackage{pgfplots}
\usepackage{tikz}
\usepackage{gas-tikz}
\usepackage{csquotes}
\usepackage{microtype}
\usepackage[style = numeric-comp,
            maxbibnames = 100,
            maxcitenames = 2,
            isbn = false,
            eprint = true,
            url = true,
            giveninits = true,
            doi = true,
            backend = bibtex]{biblatex}
\usepackage[colorlinks,
            citecolor=blue,
            urlcolor=blue,
            linkcolor=blue]{hyperref} % should always be the last package

\makeatletter
\patchcmd{\@settitle}{\uppercasenonmath\@title}{\scshape\large}{}{}
\patchcmd{\@setauthors}{\MakeUppercase}{\scshape\normalsize}{}{}
\makeatother

\vfuzz \hfuzz
\makeatletter
\@namedef{subjclassname@2020}{%
  \textup{2020} Mathematics Subject Classification}
\makeatother

\newtheorem{theorem}{Theorem}[section]
\newtheorem{lemma}[theorem]{Lemma}
\newtheorem{proposition}[theorem]{Proposition}
\newtheorem{assumption}{Assumption}
\theoremstyle{definition}

\theoremstyle{remark}
\newtheorem{remark}[theorem]{Remark}

\usetikzlibrary{external}

{multidim-lipschitz-minlp-preprint} % provide the file~s real name

\tikzexternaldisable

\pgfkeys{/pgf/number format/.cd,1000 sep={}}

\input{macros}

%% file: macros.tex
\newcommand{\ulVar}{x}
\newcommand{\llVar}{y}
\newcommand{\ulIdx}{\mathrm{u}}
\newcommand{\llIdx}{\mathrm{l}}
\newcommand{\lowerLLVar}{\tilde{\llVar}}
\newcommand{\dimUL}{{n_\ulVar}}
\newcommand{\dimLL}{{n_\llVar}}
\newcommand{\nrConstrUL}{{m_\ulIdx}}
\newcommand{\nrConstrLL}{{m_\llIdx}}

\newcommand{\st}{\text{s.t.}}

\newcommand{\field}{\mathbb}
\newcommand{\R}{\field{R}}
\newcommand{\Z}{\field{Z}}

\newcommand{\N}{\field{N}}

\newcommand{\abs}[1]{|{#1}|}
\newcommand{\Abs}[1]{\left|{#1}\right|}
\newcommand{\card}{\abs}

\newcommand{\defsep}{\colon}% alternative: \mid
\newcommand{\defset}[3][\defsep]{\set{#2#1#3}}
\newcommand{\Defset}[3][\defsep]{\Set{#2#1#3}}
\newcommand{\norm}[1]{\|{#1}\|}
\newcommand{\Norm}[1]{\left\|{#1}\right\|}
\newcommand{\set}[1]{\{#1\}}
\newcommand{\Set}[1]{\left\{#1\right\}}

\newcommand{\feasibleset}{\mathcal{F}}
\newcommand{\feasset}{\feasibleset}

\newcommand{\eps}{\varepsilon}

\newcommand{\abbr}[1][abbrev]{#1.\xspace}

\newcommand{\eg}{\abbr[e.g]}

\newcommand{\ie}{\abbr[i.e]}

\newcommand{\wrt}{\abbr[w.r.t]}

\makeatletter
\newcommand{\objref}[4]{\def\obj@rg{#4}%
  #1\ifx\obj@rg\empty#2\else#3\xspace\ref{#4}--\fi\ref}
\makeatother

\newcommand{\Probref}[1]{\text{Problem}~\eqref{#1}}

\DeclareMathOperator{\graph}{graph}
\DeclareMathOperator{\sign}{sign}
\DeclareMathOperator{\interior}{int}

\newcommand{\define}{\mathrel{{\mathop:}{=}}}
\newcommand{\enifed}{\mathrel{{=}{\mathop:}}}

\newcommand{\ubar}[1]{\underaccent{\bar}{#1}}

\newcommand{\lb}[1]{\ubar{{#1}}}
\newcommand{\ub}[1]{\bar{{#1}}}

\newcommand\rhombus[7][]{
  \FPset\xOne{#2}
  \FPset\yOne{#3}
  \FPset\xTwo{#4}
  \FPset\yTwo{#5}
  \FPset\L{#6}

  \FPeval\xU{((yTwo + L*xTwo) - (yOne - L*xOne)) / (2*L)}
  \FPeval\yU{L * xU + (yOne - L*xOne)}
  \FPeval\xL{((yOne + L*xOne) - (yTwo - L*xTwo)) / (2*L)}
  \FPeval\yL{L * xL + (yTwo - L*xTwo)}

  \path[fill=#7]
  (\FPprint\xOne,\FPprint\yOne)
  -- (\FPprint\xU,\FPprint\yU)
  -- (\FPprint\xTwo,\FPprint\yTwo)
  -- (\FPprint\xL,\FPprint\yL)
  -- cycle;

  \ifthenelse{\equal{#1}{coordinates}}{
    \FPround\xUR\xU{4}
    \FPround\yUR\yU{4}
    \FPround\xLR\xL{4}
    \FPround\yLR\yL{4}
    \node[fill=black, circle, inner sep=0.05cm] (upper) at (\FPprint\xU,\FPprint\yU) {};
    \node[above=1pt of upper] () {$(\FPprint\xUR,\FPprint\yUR)$};
    \node[fill=black, circle, inner sep=0.05cm] (lower) at (\FPprint\xL,\FPprint\yL) {};
    \node[below=1pt of lower] () {$(\FPprint\xLR,\FPprint\yLR)$};
  }{}
}

\newcommand\gapRhombus[8][]{
  \FPset\xOne{#2}
  \FPset\yOne{#3}
  \FPset\xTwo{#4}
  \FPset\yTwo{#5}
  \FPset\L{#6}
  \FPset\eps{#8}

  \FPeval\yOnePlusEps{yOne + eps}
  \FPeval\yTwoPlusEps{yTwo + eps}
  \FPeval\yOneMinusEps{yOne - eps}
  \FPeval\yTwoMinusEps{yTwo - eps}

  \FPeval\xU{((yTwo + L*xTwo) - (yOne - L*xOne)) / (2*L)}
  \FPeval\yU{L * xU + (yOne - L*xOne) + eps}
  \FPeval\xL{((yOne + L*xOne) - (yTwo - L*xTwo)) / (2*L)}
  \FPeval\yL{L * xL + (yTwo - L*xTwo) - eps}

  \path[fill=#7]
  (\FPprint\xOne,\FPprint\yOnePlusEps)
  -- (\FPprint\xU,\FPprint\yU)
  -- (\FPprint\xTwo,\FPprint\yTwoPlusEps)
  -- (\FPprint\xTwo,\FPprint\yTwoMinusEps)
  -- (\FPprint\xL,\FPprint\yL)
  -- (\FPprint\xOne,\FPprint\yOneMinusEps)
  -- cycle;

  \ifthenelse{\equal{#1}{coordinates}}{
    \FPround\xUR\xU{4}
    \FPround\yUR\yU{4}
    \FPround\xLR\xL{4}
    \FPround\yLR\yL{4}
    \node[fill=black, circle, inner sep=0.05cm] (upper) at (\FPprint\xU,\FPprint\yU) {};
    \node[above=1pt of upper] () {$(\FPprint\xUR,\FPprint\yUR)$};
    \node[fill=black, circle, inner sep=0.05cm] (lower) at (\FPprint\xL,\FPprint\yL) {};
    \node[below=1pt of lower] () {$(\FPprint\xLR,\FPprint\yLR)$};
  }{}
}

\newcommand\rhombusWithYBounds[9][]{
  \FPset\xOne{#2}
  \FPset\yOne{#3}
  \FPset\xTwo{#4}
  \FPset\yTwo{#5}
  \FPset\L{#6}
  \FPset\yLB{#8}
  \FPset\yUB{#9}

  \FPeval\xU{((yTwo + L*xTwo) - (yOne - L*xOne)) / (2*L)}
  \FPeval\yU{L * xU + (yOne - L*xOne)}
  \FPeval\xL{((yOne + L*xOne) - (yTwo - L*xTwo)) / (2*L)}
  \FPeval\yL{L * xL + (yTwo - L*xTwo)}

  \ifthenelse{\lengthtest{\yUB pt > \yU pt}}{
    \FPeval\xULeft{xU}
    \FPeval\yULeft{yU}

    \FPeval\xURight{xU}
    \FPeval\yURight{yU}
  }
  {
    \FPeval\xULeft{(yUB - (yOne - L*xOne)) / (L)}
    \FPeval\yULeft{(L * xULeft) + (yOne - L*xOne)}

    \FPeval\xURight{(yUB - (yTwo + L*xTwo)) / (-L)}
    \FPeval\yURight{(-1.0 * L * xURight) + (yTwo + L*xTwo)}
  }

  \ifthenelse{\lengthtest{\yLB pt < \yL pt}}{
    \FPeval\xLRight{xL}
    \FPeval\yLRight{yL}
    \FPeval\xLLeft{xL}
    \FPeval\yLLeft{yL}
  }
  {
    \FPeval\xLRight{(yLB - (yTwo - L*xTwo)) / (L)}
    \FPeval\yLRight{(L * xLRight) + (yTwo - L*xTwo)}

    \FPeval\xLLeft{(yLB - (yOne + L*xOne)) / (-L)}
    \FPeval\yLLeft{(-1.0 * L * xLLeft) + (yOne + L*xOne)}

  }

  \path[fill=#7]
  (\FPprint\xOne,\FPprint\yOne)
  -- (\FPprint\xULeft,\FPprint\yULeft)
  -- (\FPprint\xURight,\FPprint\yURight)
  -- (\FPprint\xTwo,\FPprint\yTwo)
  -- (\FPprint\xLRight,\FPprint\yLRight)
  -- (\FPprint\xLLeft,\FPprint\yLLeft)
  -- cycle;
}

\newcommand{\ForNeNotation}[1]{\ensuremath{#1}}

\newcommand{\nodes}{\ForNeNotation{V}\xspace}
\newcommand{\arcs}{\ForNeNotation{A}\xspace}

\newcommand{\node}{\ForNeNotation{u}\xspace}

\newcommand{\otherNode}{\ForNeNotation{v}\xspace}

\newcommand{\arc}{\ForNeNotation{a}\xspace}

\newcommand{\inArcs}[1][\node]{\ForNeNotation{\delta^{\In}(#1)}\xspace}
\newcommand{\outArcs}[1][\node]{\ForNeNotation{\delta^{\Out}(#1)}\xspace}

\ifcase1 % MCS 2013-03-09

\or

\fi

\newcommand{\pressure}{\ForNeNotation{p}\xspace}
\newcommand{\press}{\pressure}
\newcommand{\pseudocriticalPressure}{\ForNeNotation{\press_\mathrm{c}}\xspace}

\newcommand{\massflow}{\ForNeNotation{q}\xspace}
\newcommand{\mflow}{\massflow}

\newcommand{\massflux}{\ForNeNotation{\chi}\xspace}
\newcommand{\mflux}{\massflux} %% abbreviaton

\newcommand{\volumetricFlow}[1][]{\volumetricFlow[#1]}
\newcommand{\friction}{\ForNeNotation{\lambda}\xspace}

\newcommand{\nfric}{\ForNeNotation{\theta}\xspace}

\newcommand{\compressibilityfactor}{\ForNeNotation{z}\xspace}
\newcommand{\compfactor}{\compressibilityfactor}

\newcommand{\squaredMachNumber}{\ForNeNotation{\eta}\xspace}
\newcommand{\mach}{\squaredMachNumber}
\newcommand{\soundspeed}{\ForNeNotation{c}\xspace}

\newcommand{\pipe}{\mathrm{pi}}
\newcommand{\shortCut}{\mathrm{sp}}
\newcommand{\shortPipe}{\shortCut}
\newcommand{\cv}{\mathrm{cv}}
\newcommand{\cs}{\mathrm{cs}}

\newcommand{\va}{\mathrm{va}}

\newcommand{\arcsPipes}{\ForNeNotation{\arcs_\pipe}\xspace}
\newcommand{\arcsShortPipes}{\ForNeNotation{\arcs_\shortPipe}\xspace}

\newcommand{\arcsCV}{\ForNeNotation{\arcs_\cv}\xspace}

\newcommand{\arcsCS}{\ForNeNotation{\arcs_\cs}\xspace}

\newcommand{\arcsVA}{\ForNeNotation{\arcs_\va}\xspace}

\newcommand{\pipes}{\arcsPipes}
\newcommand{\shortPipes}{\arcsShortPipes}
\newcommand{\valves}{\arcsVA}
\newcommand{\controlValves}{\arcsCV}
\newcommand{\compressorStations}{\arcsCS}

\newcommand{\valveOpen}{o}

\newcommand{\In}{\mathrm{in}} % starting with capital letter because \in already
\newcommand{\Out}{\mathrm{out}} % starting with capital letter to have
\newcommand{\temperature}{\ForNeNotation{T}\xspace}
\newcommand{\temp}{\temperature} % conflict with defpattern.sty from fp
\newcommand{\pseudocriticalTemperature}{\ForNeNotation{\temperature_\mathrm{c}}\xspace}

\newcommand{\density}{\rho}
\newcommand{\dens}{\density} %% abbreviation

\newcommand{\velocity}{v}
\newcommand{\specificGasConstant}{R_{\mathrm{s}}}

\newcommand{\length}{\ForNeNotation{L}}

\newcommand{\diameter}{\ForNeNotation{D}}
\newcommand{\area}{\ForNeNotation{A}}

\newcommand{\consFuel}[1][d]{\cons^\mathrm{fuel}}

\newcommand{\slack}{\sigma}

\newcommand{\vectorOfSlackVars}[1]{\ForNeNotation{\vektor{\slack}\ifx#1\empty\else_{#1}\fi}\xspace}

\newcommand{\contractPos}{\ensuremath{c}\xspace}

\newcommand{\nodesInLimitingContractPos}[1][\contractPos]{\ensuremath{V(\limitings)}\xspace}

\newcommand{\limitings}[1][\contractPos]{\ensuremath{\mathcal{L}(#1)}\xspace}

\newcommand{\multiindices}{I_i}
\newcommand{\multind}{\multiindices}
\newcommand{\rightindex}{r_i}
\newcommand{\rind}{\rightindex}
\newcommand{\allindices}{\mathcal{I}_i}
\newcommand{\allind}{\allindices}
\newcommand{\boxset}{\mathcal{X}}
\newcommand{\splitbox}{S}

\DeclareMathOperator*{\argmin}{arg\,min}

\newcommand{\dparlong}[2]{\frac{\partial#2}{\partial#1}}
\newcommand{\dparshort}[2]{\partial_{#1}#2}
\newcommand{\dpar}{\dparlong}

\newcommand{\diff}{\xspace\,\mathrm{d}}

\newcommand{\mfluxaux}{h}
\newcommand{\RsT}{\specificGasConstant \temp}
\newcommand{\plbaux}[1]{#1^2 - \mflux^2 \RsT}
\newcommand{\pubaux}[1]{1 + \alpha #1}
\newcommand{\plb}{\plbaux{\press}}
\newcommand{\plbu}{\plbaux{\press_\node}}

\newcommand{\pub}{\pubaux{\press}}
\newcommand{\pubu}{\pubaux{\press_\node}}

\newcommand{\myeps}{\varepsilon}

\newcommand{\myfunc}[1]{\lipschitzconstant * sin(deg(#1))}
\newcommand{\filllipschitzbox}[4]{
  \addplot[fill=#3, #4, \thickness] coordinates {({(#1) - (#2)/2+\xoffset}, {\myfunc{#1} + \lipschitzconstant * (#2)/2}) ({(#1) + (#2)/2+\xoffset}, {\myfunc{#1} + \lipschitzconstant * (#2)/2}) ({(#1) + (#2)/2+\xoffset}, {\myfunc{#1} - \lipschitzconstant * (#2)/2}) ({(#1) - (#2)/2+\xoffset}, {\myfunc{#1} - \lipschitzconstant * (#2)/2}) ({(#1) - (#2)/2+\xoffset}, {\myfunc{#1} + \lipschitzconstant * (#2)/2})};
}

\providecommand{\sfnamesize}{\relsize{0}}
\providecommand{\sfname}[1]{\textsf{\sfnamesize#1}\xspace}

\newcommand{\CPLEX}{\sfname{CPLEX}}
\newcommand{\SNOPT}{\sfname{SNOPT}}
\newcommand{\BASBLib}{\sfname{BASBLib}}
\newcommand{\GAMS}{\sfname{GAMS}}
\newcommand{\Python}{\sfname{Python}}
\newcommand{\CPU}{\sfname{Intel(R) Core(TM) i7-8550U CPU}}

\newcommand{\rev}[1]{#1}

%% file: abstract.tex
We present a novel method for mixed-integer optimization
problems with multivariate and Lipschitz continuous nonlinearities.
In particular, we do not assume that the nonlinear constraints are
explicitly given but that we can only evaluate them and that we know
their global Lipschitz constants.
The algorithm is a successive linear relaxation method in which we
alternate between solving a master problem, which is a mixed-integer
linear relaxation of the original problem, and a subproblem, which is
designed to tighten the linear relaxation of the next master problem
by using the Lipschitz information about the respective functions.
By doing so, we follow the ideas of Schmidt et al.\ (2018, 2021)
and improve the tackling of multivariate constraints.
Although multivariate nonlinearities obviously increase modeling
capabilities, their incorporation also significantly increases the
computational burden of the proposed algorithm.
We prove the correctness of our method and also derive a worst-case
iteration bound.
Finally, we show the generality of the addressed problem class and the
proposed method by illustrating that both bilevel optimization
problems with nonconvex \rev{and quadratic} lower levels as well as
nonlinear and mixed-integer models of gas transport can be tackled by
our method.
We provide the necessary theory for both applications and brief\/ly
illustrate the outcomes of the new method when applied to these two
problems.

%%% Local Variables:
%%% mode: latex
%%% TeX-master: "multidim-lipschitz-minlp-preprint"
%%% End:

%% file: keywords.tex
Mixed-Integer Nonlinear Optimization,
Global Optimization,
Lipschitz Optimization,
Bilevel Optimization,
Gas Networks%
%
%
%%% Local Variables:
%%% mode: LaTeX
%%% TeX-master: "multidim-lipschitz-minlp-preprint"
%%% End:

%% file: msc2020.tex
90-08, % Operations research, mathematical programming: Computational methods
90C11, % Operations research, mathematical programming: Mixed integer programming
90C26, % Nonconvex programming, global optimization
90C30, % Operations research, mathematical programming: Nonlinear programming
90C90% Operations research, mathematical programming: Applications of mathematical programming
%
%
%%% Local Variables:
%%% mode: LaTeX
%%% TeX-master: "multidim-lipschitz-minlp-preprint"
%%% End:

%% file: introduction.tex
\section{Introduction}
\label{sec:introduction}

Mixed-integer nonlinear optimization problems (MINLPs) form one of
today's most important classes of optimization models.
The reason is, at least, twofold.
First, the capability of modeling nonlinearities allows to include
many sophisticated aspects of, \eg, physics, economics, engineering, or
medicine.
Second, the incorporation of integer variables makes it possible to
model decision making such as turning on or off a machine or investing
in a product or not.
Of course, this modeling flexibility comes at the price of models that
are hard to solve for realistically sized instances since MINLPs are
NP-hard in general \cite{Kannan_Monma:1978,Garey_Johnson:1979}.
Nevertheless, the theoretical and algorithmic advances of the last years
and decades make it possible today to solve rather large-scale
instances to global optimality in a reasonable amount of time
\cite{Belotti2013U}---in particular if problem-specific structural
properties of the model can be exploited algorithmically;
see, \eg,
\cite{Duran1986,Fletcher1994,Bonami2008,KRONQVIST2019105,Kronqvist_et_al:2019}
for the convex as well as
\cite{Al-Khayyal_Sherali:2000,%
  Smith_Pantelides:1997,%
  Tawarmalani_Sahinidis:2002,
  Leyffer_et_al:2008}
for the nonconvex case.

In addition, there has been a significant amount of research
devoted to the cases in which only very few structural assumptions
can be exploited.
This is the framework considered in this paper since we assume that
certain functions of the model are not given explicitly but can be
evaluated and some analytical properties such as Lipschitz continuity
is known.
To illustrate why this is important, let us sketch three areas of
applications in which only few assumptions on the structure of the
model (or on specific parts of the model) can be made.
First, many mixed-integer optimization problems subject to ordinary or
partial differential equations fit into this context.
In many cases, approaches in this field are driven by incorporating
the so-called control-to-state map into the optimization model
to ``eliminate'' the differential equation from the model; see,
\eg, \cite{Buchheim_et_al:2018,BeckerMeidnerVexler:2007,%
  HinzePinnauUlbrichUlbrich:2009,Troeltzsch:2010}.
This mapping, however, cannot be stated explicitly in general and one
thus has to resort to exploiting analytical properties such as Lipschitz
continuity and the ability to evaluate the mapping, at least in an
approximate way.
Second, and rather related to the first example, optimization models
incorporating constraints that rely on calls to expensive simulation
software can be cast in the framework mentioned in this paper as well
if enough analytic information is known about the input-output mapping
of the simulation code; see, \eg,
\cite{Conn_et_al:2009,Bajaj_et_al:20221}.
Third and finally, bilevel optimization problems can also be
interpreted as models in which a single constraint makes the problem
much harder to solve \cite{Dempe:2020,Dempe-et-al:2015,Dempe:2002}.
In this case, it is the constraint that models the optimality of the
decisions of the lower-level (or follower) problem given the decisions
of the leader (which is the decision maker modeled in the upper-level
problem).
The best-reply function, which models the optimal response of the
follower, can usually not be written in closed form but it can be
evaluated (by solving the lower-level problem for a given feasible
point of the upper level) and, under suitable assumptions
\cite{Dempe:2002}, analytical properties such as Lipschitz continuity
can be established.

This paper addresses a special class of MINLPs for which closed-form
expressions for the nonlinearities are not available but Lipschitz
continuity is guaranteed with known Lipschitz constants.
To this end, all three areas of application discussed above can be
addressed by the method proposed in this paper.
One part of our contribution, indeed, is that the case studies
presented later in Sections~\ref{sec:application-bilevel} and
\ref{sec:case-study-gas} explicitly show the
applicability of our method to bilevel problems with
nonconvex \rev{and quadratic} lower-level problems
(Section~\ref{sec:application-bilevel})
and to problems on gas transport networks that are subject to
differential equations (Section~\ref{sec:case-study-gas}).
Both are classes of problems that received a lot of attention
during the last years; see, \eg,
\cite{Beck_et_al:2022,Kleinert_et_al:2021c,Dempe:2020} and
\cite{Koch_et_al:2015,Pfetsch_et_al:2015}.
Before we are able to tackle these problems, we first need to formally
state the problem class under consideration, which is what we do in
Section~\ref{sec:problem-definition}.
Afterward, in Section~\ref{sec:algorithm}, we describe the main
rationale of the method, present it formally, and analyze it
theoretically.
The latter leads to a correctness theorem showing that the method
finitely terminates \rev{with an approximate solution} and we further
derive a worst-case iteration bound.

\rev{The main contributions of our work are the following.
  We develop an algorithm that only requires very weak assumptions for
  being applied.
  Hence, the method can be applied to a very large set of problems
  that cannot be solved with other classic MINLP solvers, which require
  that all constraints of the problem are given in closed form.
  To illustrate the generality of the method, we present the
  application to nonlinear gas network optimization as well as bilevel
  problems with nonconvex quadratic lower-level problems.
  Our work clearly needs to be seen as a generalization of the works
  \cite{Schmidt_et_al:2019,Schmidt_et_al:2021b}.
  In particular, \rev{w}e generalize \cite{Schmidt_et_al:2019} to the
  multidimensional case for which we present a more effective numerical
  scheme compared to \cite{Schmidt_et_al:2021b} that uses different
  geometries for the outer approximation as compared to those used in
  \cite{Schmidt_et_al:2019}.}
Since our main workhorse is the Lipschitz continuity of the
nonlinearities, we are still in line with the works
\cite{%Evtushenko1971,%
  %Piyavskii1972,%
  Pinter1996,%
  %Pinter1986a,%
  Pinter1986c,%
  HorstTuy1988,%
  Horst1987,%
  Horst1988,%
  Pinter1988,%
  Tuy1988},
to name only a few.
For a more detailed overview about this field see the textbook
\cite{Horst1996} and the references therein as well as
\cite{Schmidt_et_al:2019,Schmidt_et_al:2021b}, where we discussed the
positioning of the method in the literature in more detail.

%%% Local Variables:
%%% mode: latex
%%% TeX-master: "multidim-lipschitz-minlp-preprint"
%%% End:

%% file: problem-definition.tex
\section{Problem Definition}
\label{sec:problem-definition}

We consider the problem
\begin{subequations}
  \label{eq:mult-problem}
  \begin{align}
  \label{eq:mult-problem-obj}
  \min_{x} \quad & c^\top x\\
  \label{eq:mult-problem-lin-and-vars}
  \st \quad
  & Ax \geq b, \quad \lb{x} \leq x \leq \ub{x}, \quad x \in \R^{n}
  \times \Z^{m},\\
  \label{eq:mult-problem-nonlin}
  & f_i(x_{\multind}) = x_{\rind}, \quad i \in [p],
  \end{align}
\end{subequations}
where $c \in \R^{n+m}$, $A \in \R^{q \times (n+m)}$, and $b \in
\R^q$ are given data, $\lb{x} \in \R^n \times \Z^m$ and $\ub{x} \in
\R^n \times \Z^m$ are finite bounds, and $[p] \define \set{1, \dotsc,
  p}$.
Hence, we consider a linear objective~\eqref{eq:mult-problem-obj},
linear mixed-integer constraints~\eqref{eq:mult-problem-lin-and-vars},
and nonlinear constraints defined by the functions $f_i: \R^{l_i} \to \R$.
All $f_i$, $i \in [p]$, are Lipschitz continuous functions and $l_i =
\card{\multind}$ is the number of their arguments.
Moreover, $\multind \subset [n]$ is the index set of
the variables on which the function~$f_i$ depends.
Without loss of generality, we further assume that $\rind \in [n]$
with $\rind \notin \multind$ for all $i \in [p]$.
In what follows, we also write $x_{\allindices} = (x_{\multind},
x_{\rind}) \in \R^{l_i + 1}$.\footnote{Note that we omit
  transpositions here and in what follows for the ease of better
  reading.}

The main challenge when solving Problem~\eqref{eq:mult-problem} is
that we assume that the nonlinear functions~$f_i$ are not given in
closed form but that we can only evaluate them and that we know their
Lipschitz constants.

The $\myeps$-relaxed version of the original problem
\eqref{eq:mult-problem}
is given by
\begin{subequations}
  \label{eq:mult-problem-rela}
  \begin{align}
  \label{eq:mult-problem-rela-obj}
  \min_{x} \quad & c^\top x\\
  \label{eq:mult-problem-rela-lin-and-vars}
  \st \quad
  & Ax \geq b, \quad \lb{x} \leq x \leq \ub{x}, \quad x \in \R^{n}
  \times \Z^{m},\\
  \label{eq:mult-problem-rela-nonlin}
  & \abs{f_i(x_{\multind}) - x_{\rind}} \leq \varepsilon, \quad i \in [p],
  \end{align}
\end{subequations}
where $\varepsilon > 0$ is a prescribed tolerance.
A feasible point of \eqref{eq:mult-problem-rela} is called an
$\varepsilon$-feasible point of Problem \eqref{eq:mult-problem}.
\rev{Moreover, we call a global solution of \eqref{eq:mult-problem-rela} an
  approximate global optimal solution in what follows.}

%%% Local Variables:
%%% mode: latex
%%% TeX-master: "multidim-lipschitz-minlp-preprint"
%%% End:

%% file: algorithm.tex
\section{The Algorithm}
\label{sec:algorithm}

In this section, we introduce an iterative procedure to solve Problem
\eqref{eq:mult-problem} to approximate global optimality.
The main idea is to relax all nonlinearities by utilizing the Lipschitz
continuity of these functions.
In each iteration, the relaxed problem, which we will call the master
problem, needs to be solved to global optimality.
Subsequently, a subproblem is solved to tighten the
relaxation for the next master problem.
This procedure is then repeated until an $\myeps$-feasible
solution is found or until it can be shown that the original problem
is infeasible.

The master problem  in iteration $k$ reads
\begin{equation*}
  \label{eq:mult-problem-master}
  \tag{M($k$)}
  \begin{aligned}
    \min_{x} \quad & c^\top x\\
    \st \quad
    & Ax \geq b, \quad \lb{x} \leq x \leq \ub{x}, \quad x \in \R^{n}
    \times \Z^{m},\\
    & x_{\allindices} \in \Omega_i^k, \quad i \in [p],
  \end{aligned}
\end{equation*}
where $\Omega_i^k$ is a relaxation of the graph of the nonlinearity
$f_i$.
This relaxation will be stated in terms of mixed-integer linear
constraints; see below.
The idea behind is to partition the domain of $f_i$ into a set of boxes
that are indexed using indices in $J_i^k = \set{1,\dots, \abs{J_i^k}}$
and to linearly relax the graph over each box with the \rev{region}
$\Omega_i^k(j)$, $j \in J_i^k$, using the Lipschitz continuity
of~$f_i$.
Hence, we obtain
\begin{equation}
  \label{eq:omega-def-as-union}
  \Omega_i^k = \bigcup_{j \in J_i^k} \Omega_i^k(j).
\end{equation}
After solving the master problem, the boxes that contain the
solution~$x^k$ are identified and split in smaller boxes to get a
finer relaxation for the next iteration.
The main purpose of solving the subproblem afterward is to find
good splitting points for these boxes.
To this end, the subproblem determines a point of the graph of
each nonlinearity and, at the same time, tries to minimize
the distance to the solution of the last master problem.
Hence, the subproblem is given by
\begin{equation}
  \label{eq:mult-problem-sub}
  \tag{S($k$)}
  \min_{\tilde{x}} \quad \norm{\tilde{x} - x^k}_2^2
  \quad \st \quad f_i(\tilde{x}_{\multind}) = \tilde{x}_{\rind},
  \quad \tilde{x}_{\allind} \in \tilde{\Omega}_i^k(j_i^k), \quad i \in [p],
\end{equation}
where $j_i^k \in J_i^k$ denotes the box with $x_{\allindices}^k \in
\Omega_i^k(j_i^k)$
for all $i \in [p]$ and $\tilde{\Omega}_i^k(j_i^k)$ is a suitably
chosen sub\rev{region} of $\Omega_i^k(j_i^k)$.
The reason for the use of these sub\rev{regions} will be discussed in detail in
Section~\ref{sec:constr-subprobl}.

Figure~\ref{fig:master-and-subproblem} depicts the
subproblem~\eqref{eq:mult-problem-sub} with the \rev{regions}
$\tilde{\Omega}_i^k(j_i^k)$ and $\Omega_i^k(j_i^k)$ (left) and the
corresponding \rev{regions} $\Omega_i^{k+1}(j)$ and $\Omega_i^{k+1}(j+1)$ of
the master problem~\eqref{eq:mult-problem-master} in the next
iteration (right).

\begin{figure}
  \centering
  \begin{minipage}{.49\textwidth}
    \input{tikz-imgs/subproblem.tikz}
  \end{minipage}
  \begin{minipage}{.49\textwidth}
    \input{tikz-imgs/master-problem.tikz}
  \end{minipage}
  \caption{Visualization of the subproblem~\eqref{eq:mult-problem-sub}
    (left) and of the feasible \rev{region} of the master
    problem~\eqref{eq:mult-problem-master} in the next iteration
    (right) for a nonlinear function $f_i:\R \to \R$.}
  \label{fig:master-and-subproblem}
\end{figure}
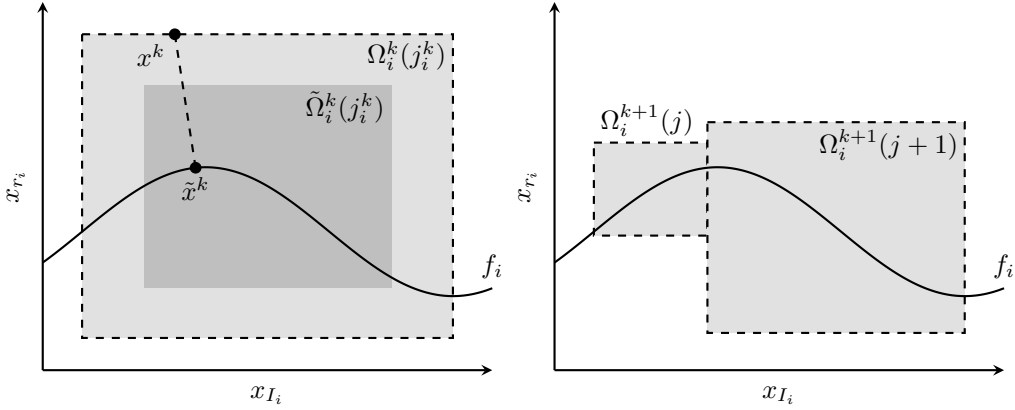

%%%%%%%%%%%%%%%%%%%%%%%%%%%%%%%%%%%%%%%%%%%%%%%%%%%%%%%%%%%%%%%%%%%%%%%%%%%
\subsection{Construction of the Master Problem's Feasible \rev{Region}}
\label{sec:constr-mast-probl}

We now describe in detail how we construct the
relaxations~$\Omega_i^k$.
First, we define the box
\begin{equation*}
  B(\lb{v}, \ub{v}) \define \defset{x \in \R^{d}}{\lb{v} \leq
  x \leq \ub{v}}
\end{equation*}
for $\lb{v}, \ub{v} \in \R^{d}$, $\lb{v} \leq \ub{v}$, and arbitrary
dimension~$d$.

For each $i \in [p]$, we assume that we are given vectors
$\lb{v}_i^k(j), \ub{v}_i^k(j) \in \R^{l_i}$ for $j \in J_i^k$ such
that the boxes $B(\lb{v}_i^k(j), \ub{v}_i^k(j))$ have pairwise disjoint
interiors and cover the bounding box of $x_{\multind}$, \ie, we have
\begin{equation}
  \label{eq:box-cover-bounds}
  B\left(\lb{x}_{\multind}, \ub{x}_{\multind}\right)
  = \bigcup_{j \in J_i^k}
  B\left(\lb{v}_i^k(j), \ub{v}_i^k(j)\right).
\end{equation}

Let $L_i$ be the Lipschitz constant of $f_i$ on $B(\lb{x}_{\multind},
\ub{x}_{\multind}) \subset \R^{l_i}$ \wrt\ a given
norm~$\norm{\cdot}$, where any (weighted) norm in $\R^{l_i}$ can
be used.
Let $m_i^k(j)$ be the center point of the box $B(\lb{v}_i^k(j),
\ub{v}_i^k(j))$, \ie,
\begin{equation*}
  m_i^k(j) = \frac{1}{2} \left(\lb{v}_i^k(j) + \ub{v}_i^k(j) \right)
\end{equation*}
holds.
Due to the Lipschitz continuity of $f_i$, we have
\begin{align*}
  x_{\rind}
  & \leq f_i(m_i^k(j)) + L_i \norm{x_{\multind} - m_i^k(j)},
  \\
  x_{\rind}
  & \geq f_i(m_i^k(j)) - L_i \norm{x_{\multind} - m_i^k(j)}
\end{align*}
for $x_{\multind} \in B(\lb{v}_i^k(j), \ub{v}_i^k(j))$.
Since $\norm{x_{\multind} - m_i^k(j)}$ attains its maximum over
$B(\lb{v}_i^k(j), \ub{v}_i^k(j))$ in the vertices of the box,
we can replace $x_{\multind}$ with $\ub{v}_i^k(j)$.
It thus holds
\begin{subequations}
  \label{eq:box-bounds-rhs}
  \begin{align}
    x_{\rind}
    & \leq f_i(m_i^k(j)) + L_i \norm{\ub{v}_i^k(j) - m_i^k(j)}
      = f_i(m_i^k(j)) + \frac{L_i}{2} \norm{\ub{v}_i^k(j) -
      \lb{v}_i^k(j)},
    \\
    x_{\rind}
    & \geq f_i(m_i^k(j)) - L_i \norm{\ub{v}_i^k(j) - m_i^k(j)}
      = f_i(m_i^k(j)) - \frac{L_i}{2} \norm{\ub{v}_i^k(j) -
      \lb{v}_i^k(j)}
  \end{align}
\end{subequations}
for $x_{\multind} \in B(\lb{v}_i^k(j), \ub{v}_i^k(j))$.
With this we can define the \rev{region}
$\Omega_i^k(j)$ for $j \in J_i^k$ as the box
\begin{equation*}
  \begin{aligned}
    \Omega_i^k(j) \define \Big\{(x_{\multind}, x_{\rind}) \in \R^{l_i + 1}
    \defsep
    & \lb{v}_i^k(j) \leq x_{\multind} \leq \ub{v}_i^k(j),\\
    & x_{\rind} \leq f_i(m_i^k(j))
    + \frac{1}{2}L_i \norm{\ub{v}_i^k(j) - \lb{v}_i^k(j)},\\
    & x_{\rind} \geq f_i(m_i^k(j))
    - \frac{1}{2}L_i \norm{\ub{v}_i^k(j) - \lb{v}_i^k(j)}\Big\}.
  \end{aligned}
\end{equation*}

The next step is to define the \rev{region} $\Omega_i^k$ as the union of all
$\Omega_i^k(j)$; see \eqref{eq:omega-def-as-union}.
Using \eqref{eq:box-cover-bounds}, we obtain a covering of the
bounding box of $x_{\multind}$ and via~\eqref{eq:box-bounds-rhs}, we
obtain bounds for~$x_{\rind}$.
Hence, it follows that $\Omega_i^k$ is a relaxation of the graph of
the function~$f_i$.
\begin{proposition}
  \label{thm:graph_in_omega}
  It holds $\graph(f_i) \subseteq \Omega_i^k$ on $B(\lb{x}_{\multind},
  \ub{x}_{\multind})$ for all $i \in [p]$ and all~$k$.
\end{proposition}

For what follows, we abbreviate
\begin{equation*}
\boxset_i^k \define \Set{B(\lb{v}_i^k(1), \ub{v}_i^k(1)), \dots,
  B(\lb{v}_i^k(\abs{J_i^k}), \ub{v}_i^k(\abs{J_i^k}))},
\end{equation*}
\ie, $\boxset_i^k$ is the set of boxes that is used to define
$\Omega_i^k$.
In our algorithm, we use $\Omega_i^k$ to replace the nonlinear
constraints~$f_i(x_{\multind}) = x_{\rind}$ for all $i \in [p]$ in
Problem~\eqref{eq:mult-problem} in order to obtain a relaxation.

\begin{lemma}
  The master problem \eqref{eq:mult-problem-master} can be modeled as
  mixed-integer linear problem.
\end{lemma}
\begin{proof}
  We can write \eqref{eq:mult-problem-master} using the following
  Big-$M$ formulation:
  \begin{subequations}
    \label{eq:mult-problem-master-mip}
    \begin{align}
      \label{eq:mult-problem-master-mip-obj}
      \min_{x, z} \quad & c^\top x\\
      \label{eq:mult-problem-master-mip-lin-and-vars}
      \st \quad
      & Ax \geq b, \quad \lb{x} \leq x \leq \ub{x}, \quad x \in \R^{n}
      \times \Z^{m},\\
      \label{eq:mult-problem-master-mip-polytope-bigm-lower}
      & x_{\multind} \geq \lb{v}_{\multind}^k(j) - M (1 - z_i^k(j)),
      & i \in [p], \ j \in J_i^k,\\
      \label{eq:mult-problem-master-mip-polytope-bigm-upper}
      & x_{\multind} \leq \ub{v}_{\multind}^k(j) + M (1 - z_i^k(j)),
      & i \in [p], \ j \in J_i^k,\\
      \label{eq:mult-problem-master-mip-lipschitz-leq}
      & x_{\rind} \leq f_i(m_i^k(j))
      + \frac{1}{2}L_i \norm{\ub{v}_i^k(j) - \lb{v}_i^k(j)}
      + M (1 - z_i^k(j)),
      & i \in [p], \ j \in J_i^k,\\
      \label{eq:mult-problem-master-mip-lipschitz-geq}
      & x_{\rind} \geq f_i(m_i^k(j))
      - \frac{1}{2}L_i \norm{\ub{v}_i^k(j) - \lb{v}_i^k(j)}
      - M (1 - z_i^k(j)),
      & i \in [p], \ j \in J_i^k,\\
      \label{eq:mult-problem-master-mip-bin-var-sum}
      & \sum_{j \in J_i^k} z_i^k(j) = 1, & i \in [p],\\
      \label{eq:mult-problem-master-mip-bin-var}
      & z_i^k(j) \in \set{0, 1}, & i \in [p], \ j \in J_i^k.
    \end{align}
  \end{subequations}
  The rationale of this model is as follows.
  For each nonlinearity~$i \in [p]$ and each box $j \in J_i^k$, we
  introduce a binary variable~$z_i^k(j)$ that indicates whether the
  solution lies in this box or not.
  If $z_i^k(j) = 1$, the
  constraints~\eqref{eq:mult-problem-master-mip-polytope-bigm-lower}--%
  \eqref{eq:mult-problem-master-mip-lipschitz-geq}
  are equivalent to the definition of $\Omega_i^k(j)$.
  If $z_i^k(j) = 0$, then
  \eqref{eq:mult-problem-master-mip-polytope-bigm-lower}--%
  \eqref{eq:mult-problem-master-mip-lipschitz-geq}
  are always fulfilled if the constant $M$ is chosen sufficiently
  large.
  constraint~\eqref{eq:mult-problem-master-mip-bin-var-sum} finally
  ensures that for each nonlinearity~$i \in [p]$ exactly one box~$j
  \in J_i^k$ is chosen.
\end{proof}

Let us remark that it is always possible in our setting to obtain
finite and sufficiently large values~$M$ by using the finite bounds on
the variables in~\eqref{eq:mult-problem-lin-and-vars}.

%%%%%%%%%%%%%%%%%%%%%%%%%%%%%%%%%%%%%%%%%%%%%%%%%%%%%%%%%%%%%%%%%%%%%%%%%%%
\subsection{Construction of the Subproblem}
\label{sec:constr-subprobl}

We now introduce the chosen sub\rev{region} of $\Omega_i^k(j)$ used in the
subproblem
\eqref{eq:mult-problem-sub}.
We define this \rev{region} as
\begin{equation*}
  \tilde{\Omega}_i^k(j) = \Omega_i^k(j) \cap \hat{\Omega}_i^k(j),
\end{equation*}
using the further sub\rev{region}
\begin{equation*}
  \hat{\Omega}_i^k(j) = \defset{(x_{\multind}, x_{\rind}) \in \R^{l_i + 1}}
  {(1 - \lambda) \lb{v}_i^k(j) + \lambda \ub{v}_i^k(j) \leq x_{\multind}
  \leq \lambda \lb{v}_i^k(j) + (1 - \lambda) \ub{v}_i^k(j)},
\end{equation*}
for some $\lambda \in (0, 1/2]$.
This ensures that the solution of the subproblem \eqref{eq:mult-problem-sub}
cannot be arbitrarily close to the edges of the used box if $\lambda >
0$ is chosen appropriately.

Let us also note that for each iteration~$k \in \N$,
the subproblem~\eqref{eq:mult-problem-sub} can be separated into $p$
smaller problems under reasonable assumptions.
If the index sets $\allindices = (\multind, \rind)$ are
non-overlapping, \ie,
\begin{equation}
  \label{eq:non-overlapping-indices}
  \left(\multind \cup \set{\rind}\right)
  \cap \left(I_j \cup \set{r_j}\right) = \emptyset
  \quad\text{for all} \quad i,j \in [p],\ i \neq j,
\end{equation}
these smaller problems can be solved in parallel.
The above assumption can always be satisfied by introducing additional
auxiliary variables.
\begin{lemma}
  Suppose that the index sets $\allind = (\multind, \rind)$ are
  non-overlapping, \ie, \eqref{eq:non-overlapping-indices} holds.
  Then, the subproblem~\eqref{eq:mult-problem-sub} is completely
  separable, \ie, we can solve the subproblem in iteration~$k$ by
  solving $p$ smaller problems given by
  \begin{equation}
    \label{eq:mult-problem-sub_sepa}
    \min_{\tilde{x}_{\allindices}} \quad
    \norm{\tilde{x}_{\allindices} - x_{\allindices}^k}_2^2
    \quad \st \quad f_i(\tilde{x}_{\multind}) = \tilde{x}_{\rind},
    \quad \tilde{x}_{\allindices} \in \tilde{\Omega}_i^k(j_i^k).
  \end{equation}
\end{lemma}
\begin{proof}
  The constraints of \eqref{eq:mult-problem-sub}, \ie,
  \begin{equation*}
    f_i(\tilde{x}_{\multind}) = \tilde{x}_{\rind},
    \quad \tilde{x}_{\allindices} \in \tilde{\Omega}_i^k(j_i^k),
    \quad i \in [p],
  \end{equation*}
  completely decouple along $i \in [p]$ and so does the objective
  function
  \begin{equation*}
    \norm{\tilde{x} - x^k}_2^2
    = \sum_{i \in [p]} \norm{\tilde{x}_{\allindices} -
      x_{\allindices}^k}_2^2.
  \end{equation*}
  Therefore, the solution of the
  subproblem~\eqref{eq:mult-problem-sub} can be obtained by
  solving Problem~\eqref{eq:mult-problem-sub_sepa} for all $i \in
  [p]$.
\end{proof}

\rev{Note that, in the formal sense of complexity theory, the
  subproblem can be as hard to solve as the originally given MINLP.
  However, the split into multiple and thus much smaller subproblems
  can make a huge difference in practice.
  Moreover, we completely split the mixed-integer aspects from the
  nonlinear aspects of the problem, which can also help a lot in
  solving the subproblems although they are still hard in the formal
  sense.}

%%%%%%%%%%%%%%%%%%%%%%%%%%%%%%%%%%%%%%%%%%%%%%%%%%%%%%%%%%%%%%%%%%%%%%%%%%%
\subsection{Formal Statement of the Algorithm}
\label{sec:formal-algorithm-statement}

Before we can formally introduce the algorithm,
we need the following notation.
Let $B(\lb{v}, \ub{v}) \subseteq \R^{d}$ be a box with an interior
point $x \in \interior B(\lb{v}, \ub{v})$, \ie, $\lb{v} < x <
\ub{v}$.
The point $x$ splits the box into a set of boxes that we define as
\begin{equation*}
  \splitbox(\lb{v}, \ub{v}, x) \define \defset{B(\lb{w}, \ub{w})}
  {(\lb{w}_\ell = \lb{v}_\ell \land \ub{w}_\ell = x_\ell)
    \lor (\lb{w}_\ell = x_\ell \land \ub{w}_\ell = \ub{v}_\ell)
    \text{ for all } \ell \in [d]}.
\end{equation*}

We can utilize this notation to get a finer relaxation of
$\graph(f)$ by splitting an element of $\boxset_i^k$ using the
solution of the subproblem~\eqref{eq:mult-problem-sub} as the
splitting point.
This yields a set of smaller boxes that still fulfills
Condition~\eqref{eq:box-cover-bounds} as the following proposition
states.
\begin{proposition}
  For a given box~$B(\lb{v}, \ub{v}) \subset \R^{l_i}$ and an interior
  point~$x \in \interior B(\lb{v}, \ub{v})$, the
  set~$\splitbox(\lb{v}, \ub{v}, x)$ contains $2^{l_i}$
  smaller boxes that have pairwise disjoint interiors and that
  completely cover the box $B(\lb{v}, \ub{v})$, \ie,
  \begin{equation*}
    B(\lb{v}, \ub{v})
    = \bigcup_{b \in \splitbox(\lb{v}, \ub{v}, x)} b.
  \end{equation*}
\end{proposition}

\begin{algorithm}
  \caption{Iterative Method to Approximately Solve
    Problem~\eqref{eq:mult-problem}}
  \label{alg:mult-method}
  \input{alg-listing}
\end{algorithm}

We can now present the complete method, which is
given in Algorithm~\ref{alg:mult-method}.
Before we prove its correctness, let us first discuss its basic
functionality.
After the master problem~\eqref{eq:mult-problem-master} is solved in
Step~\ref{alg:mult-method-solve-master}, it is checked in
Step~\ref{alg:mult-method-check-return-solution} if its solution is
already $\myeps$-feasible for the original problem.
To determine the boxes $j_i^k \in J_i^k$ in
Step~\ref{alg:mult-method-find-boxes} one can simply check
the indicator variables $z_i^k(j)$ of the MIP formulation
\eqref{eq:mult-problem-master-mip}.
If the solution is not yet $\myeps$-feasible, then there are
nonlinearities $f_i$ that have feasibility violations larger
than~$\myeps$.
For these nonlinearities we refine the relaxation of the master
problem in Step~\ref{alg:mult-method-split-box} and re-iterate.

Note that it is not necessary for the correctness of
Algorithm~\ref{alg:mult-method} to solve the
subproblem~\eqref{eq:mult-problem-sub} to global optimality.
Our rationale, however, is that optimal solutions of the subproblems
yield better splitting points that lead to faster convergence in
practice.
For the correctness of the algorithm it is sufficient to find feasible
points of \eqref{eq:mult-problem-sub} that are guaranteed to exist
due to the following lemma.

\begin{lemma}
  All subproblems~\eqref{eq:mult-problem-sub} are feasible if
  Property~\eqref{eq:non-overlapping-indices} is satisfied.
\end{lemma}
\begin{proof}
  \rev{Because of Property~\eqref{eq:non-overlapping-indices}, the nonlinear constraints $f_i(\tilde{x}_{\multind}) = \tilde{x}_{\rind}$ in~\eqref{eq:mult-problem-sub} do not depend on each other.
  From Proposition~\ref{thm:graph_in_omega} we know that the graph of $f_i$ over the initial region~$B(\lb{x}_{\multind}, \ub{x}_{\multind})$ lies entirely in $\Omega_i^k$.
  Since the nonempty sub\rev{region} $\tilde{\Omega}_i^k(j)$ restricts $\Omega_i^k$ in all but the last dimension, the claim follows.}
\end{proof}

Next, we prove that Algorithm~\ref{alg:mult-method} always terminates
after finitely many iterations.
\begin{theorem}
  \label{thm:mult-method-convergence}
  There exists a constant $K < \infty$ such that Algorithm
  \ref{alg:mult-method} either
  terminates with an approximate global optimal \rev{solution}
  $x^{k^*}$ or with the indication of infeasibility in an iteration
  $k^* \leq K$.
\end{theorem}
\begin{proof}
  The box $\Omega_i^k(j)$ is bounded for each iteration $k$ and all $i
  \in [p]$ and $j \in J_i^k$.
  For the $x_{\rind}$-coordinate, the bounding inequalities are given by
  \begin{equation*}
    f_i(m_i^k(j)) - \frac{1}{2}L_i \norm{\ub{v}_i^k(j) - \lb{v}_i^k(j)}
    \leq
    x_{\rind}
    \leq
    f_i(m_i^k(j)) + \frac{1}{2}L_i \norm{\ub{v}_i^k(j) - \lb{v}_i^k(j)}.
  \end{equation*}
  Therefore the corresponding side length of the box~$\Omega_i^k(j)$
  in its $x_{\rind}$-coordinate is
  \begin{equation*}
    d_{\rind}^k(j)
    \define
    L_i \norm{\ub{v}_i^k(j) - \lb{v}_i^k(j)}.
  \end{equation*}
  If $d_{\rind}^k(j) \leq \myeps$ holds, then the inequality
  \begin{equation*}
    \abs{f_i(x_{\multind}^k) - x_{\rind}^k} > \myeps
  \end{equation*}
  in Step~\ref{alg:mult-method-check-split-box} of
  Algorithm~\ref{alg:mult-method} cannot be fulfilled.
  It follows that the box $B(\lb{v}_i^k(j_i^k), \ub{v}_i^k(j_i^k))$
  will not be split but remains the same in
  Step~\ref{alg:mult-method-not-split-box}.

  Next, we analyze how $d_{\rind}^k(j)$ changes if a box is split into
  smaller boxes.
  Let
  \begin{equation*}
    B\left(\lb{v}_i^{k+1}(j), \ub{v}_i^{k+1}(j)\right) \in
    \splitbox\left(\lb{v}_i^k(j_i^k), \ub{v}_i^k(j_i^k), \tilde{x}_i^k\right)
  \end{equation*}
  be one of the smaller boxes that is added to $\boxset_i^k$ in Step
  \ref{alg:mult-method-split-box}.
  The side length~$d_{\rind}^{k+1}(j)$ of the corresponding
  box~$\Omega_i^{k+1}(j)$ can be bounded from above via
  \begin{equation*}
    d_{\rind}^{k+1}(j)
    = L_i \norm{\ub{v}_i^{k+1}(j) - \lb{v}_i^{k+1}(j)}
    \leq (1 - \lambda) L_i \norm{\ub{v}_i^k(j_i^k) - \lb{v}_i^k(j_i^k)}
    = (1 - \lambda) d_{\rind}^k(j_i^k).
  \end{equation*}
  This means that $d_{\rind}^k(j)$ is decreased by at least a factor of
  $(1 - \lambda)$ if a box is split in
  Step~\ref{alg:mult-method-split-box}.
  Since $\left((1 - \lambda)^k\right)_{k \in \N}$ is a geometric
  sequence with $\abs{1 - \lambda} < 1$, it converges to zero, \ie,
  $(1 - \lambda)^k \to 0 \text{ as } k \to \infty$.
  It follows that any box $B(\lb{v}_i^k(j), \ub{v}_i^k(j))$---including the
  first box $B(\lb{x}_{\multind}, \ub{x}_{\multind})$---can only
  be split finitely many times in Step~\ref{alg:mult-method-split-box} before
  the side length $d_{\rind}^k(j)$ of $\Omega_i^k(j)$ fulfills $d_{\rind}^k(j)
  \leq \myeps$.

  Since the index set~$[p]$ is finite, there are only finitely many
  boxes in $\boxset_i^k$ for all $i \in [p]$ and all $k$.
  These boxes can only be split finitely many times.
  Hence, there exists an iteration $K < \infty$ in which no box will be
  split in Step~\ref{alg:mult-method-split-box}.
  This, however, can only be the case if the if-condition in
  Step~\ref{alg:mult-method-check-split-box} does not hold for all $i
  \in [p]$.
  Thus, we have
  \begin{equation*}
    \abs{f_i(x_{\multind}^K) - x_{\rind}^K} \leq \myeps
    \quad \text{for all } i \in [p],
  \end{equation*}
  which is the if-condition in
  Step~\ref{alg:mult-method-check-return-solution}.
  This means that the algorithm terminates in
  Step~\ref{alg:mult-method-return-solution} in an iteration~$K$.
  Hence, it follows that there exists a $K < \infty$ such that
  Algorithm~\ref{alg:mult-method} terminates in
  Step~\ref{alg:mult-method-return-infeasible} or
  \ref{alg:mult-method-return-solution} in an iteration $k^* \leq
  K$.
\end{proof}

We close this section by stating and proving a result for
the worst-case number of required iterations of
Algorithm~\ref{alg:mult-method}.
\begin{theorem}
  \label{thm:iterations-worst-case}
  Algorithm \ref{alg:mult-method} terminates after at most
  \begin{equation*}
    K = \sum_{i \in [p]} \sum_{k = 0}^{S_i} 2^{k l_i}
  \end{equation*}
  iterations with
  \begin{equation*}
    S_i = \left\lceil \log_{(1 - \lambda)} \left(\frac{\myeps}{L_i
    \norm{\ub{x}_{\multind} - \lb{x}_{\multind}}}\right)\right\rceil
  \end{equation*}
  for $i \in [p]$.
\end{theorem}
\begin{proof}
  From the proof of Theorem~\ref{thm:mult-method-convergence} we know
  that a box~$B(\lb{v}_i^k(j_i^k), \ub{v}_i^k(j_i^k))$ can only be
  split finitely many times before the side length $d_{\rind}^k(j)$ of
  $\Omega_i^k(j)$ satisfies $d_{\rind}^k(j) \leq \myeps$.
  We can give an upper bound for how many iterations this takes for
  the first box $B(\lb{x}_{\multind}, \ub{x}_{\multind})$ by solving
  the equation
  \begin{equation*}
    (1 - \lambda)^k L_i \norm{\ub{x}_{\multind} -
      \lb{x}_{\multind}}
    = \myeps
  \end{equation*}
  for $k$, which yields
  \begin{equation*}
    k = \log_{(1 - \lambda)} \left(\frac{\myeps}{L_i
        \norm{\ub{x}_{\multind} - \lb{x}_{\multind}}}\right).
  \end{equation*}
  Since the box $B(\lb{v}_i^k(j_i^k), \ub{v}_i^k(j_i^k))$ will not be split
  anymore if $d_{\rind}^k(j) \leq \myeps$, we can round this value to
  obtain
  \begin{equation*}
    S_i = \left\lceil \log_{(1 - \lambda)} \left(\frac{\myeps}{L_i
          \norm{\ub{x}_{\multind} - \lb{x}_{\multind}}} \right)
    \right\rceil.
  \end{equation*}

  For each box that is split, there are $2^{l_i}$ smaller boxes that
  are added to $\boxset_i^k$.
  Therefore, for each $i \in [p]$ the maximal number of iterations
  required until there are no boxes left in~$\boxset_i^k$ that can be
  split is bounded from above by
  \begin{equation}
    \label{eq:mult-method-iteration-limit-i}
    \sum_{k = 0}^{S_i} \left( 2^{l_i} \right)^k
    = \sum_{k = 0}^{S_i} 2^{k \, l_i}.
  \end{equation}
  Since it is possible that in each iteration, there is only a single
  nonlinearity~$i \in [p]$ for which a box is split, we have to sum up
  \eqref{eq:mult-method-iteration-limit-i} for each $i \in [p]$ to get
  \begin{equation*}
    K = \sum_{i \in [p]} \sum_{k = 0}^{S_i} 2^{k l_i}
  \end{equation*}
  as an upper bound for the required number of iterations of
  Algorithm~\ref{alg:mult-method}.
\end{proof}

\begin{remark}
  Theorem~\ref{thm:iterations-worst-case} states that choosing
  $\lambda = 0.5$ results in the lowest number of iterations in the
  worst-case. Then, no subproblem~\eqref{eq:mult-problem-sub} is
  needed as one can simply evaluate the nonlinearity~$f_i$ in the
  center point~$m_i^k(j_i^k)$ of the current box to receive the
  splitting point.
  However, in practice it can be better to choose a smaller
  parameter~$\lambda$, which allows the splitting point to be closer to
  the master problem's solution and which, thus, may result in a finer
  approximation of the nonlinearity near the optimal solution of
  Problem~\eqref{eq:mult-problem}.
\end{remark}

%%% Local Variables:
%%% mode: latex
%%% TeX-master: "multidim-lipschitz-minlp-preprint"
%%% End:

%% file: tikz-imgs/subproblem.tikz
\begin{tikzpicture}
  \def \thickness {thick}
  \def \ylabelshift {25}
  \def \lipschitzconstant {1}
  \def \xoffset {0.5}
  \def \xmax {1.5*pi}
  \def \xmid {\xmax/2}
  \def \ymid {\myfunc{\xmid}}
  \def \ymin {\ymid - \lipschitzconstant * (\xmid + \xoffset)}
  \def \ymax {\ymid + \lipschitzconstant * (\xmid + \xoffset)}
  \def \xmasterone {0.375*pi}
  \def \ymasterone {\ymid + 1 * \lipschitzconstant * \xmid}
  \def \xsplitone {1.44271}
  \def \xmidleftone {\xsplitone/2}
  \def \xmidrightone {(\xsplitone + \xmax)/2}
  \begin{axis}[
    axis line style = \thickness,
    axis y line=left,
    axis x line=bottom,
    xtick=\empty,
    ytick=\empty,
    x label style={yshift=12.5},
    y label style={yshift=-\ylabelshift},
    ymax=\ymax,
    ymin=\ymin,
    xlabel={$x_{\multind}$},
    ylabel={$x_{\rind}$},
    width=\textwidth+\ylabelshift,
    % height=\tikzimageheight cm,
    axis on top,
    clip=false
    ]
    % \Omega_i^k(j_i^k)
    \filllipschitzbox{\xmid}{\xmax}{black!12}{dashed}
    % \tilde{\Omega}_i^k(j_i^k)
    \filllipschitzbox{\xmid}{2 * \xmax / 3}{black!25}{draw=none}
    % Box labels
    \addplot[black] coordinates {({\xmax+\xoffset}, {\ymid + \lipschitzconstant * (\xmax / 2)})} node[below left, xshift=1, yshift=1] {$\Omega_i^k(j_i^k)$};
    \addplot[black] coordinates {({\xmid+(\xmax / 3)+\xoffset}, {\ymid + \lipschitzconstant * (\xmax / 3)})} node[below left, xshift=1, yshift=1] {$\tilde{\Omega}_i^k(j_i^k)$};
    % The function
    \addplot[smooth,black,domain=0:\xmax+2*\xoffset,samples=100, \thickness]{\myfunc{x-\xoffset}} node[above,pos=1] {$f_i$};
    % Master problem solution
    \addplot[black, mark=*] coordinates {({(\xmasterone)+\xoffset}, {\ymasterone})} node[below left, pos=1] {$x^k$};
    % Split point
    \addplot[black, mark=*] coordinates {({(\xsplitone)+\xoffset}, {\myfunc{\xsplitone}})} node[below,pos=1] {$\tilde{x}^k$};
    % Distance between master problem solution and split point
    \addplot[black, dashed, \thickness] coordinates {({\xmasterone+\xoffset}, {\ymasterone}) ({\xsplitone+\xoffset}, {\myfunc{\xsplitone}})};
  \end{axis}
\end{tikzpicture}%
%
%%% Local Variables:
%%% mode: latex
%%% TeX-master: "../multidim-lipschitz-minlp-preprint"
%%% End:

%% file: tikz-imgs/master-problem.tikz
\begin{tikzpicture}
  \def \thickness {thick}
  \def \ylabelshift {25}
  \def \lipschitzconstant {1}
  \def \xoffset {0.5}
  \def \xmax {1.5*pi}
  \def \xmid {\xmax/2}
  \def \ymid {\myfunc{\xmid}}
  \def \ymin {\ymid - \lipschitzconstant * (\xmid + \xoffset)}
  \def \ymax {\ymid + \lipschitzconstant * (\xmid + \xoffset)}
  \def \xmasterone {0.375*pi}
  \def \ymasterone {\ymid + 1 * \lipschitzconstant * \xmid}
  \def \xsplitone {1.44271}
  \def \xmidleftone {\xsplitone/2}
  \def \xmidrightone {(\xsplitone + \xmax)/2}
  \begin{axis}[
    axis line style = \thickness,
    axis y line=left,
    axis x line=bottom,
    xtick=\empty,
    ytick=\empty,
    x label style={yshift=12.5},
    y label style={yshift=-\ylabelshift},
    ymax=\ymax,
    ymin=\ymin,
    xlabel={$x_{\multind}$},
    ylabel={$x_{\rind}$},
    width=\textwidth+\ylabelshift,
    % height=\tikzimageheight cm,
    axis on top,
    clip=false
    ]
    % Box for left mid point
    \filllipschitzbox{\xmidleftone}{\xsplitone}{black!12}{dashed}
    % Box for right mid point
    \filllipschitzbox{\xmidrightone}{\xmax - \xsplitone}{black!12}{dashed}
    % Box labels
    \addplot[black] coordinates {({\xoffset}, {\myfunc{\xmidleftone} + \lipschitzconstant * \xmidleftone})} node[above right, xshift=-1, yshift=-1] {$\Omega_i^{k+1}(j)$};
    \addplot[black] coordinates {({\xmax+\xoffset}, {\myfunc{\xmidrightone} + \lipschitzconstant * (\xmax - \xmidrightone)})} node[below left, xshift=1, yshift=1] {$\Omega_i^{k+1}(j+1)$};
    % The function
    \addplot[smooth,black,domain=0:\xmax+2*\xoffset,samples=100, \thickness]{\myfunc{x-\xoffset}} node[above,pos=1] {$f_i$};
  \end{axis}
\end{tikzpicture}%
%
%%% Local Variables:
%%% mode: latex
%%% TeX-master: "../multidim-lipschitz-minlp-preprint"
%%% End:

%% file: alg-listing.tex
\begin{algorithmic}[1]
  \REQUIRE Problem~\eqref{eq:mult-problem} and
  $\myeps > 0$.\label{alg:mult-method-require}
  \ENSURE An approximate globally optimal and $\myeps$-feasible
  point for Problem~\eqref{eq:mult-problem} or an indication that
  Problem~\eqref{eq:mult-problem} is infeasible.
  \medskip
  \STATE Set $k \leftarrow 0$ and initialize
  $\boxset_i^0 = \set{B(\lb{x}_{\multind}, \ub{x}_{\multind})}$ for all $i
  \in [p]$.
  \FOR{$k = 0, 1, 2, \dotsc$}
  \STATE Solve the master problem~\eqref{eq:mult-problem-master} to global
  optimality.
  \label{alg:mult-method-solve-master}
  \IF{\eqref{eq:mult-problem-master} is infeasible}
  \RETURN ``Problem~\eqref{eq:mult-problem} is infeasible.''
  \label{alg:mult-method-return-infeasible}
  \ENDIF
  \STATE Let $x^k$ denote the optimal solution of
  \eqref{eq:mult-problem-master}.
  \label{alg:mult-method-returns-master-solution}
  \IF{$\abs{f_i(x_{\multind}^k) - x_{\rind}^k} \leq \myeps$ for
    all $i \in [p]$}
  \label{alg:mult-method-check-return-solution}
  \RETURN $x^k$.
  \label{alg:mult-method-return-solution}
  \ENDIF
  \STATE Determine the boxes $j_i^k \in J_i^k$ for all $i \in [p]$.
  \label{alg:mult-method-find-boxes}
  \STATE Solve the subproblem \eqref{eq:mult-problem-sub} and
  let~$\tilde{x}^k$ denote the optimal solution.
  \FOR{$i \in [p]$}
  \IF{$\abs{f_i(x_{\multind}^k) - x_{\rind}^k} > \myeps$}
  \label{alg:mult-method-check-split-box}
  \STATE $\boxset_i^{k+1} \leftarrow \left(\boxset_i^k \setminus
    \set{B(\lb{v}_i^k(j_i^k), \ub{v}_i^k(j_i^k))}\right) \cup
  \splitbox\left(\lb{v}_i^k(j_i^k), \ub{v}_i^k(j_i^k), \tilde{x}_{\multind}^k\right)$.
  \label{alg:mult-method-split-box}
  \ELSE
  \STATE $\boxset_i^{k+1} \leftarrow \boxset_i^k$.
  \label{alg:mult-method-not-split-box}
  \ENDIF
  \ENDFOR
  \ENDFOR
\end{algorithmic}

%%% Local Variables:
%%% mode: latex
%%% TeX-master: "multidim-lipschitz-minlp-preprint"
%%% End:

%% file: case-study-bilevel.tex
%!TEX root = multidim-lipschitz-minlp-preprint.tex
\section{Application to Nonlinear Bilevel Problems}
\label{sec:application-bilevel}

The method developed in the previous section can be applied to
nonlinear bilevel problems with nonconvex lower-level models, which
is an extremely challenging class of problems.
To illustrate this, we consider optimistic MIQP-QP bilevel problems of
the form
\begin{align}
  \label{prob:nonlinear-bilevel}
  \begin{split}
    \min_{\ulVar, \llVar} \quad
    &
    \frac{1}{2} \ulVar^\top
    H_\ulIdx \ulVar
    + c_\ulIdx^\top \ulVar
    + \frac{1}{2} \llVar^\top G_\ulIdx \llVar
    + d_\ulIdx^\top \llVar
    \\
    \st \quad
    &
    A \ulVar + B \llVar \leq a,
    \quad \lb{x} \leq x \leq \ub{x},
    \quad \ulVar \in \R^{\dimUL},
    \\
    &
    \llVar \in \argmin_{\lowerLLVar} \bigg\{
    \frac{1}{2} \lowerLLVar^\top G_\llIdx \lowerLLVar
    + d_\llIdx^\top \lowerLLVar :
    \begin{aligned}[t]
      & C \ulVar + D \lowerLLVar \leq b,
      \ \lb{y} \leq \lowerLLVar \leq \ub{y},
      \ \lowerLLVar \in \R^\dimLL \bigg\},
    \end{aligned}
  \end{split}
\end{align}
where $\ulVar \in \R^\dimUL$ and $\llVar \in \R^\dimLL$ denote the upper-
and lower-level variables, which are finitely bounded by $\lb{x}$,
$\ub{x}$, $\lb{y}$, and $\ub{y}$.
Further, we have matrices $A \in \R^{\nrConstrUL
  \times \dimUL}$, $B \in \R^{\nrConstrUL\times\dimLL}$,
$C\in \R^{\nrConstrLL\times\dimUL}$,
$D\in\R^{\nrConstrLL\times\dimLL}$,
as well as right-hand side vectors $a\in\R^\nrConstrUL$ and
$b\in\R^\nrConstrLL$.
In addition, we have $c_\ulIdx \in \R^\dimUL$ and $d_\ulIdx, d_\llIdx
\in \R^\dimLL$.
Finally, $H_\ulIdx \in \R^{\dimUL \times \dimUL}$, $G_\ulIdx \in
\R^{\dimLL\times\dimLL}$ are positive semidefinite and
symmetric matrices, while $G_\llIdx\in \R^{\dimLL\times\dimLL}$ is a
possibly indefinite and symmetric matrix.
Thus, the upper level is a convex-quadratic problem over
linear constraints and the lower-level problem
\begin{equation}
  \label{prob:lower-level}
  \min_{\lowerLLVar}
  \quad
  \frac{1}{2} \lowerLLVar^\top G_\llIdx \lowerLLVar
  + d_\llIdx^\top \lowerLLVar
  \quad
  \st
  \quad
  C \ulVar + D \lowerLLVar \leq b,
  \ \lb{y} \leq \lowerLLVar \leq \ub{y},
  \ \lowerLLVar \in \R^\dimLL,
\end{equation}
is an $\ulVar$-parameterized and continuous but nonconvex quadratic
problem.
Let $\varphi(\cdot)$ be the optimal-value function of the
lower level, \ie,
\begin{equation*}
  \varphi({\ulVar})
  \define
  \min_{\lowerLLVar}
  \bigg\{
  \frac{1}{2} \lowerLLVar^\top G_\llIdx \lowerLLVar + d_\llIdx^\top \lowerLLVar :
  C \ulVar + D \lowerLLVar \leq b, \ \lb{y} \leq \lowerLLVar \leq \ub{y}, \ \lowerLLVar \in \R^\dimLL
  \bigg\}.
\end{equation*}
With this, we can rewrite \Probref{prob:nonlinear-bilevel}
equivalently as the single-level problem
\begin{subequations}
  \label{prob:nonlinear-bilevel-single-level}
  \begin{align}
    \min_{\ulVar, \llVar} \quad
    &
      \frac{1}{2} \ulVar^\top
      H_\ulIdx \ulVar
      + c_\ulIdx^\top \ulVar
      + \frac{1}{2} \llVar^\top G_\ulIdx \llVar
      + d_\ulIdx^\top \llVar
    \\
    \st \quad
    &
      A \ulVar + B \llVar \leq a, \quad C \ulVar + D \llVar \leq b,
    \\
    & \lb{x} \leq x \leq \ub{x},
      \quad \lb{y} \leq y \leq \ub{y},
      \quad \ulVar \in \R^\dimUL,
      \quad \llVar \in \R^\dimLL,
    \\
    &
      \frac{1}{2} \llVar^\top G_\llIdx \llVar + d_\llIdx^\top \llVar
      \leq
      \varphi(\ulVar), \label{prob:nonlinear-bilevel-single-level:optimalityLL}
  \end{align}
\end{subequations}
see, \eg, \cite{Dempe:2002}.
We now reformulate Problem~\eqref{prob:nonlinear-bilevel-single-level}
so that it fits into the framework introduced above.
Therefore, we introduce the auxiliary variables $\eta_1$ and $\eta_2$
as well as the nonlinear function $f: \R^{\dimLL} \rightarrow \R$ with
$f(\llVar) = 1/2 \llVar^\top G_\llIdx \llVar + d_\llIdx^\top \llVar$.
Based on this notation,
Problem~\eqref{prob:nonlinear-bilevel-single-level} can be restated as
\begin{subequations}
  \label{prob:nonlinear-bilevel-single-level-reform-to-fit-our-setup}
  \begin{align}
    \min_{\ulVar, \llVar, \eta_1, \eta_2} \quad
    &
      \frac{1}{2} \ulVar^\top
      H_\ulIdx \ulVar
      + c_\ulIdx^\top \ulVar
      + \frac{1}{2} \llVar^\top G_\ulIdx \llVar
      + d_\ulIdx^\top \llVar
    \\
    \st \quad
    &
      A \ulVar + B \llVar \leq a,
      \quad C \ulVar + D \llVar \leq b,
    \\
    & \lb{x} \leq x \leq \ub{x},
      \quad \lb{y} \leq y \leq \ub{y},
      \quad \ulVar \in \R^\dimUL,
      \quad \llVar \in \R^\dimLL,
    \\
    & \eta_2 - \eta_1 \leq 0,
      \quad \varphi(\ulVar) = \eta_1,
      \quad f(\llVar) = \eta_2,
      \quad \eta_1, \eta_2 \in \R.
  \end{align}
\end{subequations}
Now, the method developed in Section~\ref{sec:algorithm} can be
applied to
\eqref{prob:nonlinear-bilevel-single-level-reform-to-fit-our-setup} if
(i) the nonconvex functions $\varphi$ and $f$ are Lipschitz
continuous on the projections of the bilevel constraint region onto
the decision space of the upper and lower level, respectively, \ie,
on the domains
\begin{align*}
  \mathcal{F}_\ulVar
  & \define \Defset{\ulVar \in [\lb{x}, \ub{x}]}{\exists \llVar \in \R^\dimLL
    \text{ such that } A \ulVar + B \llVar \leq a, \, C \ulVar + D
    \llVar \leq b, \, \lb{y} \leq y \leq \ub{y}},
  \\
  \mathcal{F}_\llVar
  & \define \Defset{\llVar \in [\lb{y}, \ub{y}]}{\exists \ulVar \in \R^\dimUL
    \text{ such that } A \ulVar + B \llVar \leq a, \, C \ulVar + D
    \llVar \leq b, \, \lb{x} \leq x \leq \ub{x}},
\end{align*}
and if (ii) the associated Lipschitz constants are computable.

What makes things more complicated compared to the general setup
described in Section~\ref{sec:algorithm} is that we can only evaluate
the optimal-value function~$\varphi(\ulVar)$ but we cannot optimize
over it.
Thus, we cannot use subproblem~\eqref{eq:mult-problem-sub} directly
to obtain a new splitting point.
On the other hand, following the strategy to take the box center~$m$
as the new splitting point simplifies solving problem
\eqref{eq:mult-problem-sub} to evaluating~$\varphi(m)$.
More precisely, using the box center corresponds to setting $\lambda =
1/2$.
This, however, is only applicable if $\varphi(m)$ is well-defined,
which we ensure with the following assumption.
\begin{assumption}
  \label{ass:compactness_LL}
  The set $\mathcal{T}(\ulVar) \define \defset{ \llVar \in \R^{\dimLL}}{D
    \llVar \leq b - C \ulVar, \, \lb{y} \leq y \leq \ub{y}}$ is nonempty
  for all $\ulVar \in B(\lb{\ulVar}, \ub{\ulVar})$.
\end{assumption}
This assumption implies that $- \infty < \varphi(\ulVar) < +
\infty$ for all $\ulVar \in B(\lb{\ulVar}, \ub{\ulVar})$, \ie, a
minimizer~$\llVar$ of Problem \eqref{prob:lower-level} exists for all
$\ulVar \in B(\lb{\ulVar}, \ub{\ulVar})$ and, thus, for every
possible box center~$m$.

\rev{Before we present the theoretical developments that are required
  to apply our method to the introduced class of bilevel problems, let
  us brief\/ly discuss that interpreting approximate solutions of
  bilevel problems with nonconvex lower-level problems needs to be
  done with some care.
  In particular, it is shown in \cite{Beck_et_al:2022c} that
  lower-level solutions that are only $\eps$-feasible can lead to
  upper-level solutions that can be arbitrary far away from actual
  bilevel solutions.
  Since this also applies to our method, we later always present the
  difference of our solutions to the known optimal solutions of the
  bilevel problems in our test set.}

%%%%%%%%%%%%%%%%%%%%%%%%%%%%%%%%%%%%%%%%%%%%%%%%%%%%%%%%%%%%%%%%%%%%%%%%%%%
\subsection{Lipschitz Continuity Properties}
\label{sec:lipsch-cont-prop}

To apply our method with the outlined modifications for the subproblem
to Problem~\eqref{prob:nonlinear-bilevel-single-level-reform-to-fit-our-setup},
it remains to show that the properties (i) and (ii) are fulfilled.
We start with the nonconvex function~$f$.
Since the relevant domain $B(\lb{\llVar}, \ub{\llVar})$ of this
function is compact, continuous differentiability of~$f$ implies
global Lipschitz continuity of~$f$ on this set.
Since $B(\lb{\llVar}, \ub{\llVar})$ is convex and compact, the
tightest Lipschitz constant can be computed by solving the
optimization problem
\begin{equation}
  \label{eq:Lipschitz_constant_f}
  \max_{\llVar \in B(\lb{\llVar}, \ub{\llVar})}
  || G_\llIdx  \llVar + d_\llIdx ||.
\end{equation}
Note that it would also be possible to compute the Lipschitz
constant in~\eqref{eq:Lipschitz_constant_f}
over the feasible set of the master problem, \ie, over the
set~$\mathcal{F}_\llVar$.
However, this involves solving an optimization problem not over
a simple box but over a more complex polytope.
In our computational study, we test both variants.
We will denote the former method as the ``fast'' method and the
latter as the ``slow'' method.
In the absence of lower- and upper-level constraints
except for simple variable bounds, both approaches coincide.

Next, we continue with the more difficult case of proving Lipschitz
continuity of the optimal-value function $\varphi$.
To this end, we exploit a variant of the Hoffman Lemma; for the
original version see the main theorem in \cite{Hoffman:1952} or
Lemma~5.8 in \cite{Still:2018}.
For the ease of presentation, we assume from now on that the finite
bounds on~$y$ are part of the lower-level inequality constraints and,
thus, also part of the matrix~$C$.
\begin{lemma}[see Corollary 5.1 in \cite{Still:2018}]
  \label{lemma:variant_Hoffman_lemma}
  There exists $L_H > 0$ such that for any~$\ulVar,\tilde{\ulVar} \in
  B(\lb{\ulVar}, \ub{\ulVar})$ it holds:
  For any $\llVar \in \mathcal{T}(\ulVar)$ we can find a point
  $\tilde{\llVar} \in \mathcal{T}(\tilde{\ulVar})$ with
  \begin{align*}
    ||\tilde{\llVar}-\llVar||
    \leq L_H ||C (\ulVar-\tilde{\ulVar})||
    \leq L_H ||C|| \, ||\tilde{\ulVar}-\ulVar||.
  \end{align*}
\end{lemma}
The scalar $L_H$ is the so-called Hoffman constant.
A sharp characterization of this constant and an algorithm to compute
it can be found in \cite{Li:1993,Pena:2018}.
Based on the introduced variant of the Hoffman Lemma, we can now
establish Lipschitz continuity of the optimal-value function $\varphi$
under Assumption \ref{ass:compactness_LL}.
Our proof follows the idea of the proof of Corollary~5.2
in~\cite{Still:2018}.
There, the Lipschitz continuity of the
optimal-value function of a linear program with right-hand side
perturbation is demonstrated.
In contrast to this, we have a quadratic program with right-hand side
perturbation.
\begin{theorem}
  Suppose Assumption~\ref{ass:compactness_LL} holds.
  Then, there exists $L > 0$ such that for any $\ulVar,\tilde{\ulVar}
  \in B(\lb{\ulVar}, \ub{\ulVar})$ it holds
  \begin{align*}
    | \varphi(\tilde{\ulVar}) - \varphi(\ulVar) |
    \leq
    L || \tilde{\ulVar} - \ulVar ||.
  \end{align*}
\end{theorem}
\begin{proof}
  Let any $\ulVar,\tilde{\ulVar} \in B(\lb{\ulVar}, \ub{\ulVar})$ be
  given.
  By Assumption~\ref{ass:compactness_LL}, minimizers $\llVar$ and
  $\tilde{\llVar}$ of Problem~\eqref{prob:lower-level} exist for
  $\ulVar$ and $\tilde{\ulVar}$, respectively.
  By Lemma~\ref{lemma:variant_Hoffman_lemma}, for $\llVar$ we can find
  a point $\hat{\llVar} \in \mathcal{T}(\tilde{\ulVar})$ such that
  \begin{align*}
    ||\hat{\llVar} - \llVar||
    \leq
    L_H ||C|| \, || \tilde{\ulVar} - \ulVar ||
  \end{align*}
  holds for some $L_H > 0$.
  Based on this, we can conclude
  \begingroup
  \allowdisplaybreaks
  \begin{align*}
   \varphi(\tilde{\ulVar}) &- \varphi(\ulVar)
    \\
    \leq \ & \frac{1}{2} \hat{\llVar}^\top G_\llIdx \hat{\llVar} +
             d_\llIdx^\top \hat{\llVar} - \left(\frac{1}{2} \llVar^\top
             G_\llIdx \llVar + d_\llIdx^\top \llVar \right)
    \\
    = \ & \frac{1}{2} \hat{\llVar}^\top G_\llIdx \hat{\llVar} -
             \frac{1}{2} \llVar^\top G_\llIdx \llVar + d_\llIdx^\top
             (\hat{\llVar} - \llVar)
    \\
    = \ & \frac{1}{2} \hat{\llVar}^\top G_\llIdx \hat{\llVar} -
             \frac{1}{2} \hat{\llVar}^\top G_\llIdx \llVar + \frac{1}{2}
             \hat{\llVar}^\top G_\llIdx \llVar
             - \frac{1}{2} \llVar^\top G_\llIdx \llVar  + \frac{1}{2}
             \llVar^\top G_\llIdx \hat{\llVar} - \frac{1}{2} \llVar^\top
             G_\llIdx \hat{\llVar}
             + d_\llIdx^\top (\hat{\llVar} - \llVar)
    \\
    = \ & \frac{1}{2} \hat{\llVar}^\top G_\llIdx (\hat{\llVar} - \llVar) +
             \frac{1}{2} \llVar^\top G_\llIdx (\hat{\llVar} - \llVar) +
             \frac{1}{2} \hat{\llVar}^\top G_\llIdx \llVar - \frac{1}{2}
             \llVar^\top G_\llIdx \hat{\llVar}
             + d_\llIdx^\top (\hat{\llVar} - \llVar)
    \\
    = \ & \frac{1}{2} \left( \hat{\llVar}^\top G_\llIdx + d_\llIdx^\top +
             \llVar^\top G_\llIdx + d_\llIdx^\top  \right)
             (\hat{\llVar} - \llVar)
             + \frac{1}{2} \llVar^\top G_\llIdx^\top \hat{\llVar} -
             \frac{1}{2} \llVar^\top G_\llIdx \hat{\llVar}
    \\
    \leq \ & \frac{1}{2} \left( ||G_\llIdx^\top \hat{\llVar}  + d_\llIdx|| +
             ||G_\llIdx^\top \llVar  + d_\llIdx||\right)  ||\hat{\llVar}
             - \llVar||
    \\
    \leq \ & L_{G_\llIdx}  ||\hat{\llVar} - \llVar||
    \\
    \leq \ & L_H  ||C||  L_{G_\llIdx}  ||\tilde{\ulVar} - \ulVar||,
  \end{align*}
  \endgroup
  where we use the symmetry of~$G_\llIdx$ and $L_{G_\llIdx} \define
  \max \defset{ ||  G_\llIdx  \llVar  + d_\llIdx|| }{ \ulVar \in
    B(\lb{\ulVar}, \ub{\ulVar}), \llVar \in \mathcal{T}(\ulVar) }$,
  which is well-defined due to Assumption~\ref{ass:compactness_LL}.

  Analogously, by Lemma~\ref{lemma:variant_Hoffman_lemma}, for
  $\tilde{\llVar}$ we can find a point $\hat{\llVar} \in
  \mathcal{T}(\ulVar)$ such that
  \begin{align*}
    ||\hat{\llVar} - \tilde{\llVar}||
    \leq L_H ||C|| \, ||\ulVar-\tilde{\ulVar}||
  \end{align*}
  holds for the same $L_H > 0$.
  With the same arguments as before, we obtain
  \begin{equation*}
    \varphi(\ulVar) - \varphi(\tilde{\ulVar})
    \leq
    L_H ||C|| L_{G_\llIdx} ||\tilde{\ulVar}-\ulVar||.
  \end{equation*}
  Consequently,
  \begin{align*}
    | \varphi(\tilde{\ulVar}) - \varphi(\ulVar) |
    \leq
    L_H ||C|| L_{G_\llIdx} ||\tilde{\ulVar}-\ulVar||
    \enifed L ||\tilde{\ulVar}-\ulVar||
  \end{align*}
  holds and the claim follows.
\end{proof}

\rev{Let us finally note that the presented method is not restricted
  to nonconvex but quadratic problems in the lower level in general.
  If one has knowledge about the Lipschitz constant of a nonconvex
  lower-level problem with more general nonlinearities, the method can
  be applied as it is explained in this section.}

%%%%%%%%%%%%%%%%%%%%%%%%%%%%%%%%%%%%%%%%%%%%%%%%%%%%%%%%%%%%%%%%%%%%%%%%%%%
\subsection{Implementation Details}
\label{sec:impl-deta}

In this section, we discuss some implementation details to clarify
how we modified and extended Algorithm~\ref{alg:mult-method} to get a
more tailored method for the considered bilevel setup.

%%%%%%%%%%%%%%%%%%%%%%%%%%%%%%%%%%%%%%%%%%%%%%%%%%%%%%%%%%%%%%%%%%%%%%%%%%%
\subsubsection{``Slow'' and ``Fast'' Method for $\varphi$}

Because $\mathcal{T}(\ulVar) \subseteq B(\lb{\llVar}, \ub{\llVar}) $
holds for all $\ulVar \in B(\lb{\ulVar}, \ub{\ulVar})$, we immediately
get
\begin{align*}
  L_{G_\llIdx} = \max_{\ulVar \in B(\lb{\ulVar}, \ub{\ulVar}),
  \llVar \in \mathcal{T}(\ulVar)} ||  G_\llIdx  \llVar  + d_\llIdx
  || \leq \max_{\llVar \in B(\lb{\llVar}, \ub{\llVar})} ||
  G_\llIdx  \llVar  + d_\llIdx ||.
\end{align*}
Since we have to compute $L_{G_\llIdx}$ for computing the Lipschitz
constant of $\varphi$, we can distinguish between the ``fast'' and the
``slow'' method for $\varphi$ as well.

%%%%%%%%%%%%%%%%%%%%%%%%%%%%%%%%%%%%%%%%%%%%%%%%%%%%%%%%%%%%%%%%%%%%%%%%%%%
\subsubsection{Additional Nonlinearities}
\label{sec:additional-nonlinearities}

Bilinear nonlinearities of the
form~$\ulVar_i \llVar_j$ in the lower- or upper-level objective
function can be easily reformulated to fit in our setup.
If such nonlinearities occur,
\eg, in the lower level, an additional variable $\llVar_k$ is
introduced in the lower level. Moreover, the constraint
$\llVar_k = \ulVar_i$ is added to the lower level, while
the nonlinear objective term $\ulVar_i \llVar_j$ is replaced by
the product~$\llVar_k\llVar_j$.
The resulting bilevel problem then fits in our setup.

%%%%%%%%%%%%%%%%%%%%%%%%%%%%%%%%%%%%%%%%%%%%%%%%%%%%%%%%%%%%%%%%%%%%%%%%%%%
\subsubsection{Box Filtering}
\label{sec:box-filtering}

Figure~\ref{fig:box_filtering} illustrates the case
that after splitting an initial bounding box, here $[0,3]^2$,
a few times, there might be boxes (such as $[2.25, 3]^2$) that do not
include any point from the bilevel constraint region, which is colored
in red in Figure~\ref{fig:box_filtering}.
To avoid further investigation of these boxes, we can detect these
boxes by checking if the intersection of the bilevel constraint region
with newly created boxes is empty.
However, this requires to solve an LP feasibility problem and thus
creates some additional computational effort.
\rev{The benefit, however, is that constraints
  \eqref{eq:mult-problem-master-mip-polytope-bigm-lower}--%
  \eqref{eq:mult-problem-master-mip-lipschitz-geq}
  and the corresponding binary variables are not added to the master
  problem for the filtered boxes in the given application. However, as
  the bilevel constraint region is also accounted for in the master
  problem, these filtered boxes are not further splited anyway, i.e.,
  the constraints and variables that are not added to the master
  problem would be redundant if added. Hence, the additional
  computational effort of box filtering should only be undertaken if
  necessary.}

\begin{figure}
  \centering
  \input{tikz-imgs/box-filtering.tikz}
  \caption{Example for box filtering}
  \label{fig:box_filtering}
\end{figure}
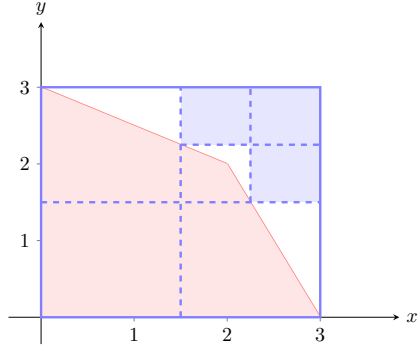

Note that box filtering is not necessary if there are no
lower-level and upper-level constraints except for simple variable
bounds since the intersection of the bilevel constraint
region and every possible box created by the algorithm can never be
empty.
\rev{In contrast, box filtering is necessary if the ``slow'' method is applied since, e.g., Problem~\eqref{eq:Lipschitz_constant_f} for computing the Lipschitz constant of the nonconvex function $f$ might be infeasible on newly created boxes. Thus, these boxes must be filtered after each box splitting and before Lipschitz constants are updated. This does not occur with the ``fast'' method and box filtering is therefore not necessary for this method.}

%%%%%%%%%%%%%%%%%%%%%%%%%%%%%%%%%%%%%%%%%%%%%%%%%%%%%%%%%%%%%%%%%%%%%%%%%%%
\subsubsection{Tighter Lipschitz Constants for Box-Constrained Lower
  Levels}
For instances with simple variable bounds on the lower
level that are not influenced by upper-level decisions,
we can compute tighter Lipschitz constants for $\varphi$.
To this end, we now explicitly take into account
bilinear terms of the form $\ulVar_i \llVar_j$ in the
lower-level objective function and do not reformulate
these terms
as described in Section~\ref{sec:additional-nonlinearities}.
Hence, the lower-level objective function is
given by
\begin{align*}
  f(\llVar, \ulVar) = \frac{1}{2}
  \begin{pmatrix}
    \llVar \\ \ulVar
  \end{pmatrix}^\top
  \begin{bmatrix}
    E & F^\top \\
    F & 0
  \end{bmatrix}
        \begin{pmatrix}
          \llVar \\ \ulVar
        \end{pmatrix}
  +
  \begin{pmatrix}
    e \\ 0
  \end{pmatrix}^\top
  \begin{pmatrix}
    \llVar \\ \ulVar
  \end{pmatrix},
\end{align*}
where $E$ and $F \neq 0$ are suitably chosen matrices and $e$ is a suitably
chosen vector.
Now, for given $\hat{\ulVar}, \tilde{\ulVar} \in B(\lb{\ulVar},
\ub{\ulVar})$, let $\hat{\llVar}$ and $\tilde{\llVar}$ be defined as
\begin{align}
  \label{eq:tighter_Lipschitz_constant}
  \hat{\llVar} \define \argmin_{\llVar \in B(\lb{\llVar},
  \ub{\llVar})} f(\llVar, \hat{\ulVar}),
  \quad
  \tilde{\llVar} \define \argmin_{\llVar \in B(\lb{\llVar},
  \ub{\llVar})} f(\llVar, \tilde{\ulVar}),
\end{align}
where $B(\lb{\llVar}, \ub{\llVar})$ is the lower level's feasible set
in this case.
Note that $B(\lb{\llVar}, \ub{\llVar}) = \mathcal{T}(\ulVar)$ holds
for all $\ulVar \in B(\lb{\ulVar}, \ub{\ulVar})$ because $\ulVar$ only
influences the lower-level objective function but not the lower
level's feasible set.
Thus, it holds
\begin{align*}
  \varphi(\hat{\ulVar}) - \varphi(\tilde{\ulVar}) =\ &
  f(\hat{\llVar},\hat{\ulVar}) - f(\tilde{\llVar}, \tilde{\ulVar})
  \leq
  f(\tilde{\llVar},\hat{\ulVar}) - f(\tilde{\llVar},
  \tilde{\ulVar})
  \\
  = \ &
  \tilde{\llVar}^\top F^\top \hat{\ulVar}  - \tilde{\llVar}^\top
    F^\top \tilde{\ulVar}
  \leq
  ||F \tilde{\llVar}|| \, || \hat{\ulVar} - \tilde{\ulVar} ||.
\end{align*}
Note that $f(\hat{\llVar},\hat{\ulVar}) \leq f(\tilde{\llVar},
\hat{\ulVar})$ holds since $\hat{\llVar}$ minimizes
$f(\llVar,\hat{\ulVar})$ over $B(\lb{\llVar}, \ub{\llVar})$ and
$\tilde{y} \in B(\lb{\llVar}, \ub{\llVar})$.
This is not necessarily true for the case of general lower-level
constraints, \ie, when optimizing
in~\eqref{eq:tighter_Lipschitz_constant} over
$\mathcal{T}(\hat{\ulVar})$ and $\mathcal{T}(\tilde{\ulVar})$,
respectively, since $\mathcal{T}(\hat{\ulVar}) \neq
\mathcal{T}(\tilde{\ulVar})$ and $\tilde{\llVar} \notin
\mathcal{T}(\hat{\ulVar})$ might hold.
Analogously, we obtain $\varphi(\tilde{\ulVar}) -
\varphi(\hat{\ulVar}) \leq ||F \hat{\llVar}|| \, || \hat{\ulVar} -
\tilde{\ulVar} ||$ and, consequently,
\begin{align*}
  | \varphi(\hat{\ulVar}) - \varphi(\tilde{\ulVar}) |
  \leq L_F || \hat{\ulVar} - \tilde{\ulVar} ||
\end{align*}
holds with $L_F \define \max \defset{||Fy||}{\llVar \in B(\lb{\llVar},
  \ub{\llVar})}$.
Thus, $L_F$ is a valid Lipschitz constant.

%%%%%%%%%%%%%%%%%%%%%%%%%%%%%%%%%%%%%%%%%%%%%%%%%%%%%%%%%%%%%%%%%%%%%%%%%%%
\subsubsection{Big-$M$s}

As a valid Big-$M$ in the master problem, we use the
maximum of $||\ub{\llVar} - \lb{\llVar}||_\infty$,
$||\ub{\ulVar} - \lb{\ulVar}||_\infty$, and
\begin{equation*}
  \max_{\llVar \in B(\lb{\llVar}, \ub{\llVar})} \frac{1}{2}
  \llVar^\top G_\llIdx \llVar + d_\llIdx \llVar
  -
  \min_{\llVar \in B(\lb{\llVar}, \ub{\llVar})} \frac{1}{2}
  \llVar^\top G_\llIdx \llVar + d_\llIdx \llVar.
\end{equation*}

%%%%%%%%%%%%%%%%%%%%%%%%%%%%%%%%%%%%%%%%%%%%%%%%%%%%%%%%%%%%%%%%%%%%%%%%%%%
\subsubsection{Lipschitz Constant Updates}

The Lipschitz constants are updated after each box splitting since it
is to be expected that the constants get smaller if they are computed
on smaller sets.

%%%%%%%%%%%%%%%%%%%%%%%%%%%%%%%%%%%%%%%%%%%%%%%%%%%%%%%%%%%%%%%%%%%%%%%%%%%
\subsection{Numerical Results}
\label{sec:numerical-results-bilevel}

In our computational study below, we consider the QP-QP instances
from the \BASBLib library \cite{BASBLib}.
We first describe which instances need to be excluded because they do
not fit into our setup.
First, we exclude the three instances \sfname{d\_1992\_01},
\sfname{b\_1984\_02}, and \sfname{dd\_2012\_02}
because they contain nonlinear constraints.
Next, the two instances \sfname{y\_1996\_02} and
\sfname{lmp\_1987\_01} are excluded due to nonconvex upper-level
objective functions, \ie, the matrix $H_\ulIdx$ is not positive
semidefinite for these instances.
The instance \sfname{sc\_1998\_01} is not considered because $C$ is
the zero matrix and, thus, the resulting optimization problem is not a
``true'' bilevel problem because the lower-level problem is not
constrained by upper-level variables.
Such problems can easily be solved by backwards induction\rev{, \ie, an optimal solution to the lower-level problem can first be determined and, given this lower-level optimal solution, the upper-level problem can be solved to optimality}.
Finally, we have to exclude the following four instances because they
violate Assumption~\ref{ass:compactness_LL}:
\begin{itemize}
\item \sfname{tmh\_2007\_01} (lower level is not feasible for $x=9$),
\item \sfname{b\_1988\_01} (lower level is not feasible for $x=9$),
\item \sfname{b\_1998\_07} (lower level is not feasible for $x=9$), and
\item \sfname{cw\_1990\_02} (lower level is not feasible for $x=7$).
\end{itemize}
In total, 10~instances (out of 20) remain; see
Table~\ref{tab:overview_instances}.
Note that, due to the applied reformulations as described in
Section~\ref{sec:additional-nonlinearities}, the reported number of
variables might differ from those reported in the \BASBLib.
Moreover, the reported number of constraints does not include
additional constraints necessary due to these reformulations.
Variable bounds are also not counted as constraints.

\begin{table}
  \caption{Overview of the considered \BASBLib instances.
    The number of lower- and upper-level constraints do not contain
    simple variable bounds and do not contain the additional
    constraints obtained due to the reformulation discussed in
    Section~\ref{sec:additional-nonlinearities}.
    Contrarily, the additional variables due to this
    reformulation are counted.}
  \label{tab:overview_instances}
  \input{tables/overview-basblib-instances}
\end{table}

We implemented the algorithm in \Python~3.7.9.
All computations were conducted on a machine with an \sfname{Intel(R)
  Core(TM) i7-8565U} CPU with $4$~cores, \SIrange{1.8}{4.6}{\giga\hertz},
and \SI{16}{\giga\byte}\,RAM.
The master problem~\eqref{eq:mult-problem-master} and the
subproblem~\eqref{eq:mult-problem-sub} are modeled using
\sfname{Pyomo}~5.7.2~\cite{hart2017pyomo} and solved with
\sfname{Gurobi}~9.1.0~\cite{gurobi}.
The Hoffman constant is computed with the algorithm described in
\cite{Pena:2018} using the \sfname{MATLAB} code made publicly
available by the authors of \cite{Pena:2018}.\footnote{The
  \sfname{MATLAB} code can be found at
  \url{http://www.andrew.cmu.edu/user/jfp/hoffman.html}.}
We set $\myeps = 10^{-1}$ and use a time limit of \SI{5}{\hour}.
If the instance is solved within the time limit, we
decrease~$\myeps$ by subsequently dividing by ten until
$\myeps = 10^{-5}$ is reached to see how much accuracy can be
reached by our algorithm in the given time limit.
The obtained results for the case of box filtering and determining
the Lipschitz constant with the ``slow'' method are summarized in
Table~\ref{tab:slow_filtering}, while the results for the ``fast''
method without box filtering are given in
Table~\ref{tab:fast_no_filtering}.
Note that Table~\ref{tab:slow_filtering} contains results for
4~instances while Table~\ref{tab:fast_no_filtering} contains results
for 10~instances.
The reason is that 6~instances only have simple variable bounds so
that there is no difference between the ``slow'' and the ``fast''
method \rev{ as well as between box filtering and no box filtering (except for the additional computational effort due to the LP feasibility problems solved in the former case, see Section~\ref{sec:box-filtering})}.
In these cases, we only list the respective instances in
Table~\ref{tab:fast_no_filtering}.
The tables are organized as follows.
The first column states the ID of the instance and the second one
states the used $\myeps$ for the termination criterion.
The number of required iterations is denoted by $k$, ``runtime''
states the runtimes in seconds, and the ``final $\myeps$'' column
contains the tolerance that is actually reached.
Finally, the columns ``diff to opt.'' and ``diff to opt.\ value''
contain the 2-norm distance of our solution to the one
reported in the \BASBLib and the respective difference
in the objective value.

\begin{table}
  \caption{Computational results with ``slow'' method and box
    filtering.
    Runtimes are given in seconds.}
  \label{tab:slow_filtering}
  \input{tables/results-w-slow-and-box-filter}
\end{table}

\begin{table}
  \caption{Computational results with ``fast'' method and no box
    filtering.
    Runtimes are given in seconds.}
  \label{tab:fast_no_filtering}
  \input{tables/results-w-fast-and-no-box-filter}
\end{table}

Before we discuss the results in detail, let us comment on two
important aspects.
First, it can be expected that our method performs rather bad if
there are multiplicities.
To get a finer relaxation, all boxes covering the multiple solutions
have to be split at least once if $\myeps$-feasibility is not
reached by the first split.
Second, it is to be expected that our method performs rather good for
instances with small variable ranges, \eg, ranges such as
$[0,1]^\dimUL \times [0,1]^\dimLL$ instead of $[0,1000]^\dimUL \times
[0,1000]^\dimLL$, as well as small lower-level objective function
ranges.
In particular, the range of the lower level's objective function is
small in the case of small Lipschitz constants.
The Lipschitz constant derived here for the optimal-value function is
valid for all quadratic programs with right-hand side perturbations
that satisfy Assumption~\ref{ass:compactness_LL}.
However, for specific bilevel applications much tighter Lipschitz
constants might be derived by exploiting problem-specific structural
properties.

In what follows, we comment on the obtained computational
results in the light of these two aspect.
Multiplicities are reported for instance \sfname{as\_1981\_01}.
As can be seen in Table~\ref{tab:slow_filtering}, for
$\myeps = 10^{-2}$, this instance is not solved to
$\myeps$-feasibility within the time limit despite
the fact that an optimal solution is already reached.
This effect is even enhanced if box filtering is deactivated \rev{ and the ``fast'' method is used};
see Table~\ref{tab:fast_no_filtering}.
In this case, the final $\myeps$ is 47.3, which is very large.
As this instance has rather many general constraints, box filtering \rev{ with the ``slow'' method}
improves the solution process significantly.
However, for instances with less general constraints like
\sfname{sa\_1981\_01} and \sfname{sa\_1981\_02}, the
additional computational burden of the ``slow'' method and box
filtering can outweigh its advantages.

Finally, let us discuss the results in dependence of the tightness of
the Lipschitz constants and the size of the variable ranges.
To this end, we first order the considered instances by increasing
Lipschitz constants of the function $f$:
\sfname{b\_1998\_02}, \sfname{b\_1998\_03},
\sfname{d\_2000\_01}, \sfname{d\_1978\_01}, \sfname{fl\_1995\_01},
\sfname{sa\_1981\_02}, \sfname{as\_1981\_01}, \sfname{sa\_1981\_01},
\sfname{b\_1998\_04}, and \sfname{b\_1998\_05}.\footnote{In case of
identical Lipschitz constants, alphabetical sorting is used.
The ordering of the considered instances by increasing Lipschitz
constants of the optimal-value function $\varphi$ is slightly different:
\sfname{b\_1998\_02}, \sfname{b\_1998\_03},
\sfname{d\_1978\_01}, \sfname{fl\_1995\_01}, \sfname{sa\_1981\_02},
\sfname{d\_2000\_01}, \sfname{sa\_1981\_01}, \sfname{b\_1998\_04},
\sfname{b\_1998\_05}, and \sfname{as\_1981\_01}.}
Indeed, the three instances with lowest Lipschitz constants are solved
for the smallest~$\myeps$ within the time limit; see
Tables~\ref{tab:slow_filtering} and~\ref{tab:fast_no_filtering}.
In addition to these 3~instances, the ``fast'' method with no box
filtering also solves the instances \sfname{d\_1978\_01},
\sfname{sa\_1981\_01}, and \sfname{b\_1998\_04} at least for the
initial $\myeps$.
Besides having one of the lowest Lipschitz constants, the instance
\sfname{d\_1978\_01} has relatively low variable ranges.
The other two instances \sfname{sa\_1981\_01} and \sfname{b\_1998\_04}
have relatively low number of lower- and
upper-level variables, which leads to
reduced dimensions of the boxes
and therefore of the worst-case number of iterations.
Nevertheless, the disadvantage of the large Lipschitz constant of
\sfname{b\_1998\_04} is reflected in the results as for $\myeps =
10^{-1}$ already over \SI{2}{\hour} are needed to compute an
$\myeps$-feasible solution.

In total, our method solves 7 out of 10~instances.
Interestingly, 2 out of the 3 unsolved instances even have a convex
lower-level problem and could thus be solved with specialized methods
such as, \eg, \cite{Kleinert_et_al:2021b} that explicitly exploit this
property.
\rev{In addition, as our method only requires very weak assumptions and can be applied to a broad range of problems besides nonconvex bilevel problems, it cannot be expected that it outperforms specifically tailored methods like, e.g., the BASBL solver \cite{Paulavicius_et_al:2020}.}

%%% Local Variables:
%%% mode: latex
%%% TeX-master: "multidim-lipschitz-minlp-preprint"
%%% End:

%% file: tikz-imgs/box-filtering.tikz
\begin{tikzpicture}[scale=0.75]
  \begin{axis}[ymin=0, ymax=3.5, xmin=0, xmax=3.5, axis lines=middle, xtick={0, 1, 2, 3}, ytick={0, 1, 2, 3}, xlabel=$x$, ylabel=$y$, every axis x label/.style={
      at={(ticklabel* cs:1.00)},
      anchor=west,
    },
    every axis y label/.style={
      at={(ticklabel* cs:1.00)},
      anchor=south,
    },enlargelimits=true]
    \addplot [domain=0.0:2.0, samples=100, thick, color=red!50]
    {3 - 1 / 2 * x};
    \addplot [domain=2.0:3.0, samples=100, thick, color=red!50]
    {6 - 2 * x};
    \path [fill=red!10]
    (0,300) --
    (0,0) --
    (300,0) --
    (200,200) --
    (0,300);
    \path [fill=blue!10]
    (225,225) --
    (300,225) --
    (300,300) --
    (225,300) --
    (225,225);
    \path [fill=blue!10]
    (150,225) --
    (225,225) --
    (225,300) --
    (150,300) --
    (150,225);
    \path [fill=blue!10]
    (225,150) --
    (300,150) --
    (300,225) --
    (225,225) --
    (225,150);
    \draw[very thick, blue!50] (0,0) rectangle (300,300);
    \draw[dashed, very thick, blue!50] (150,0) -- (150,300);
    \draw[dashed, very thick, blue!50] (0,150) -- (300,150);
    \draw[dashed, very thick, blue!50] (225,150) -- (225,300);
    \draw[dashed, very thick, blue!50] (150,225) -- (300,225);
  \end{axis}
\end{tikzpicture}%
%
%
%%% Local Variables:
%%% mode: latex
%%% TeX-master: "../multidim-lipschitz-minlp-preprint"
%%% End:

%% file: tables/overview-basblib-instances.tex
\begin{tabular}{lccccc}
  \toprule
  ID & $\dimUL$ & $\nrConstrUL$ & $\dimLL$ & $\nrConstrLL$ & $G_\llIdx
                                                             \succeq 0$?
  \\
  \midrule
  \sfname{as\_1981\_01} & 4 & 1 & 4 & 4 & yes \\
  \sfname{b\_1998\_02} & 2 & 0 & 3 & 0 & no\\
  \sfname{b\_1998\_03} & 2 & 0 & 3 & 0 & no \\
  \sfname{b\_1998\_04} & 1 & 0 & 2 & 0 & no \\
  \sfname{b\_1998\_05} & 1 & 0 & 2 & 0 & no \\
  \sfname{d\_1978\_01} & 2 & 0 & 4 & 0 & yes \\
  \sfname{d\_2000\_01} & 2 & 0 & 2 & 2 &  no\\
  \sfname{fl\_1995\_01} & 2 & 0 & 4 & 0 & yes \\
  \sfname{sa\_1981\_01} & 1 & 1 & 2 & 1 & yes \\
  \sfname{sa\_1981\_02} & 2 & 2 & 4 & 0 & yes \\
  \bottomrule
\end{tabular}%
%
%
%%% Local Variables:
%%% mode: latex
%%% TeX-master: "../multidim-lipschitz-minlp-preprint"
%%% End:

%% file: tables/results-w-slow-and-box-filter.tex
\begin{tabular}{c%
 c%
 c%
 c%
 S[round-mode = places, round-precision = 3, table-format=2.3]%
 S[round-mode = places, round-precision = 3, table-format=1.3]%
 S[round-mode = places, round-precision = 3, table-format=2.3]%
}
\toprule
{ID}
 & {$\varepsilon$}
 & {$k$}
 & {runtime}
 & {final $\varepsilon$}
 & {diff to opt.}
 & {diff to opt.\ value}\\
\midrule
\sfname{as\_1981\_01} & $10^{-1}$ & 166 & 3862 & 0.092162 & 0.000 & 0.000\\
\sfname{as\_1981\_01} & $10^{-2}$ & 580 & 18\,000 & 0.014659 & 0.000 & 0.000\\
\midrule
\sfname{d\_2000\_01} & $10^{-1}$ & 37 & 117 & 0.075152 & 0.031250 & 0.000\\
\sfname{d\_2000\_01} & $10^{-2}$ & 64 & 209 & 0.009695 & 0.023739 & 0.000\\
\sfname{d\_2000\_01} & $10^{-3}$ & 95 & 377 & 0.000536 & 0.000466 & 0.000\\
\sfname{d\_2000\_01} & $10^{-4}$ & 123 & 513 & 0.000078 & 0.000059 & 0.000\\
\sfname{d\_2000\_01} & $10^{-5}$ & 151 & 692 & 0.000007 & 0.000010 & 0.000\\
\midrule
\sfname{sa\_1981\_01} & $10^{-1}$ & 2092 & 18\,000 & 0.285832 & 0.378363 & 3.980659\\
\midrule
\sfname{sa\_1981\_02} & $10^{-1}$ & 464 & 18\,000 & 14.862275 & 4.550069 & 89.157104\\
\bottomrule
\end{tabular}

%
%
%
%%% Local Variables:
%%% mode: latex
%%% TeX-master: "../multidim-lipschitz-minlp-preprint"
%%% End:

%% file: tables/results-w-fast-and-no-box-filter.tex
\begin{tabular}{c%
 c%
 c%
 c%
 S[round-mode = places, round-precision = 3, table-format=2.3]%
 S[round-mode = places, round-precision = 3, table-format=2.3]%
 S[round-mode = places, round-precision = 3, table-format=2.3]%
}
\toprule
{ID}
 & {$\varepsilon$}
 & {$k$}
 & {runtime}
 & {final $\varepsilon$}
 & {diff to opt.} & {diff to opt.\ value}\\
\midrule
\sfname{as\_1981\_01} & $10^{-1}$ & 560 & 18\,000 & 47.343750 & 10.440307 & 0.000\\
\midrule
\sfname{b\_1998\_02} & $10^{-1}$ & 4 & 1 & 0.025000 & 0.000 & 0.000\\
\sfname{b\_1998\_02} & $10^{-2}$ & 8 & 5 & 0.008984 & 0.004419 & -0.000005\\
\sfname{b\_1998\_02} & $10^{-3}$ & 16 & 11 & 0.000878 & 0.001924 & -0.000001\\
\sfname{b\_1998\_02} & $10^{-4}$ & 40 & 36 & 0.000098 & 0.000245 & 0.000\\
\sfname{b\_1998\_02} & $10^{-5}$ & 37 & 33 & 0.000010 & 0.000311 & 0.000\\
\midrule
\sfname{b\_1998\_03} & $10^{-1}$ & 5 & 3 & 0.036797 & 0.000 & 0.000\\
\sfname{b\_1998\_03} & $10^{-2}$ & 10 & 7 & 0.008355 & 0.000 & 0.000\\
\sfname{b\_1998\_03} & $10^{-3}$ & 14 & 8 & 0.000795 & 0.001139 & 0.000\\
\sfname{b\_1998\_03} & $10^{-4}$ & 27 & 22 & 0.000095 & 0.000237 & 0.000\\
\sfname{b\_1998\_03} & $10^{-5}$ & 84 & 102 & 0.000006 & 0.000205 & 0.000\\
\midrule
\sfname{b\_1998\_04} & $10^{-1}$ & 1180 & 8038 & 0.060597 & 0.020552 & 0.178423\\
\sfname{b\_1998\_04} & $10^{-2}$ & 1538 & 18\,000 & 0.058177 & 0.019288 & 0.149921\\
\midrule
\sfname{b\_1998\_05} & $10^{-1}$ & 1485 & 18\,000 & 0.476837 & 0.001544 & 0.002048\\
\midrule
\sfname{d\_1978\_01} & $10^{-1}$ & 100 & 237 & 0.058624 & 0.490697 & 0.356140\\
\sfname{d\_1978\_01} & $10^{-2}$ & 507 & 18\,000 & 0.029297 & 0.328590 & 0.236679\\
\midrule
\sfname{d\_2000\_01} & $10^{-1}$ & 179 & 193 & 0.097484 & 0.000393 & 0.000\\
\sfname{d\_2000\_01} & $10^{-2}$ & 136 & 132 & 0.006502 & 0.000497 & 0.000\\
\sfname{d\_2000\_01} & $10^{-3}$ & 177 & 272 & 0.000762 & 0.000022 & 0.000\\
\sfname{d\_2000\_01} & $10^{-4}$ & 206 & 357 & 0.000095 & 0.000009 & 0.000\\
\sfname{d\_2000\_01} & $10^{-5}$ & 385 & 1078 & 0.000006 & 0.000011 & 0.000\\
\midrule
\sfname{fl\_1995\_01} & $10^{-1}$ & 570 & 18\,000 & 0.122925 & 0.414284 & 0.855469\\
\midrule
\sfname{sa\_1981\_01} & $10^{-1}$ & 330 & 205 & 0.082368 & 0.202975 & 2.316284\\
\sfname{sa\_1981\_01} & $10^{-2}$ & 830 & 1221 & 0.008941 & 0.067658 & 0.778198\\
\sfname{sa\_1981\_01} & $10^{-3}$ & 3120 & 18\,000 & 0.001280 & 0.017596 & 0.202974\\
\midrule
\sfname{sa\_1981\_02} & $10^{-1}$ & 725 & 18\,000 & 0.826788 & 1.059649 & 20.994568\\
\bottomrule
\end{tabular}

%% file: case-study-gas.tex
\section{Application to Gas Network Optimization}
\label{sec:case-study-gas}

In this section, we use Algorithm~\ref{alg:mult-method} to solve
stationary gas network optimization problems.
We start by modeling the gas network and state an implicit nonlinear
pressure law function for gas flow in pipes.
We cannot state this function explicitly but it is possible to
evaluate it rather cheaply.
Then, we analyze its derivatives to derive suitable Lipschitz
constants.
Finally, numerical results on test instances show the successful
application of our method.

%%%%%%%%%%%%%%%%%%%%%%%%%%%%%%%%%%%%%%%%%%%%%%%%%%%%%%%%%%%%%%%%%%%%%%%%%%%
\subsection{Modeling}

We model the gas network as a directed and weakly connected
graph~$(\nodes, \arcs)$, where the arcs~$\arcs$ are composed of
pipes~$\pipes$, short pipes~$\shortPipes$, valves~$\valves$,
compressor stations~$\compressorStations$, and control
valves~$\controlValves$, \ie,
\begin{equation*}
  \arcs = \pipes \cup \shortPipes \cup \valves \cup
  \compressorStations \cup \controlValves.
\end{equation*}

The two main variables that describe the state of the gas flowing
through the network are the pressure $\press$ and mass flow $\mflow$.
Each node $\node \in \nodes$ has a bounded pressure variable
$\press_\node \in [\lb{\press}_\node, \ub{\press}_\node]$ and a given
mass flow $\mflow_\node$ that is supplied to or withdrawn from the
network.
A node $\node \in \nodes$ is called an entry node if $\mflow_\node >
0$, an exit node if $\mflow_\node < 0$, and an inner node if $\mflow_\node
= 0$.
In addition, each arc $\arc \in \arcs$ has a variable
$\mflow_\arc \in [\lb{\mflow}_\arc, \ub{\mflow}_\arc]$ that
models the mass flow in the arc.

The balance equation
\begin{equation}
  \label{eq:mass-balance}
  \mflow_\node + \sum_{\arc \in \inArcs} \mflow_\arc
  = \sum_{\arc \in \outArcs} \mflow_\arc
  \quad \text{for all } \node \in \nodes
\end{equation}
ensures that no gas is gained or lost.
Here, $\inArcs$ and $\outArcs$ denote the sets of in- and
outgoing arcs of node~$\node$.

A short pipe $\arc = (\node, \otherNode) \in \shortPipes$ directly
connects its nodes~$\node$ and $\otherNode$.
Therefore, the related pressure values coincide:
\begin{equation}
  \label{eq:short-pipe-pressure}
  \press_\node = \press_\otherNode
  \quad \text{for all } \arc = (\node, \otherNode) \in \shortPipes.
\end{equation}

A valve $\arc = (\node, \otherNode) \in \valves$ is either open or
closed.
If it is open, it is modeled as a short pipe.
If it is closed, the mass flow $\mflow_\arc$ has to be zero and the
related pressures are decoupled but the corresponding pressure
difference between the nodes $\node$ and $\otherNode$ cannot
exceed a given value~$\Delta \ub{\press}_\arc$.
This can be modeled by introducing a binary variable $\valveOpen_\arc$
that indicates if the valve $\arc$ is open ($\valveOpen_\arc = 1$) or
closed ($\valveOpen_\arc = 0$).
The valve model then reads
\begin{subequations}
  \label{eq:valve-model}
  \begin{align}
    \mflow_\arc
    & \geq \lb{\mflow}_\arc \valveOpen_\arc
    & \text{for all } \arc \in \valves,
    \\
    \mflow_\arc
    & \leq \ub{\mflow}_\arc \valveOpen_\arc
    & \text{for all } \arc \in \valves,
    \\
    \press_\node - \press_\otherNode
    & \leq \Delta \ub{\press}_\arc (1 - \valveOpen_\arc)
    & \text{for all } \arc = (\node, \otherNode) \in \valves,
    \\
    \press_\otherNode - \press_\node
    & \leq \Delta \ub{\press}_\arc (1 - \valveOpen_\arc)
    & \text{for all } \arc = (\node, \otherNode) \in \valves.
  \end{align}
\end{subequations}

Compressor stations $\arc \in \compressorStations$ have a fixed flow
direction, \ie, $\lb{\mflow}_\arc \geq 0$ holds, and can increase the
pressure of the gas. We model this pressure increase by introducing
the variable~$\Delta \press_\arc \in [0, \Delta \ub{\press}_\arc]$.
Then, the compressor station model is given by
\begin{equation}
  \label{eq:compressor-station-pressure}
  \press_\otherNode = \press_\node + \Delta \press_\arc
  \quad \text{for all } \arc = (\node, \otherNode) \in
  \compressorStations.
\end{equation}
Note that this model is a significant simplification of how
compressor stations in real gas networks operate.
More complicated and realistic models can be found
in, \eg, \cite{Pfetsch_et_al:2015,Rose_et_al:2016}.

Control valves $\arc \in \controlValves$ are modeled similar to
compressor stations but decrease the gas pressure instead of
increasing it.
We again have $\lb{\mflow}_\arc \geq 0$ and
\begin{equation}
  \label{eq:control-valve-pressure}
  \press_\otherNode = \press_\node - \Delta \press_\arc
  \quad \text{for all } \arc = (\node, \otherNode) \in \controlValves
\end{equation}
with $\Delta \press_\arc \in [0, \Delta \ub{\press}_\arc]$.

Until now, we have modeled all gas network components in a
(mixed-integer) linear way.
The only components that are still missing are pipes~$\arc \in
\pipes$.
We will describe the pressure loss in a pipe in dependence of the
inflow pressure and the mass flow using a nonlinear and Lipschitz
continuous function
\begin{equation}
  \label{eq:gas-equation}
  \press_\otherNode = \press_{\otherNode, \arc}(\press_\node,
  \mflow_\arc) \quad \text{for all } \arc = (\node, \otherNode) \in
  \pipes,
\end{equation}
which we will analyze in the next section.

The goal is to minimize the overall activity of the compressor
stations.
Therefore, the full stationary gas network optimization problem is
given by
\begingroup
\allowdisplaybreaks
\begin{align*}
  \min \quad & \sum_{\arc \in \compressorStations} \Delta \press_\arc\\
  \st \quad & \text{mass balance } \eqref{eq:mass-balance},\\
             & \text{short pipe model } \eqref{eq:short-pipe-pressure},\\
             & \text{valve model } \eqref{eq:valve-model},\\
             & \text{compressor station model } \eqref{eq:compressor-station-pressure},\\
             & \text{control valve model } \eqref{eq:control-valve-pressure},\\
             & \text{pipe model } \eqref{eq:gas-equation},\\
             & \press_\node \in [\lb{\press}_\node, \ub{\press}_\node]
               \quad \text{for all } \node \in \nodes,\\
             & \mflow_\arc \in [\lb{\mflow}_\arc, \ub{\mflow}_\arc]
               \quad \text{for all } \arc \in \arcs,\\
             & \Delta \press_\arc \in [0, \Delta \ub{\press}_\arc]
               \quad \text{for all } \arc \in \compressorStations \cup
               \controlValves,\\
             & \valveOpen_\arc \in \set{0, 1}
               \quad \text{for all } \arc \in \valves.
\end{align*}
\endgroup
This model has the form of Problem~\eqref{eq:mult-problem}.
The pipe equations~\eqref{eq:gas-equation} constitute the Lipschitz
nonlinearities~\eqref{eq:mult-problem-nonlin} while the other
constraints fit~\eqref{eq:mult-problem-lin-and-vars} as they are
linear.
Hence, we can use Algorithm~\ref{alg:mult-method} to solve it.

\rev{In contrast to the gas network model considered
  in~\cite{Schmidt_et_al:2019}, we do not restrict ourselves to
  tree-structured networks.
  This has the consequence that the mass flows in the network cannot be
  pre-computed.
  Therefore, the nonlinearity on each arc is multivariate as it
  depends on the pressure and the mass flow.
  This means that Algorithm~\ref{alg:mult-method} can be applied to a
  much broader class of gas transport models.}

%%%%%%%%%%%%%%%%%%%%%%%%%%%%%%%%%%%%%%%%%%%%%%%%%%%%%%%%%%%%%%%%%%%%%%%%%%%
\subsection{Lipschitz Continuity of the Gas Flow Equation}

In this section, we derive and analyze the nonlinear pipe
model~\eqref{eq:gas-equation}.
For the sake of readability, we henceforth omit the subscript $\arc$
that indicates the pipe $\arc \in \pipes$.

Gas flow along a pipe can be modeled by the stationary momentum
equation.
This ordinary differential equation (ODE) reads
\begin{equation*}
  %\label{eq:euler}
  \dparshort{x}{\left(\press + \frac{\mflux^2}{\dens}\right)} = -\frac{1}{2}
  \nfric \frac{\mflux \abs{\mflux}}{\dens},
  \quad \mflux = \density \velocity,
  \quad \nfric = \frac{\friction}{\diameter}
\end{equation*}
where $\press$, $\velocity$, $\mflux$, and $\dens$ model the pressure,
velocity, mass flux, and density of the gas and~$\friction$ as well as
$\diameter$ denote the pipe's friction coefficient and the diameter of the
pipe; see, \eg, \cite{Gugat_et_al:2018b}.
The relation between the mass flow~$\mflow$ and mass flux~$\mflux$ is
given by $\mflow = \area \mflux$ where~$A = \pi \diameter^2 / 4$ is the
cross-sectional area of the pipe.

The pressure $\press$ and density $\dens$ are coupled by the equation
of state
\begin{equation}
  \label{eq:state}
  \press = \specificGasConstant \temp \compfactor \dens
\end{equation}
for real gas, where $\specificGasConstant$ denotes the specific gas
constant.
The compressibility factor $\compfactor$ can be computed by the
so-called AGA formula
\begin{equation*}
  \compfactor = 1 + \alpha \press,\quad \alpha = 0.257
  \frac{1}{\pseudocriticalPressure} - 0.533
  \frac{\pseudocriticalTemperature}{\pseudocriticalPressure \temp} < 0
\end{equation*}
with pseudocritical pressure $\pseudocriticalPressure$, pseudocritical
temperature $\pseudocriticalTemperature$, and temperature~$\temp$ that we
assume to be constant; see, \eg, \cite{Kralik_et_al:1988}.
We only consider a positive compressibility factor, which is equivalent
to
\begin{equation}
  \label{eq:press_upper_bound}
  \press < \frac{1}{\abs{\alpha}}.
\end{equation}
We further introduce the speed of sound~$\soundspeed$ which is defined
via
\begin{equation*}
  \frac{1}{\soundspeed^2} = \dparlong{\pressure}{\density}
\end{equation*}
and the squared mach number
\begin{equation*}
  \mach = \frac{\velocity^2}{\soundspeed^2}.
\end{equation*}
\begin{lemma}
  It holds
  \begin{equation*}
    \mach = \specificGasConstant \temp \frac{\mflux^2}{\press^2}.
  \end{equation*}
\end{lemma}
\begin{proof}
  We solve the equation of state \eqref{eq:state} for $\dens$ and
  obtain
  \begin{equation*}
    \dens = \frac{\press}{\specificGasConstant \temp \compfactor}.
%    \frac{1}{\dens} = \specificGasConstant \temp \left(\alpha +
%    \frac{1}{\press}\right).
  \end{equation*}
  We can use this and the definition of the speed of sound to get
  \begin{align*}
    \mach & = \velocity^2 \frac{1}{\soundspeed^2}
    = \left(\frac{\mflux}{\dens}\right)^2 \dparlong{\pressure}{\density}
    = \left(\frac{\mflux \specificGasConstant \temp
    \compfactor}{\press}\right)^2
    \frac{\specificGasConstant \temp \compfactor - \press \specificGasConstant
    \temp \alpha}{(\specificGasConstant \temp \compfactor)^2}\\
    & = \left(\frac{\mflux \specificGasConstant \temp
    \compfactor}{\press}\right)^2
    \frac{1}{\specificGasConstant \temp \compfactor^2}
    = \specificGasConstant \temp \frac{\mflux^2}{\press^2}. \qedhere
  \end{align*}
\end{proof}
We assume that the velocity of the gas is subsonic, \ie,
$\mach < 1$ holds, as it is the case for real-world gas networks.
This is equivalent to
\begin{equation}
  \label{eq:press_lower_bound}
  \press > \abs{\mflux} \sqrt{\specificGasConstant \temp}.
\end{equation}
In what follows, we only consider pressures within the interval
$(\abs{\mflux} \sqrt{\specificGasConstant \temp},
1 / \abs{\alpha})$.
Therefore, we have
\begin{equation*}
  \plb > 0 \quad \text{and} \quad \pub > 0,
\end{equation*}
which we will use many times throughout this section.

For a pipe $(\node, \otherNode)$ with length~$\length$ and pressure
$\press_\node$ at node~$\node$, the pressure function reads
\begin{subequations}
\begin{align}
  \label{eq:pressure_function}
  \press(x, \press_\node, \mflux)
  & = F^{-1}\left(F(\press_\node) - \frac{1}{2}
    \specificGasConstant \temp \mflux \abs{\mflux} \nfric x\right),
  \\
  \label{eq:pressure_function_F}
  F(\press)
  & = \frac{1}{\alpha} \press + \left(\mflux^2
    \specificGasConstant \temp - \frac{1}{\alpha^2}\right) \ln(\abs{1+\alpha
    \press}) - \mflux^2 \specificGasConstant \temp \ln(\press),
\end{align}
\end{subequations}
for $x\in[0, \length]$; see
\cite{Gugat_et_al:2018c,Gugat_et_al:2018b}.
From \cite{Gugat_et_al:2018c}, we further know the following
properties of $F$.
\begin{lemma}
  The function $F$ as defined in~\eqref{eq:pressure_function_F} is
  differentiable for $\press \in (\abs{\mflux} \sqrt{\specificGasConstant
  \temp}, 1 / \abs{\alpha})$ with
  \begin{equation}
    \label{eq:pressure_function_F_deriv_p}
    F'(\press) = \frac{\press^2 - \mflux^2 \specificGasConstant
      \temp}{\press (1 + \alpha \press)} > 0.
  \end{equation}
  The second derivative fulfills
  \begin{equation*}
    %\label{eq:pressure_function_F_deriv_p_2}
    F''(\press) = \frac{\press^2 +
    \mflux^2 \specificGasConstant \temp (1 + 2 \alpha \press)}{\press^2 (1 +
    \alpha p)^2} > 0.
  \end{equation*}
\end{lemma}
The property in~\eqref{eq:pressure_function_F_deriv_p} implies that
$F$ is strictly increasing.
Therefore, the inverse in~\eqref{eq:pressure_function} is well-defined.
To evaluate the pressure function $\press$ in
\eqref{eq:pressure_function}, the equation
\begin{equation}
  \label{eq:pressure_equation}
  F(\press) = F(\press_\node) - \frac{1}{2}
  \specificGasConstant \temp \mflux \abs{\mflux} \nfric x
\end{equation}
needs to be solved.
This can be done numerically using Newton's method since $F$ is
strictly increasing and convex; see \cite{Gugat_et_al:2018c}.
In the same way, \eqref{eq:pressure_equation} can be solved for
$\press_\node$ if $\press$ and $\mflux$ are given.

We are interested in the pressure at the end of the pipe, \ie, for $x
= \length$, and need to find a Lipschitz constant for
\begin{equation*}
  %\label{eq:gas-lipschitz-function}
  \press_\otherNode
  =
  \press_\otherNode(\press_\node, \mflux)
  \define
  \press(\length, \press_\node, \mflux).
\end{equation*}
To this end, we make the following assumption.
\begin{assumption}
  \label{thm:assumption-gas-bounds}
We assume that the variables $\mflux$ and $\press_\node$ of each pipe
$(\node, \otherNode)$ are bounded by
\begin{equation*}
  \mflux \in [\lb{\mflux}, \ub{\mflux}] = [\lb{\mflow} / A,
  \ub{\mflow} / A],
  \quad \press_\node \in [\lb{\press}_\node, \ub{\press}_\node]
  \subset \left(\abs{\mflux} \sqrt{\specificGasConstant \temp},
  \frac{1}{\abs{\alpha}}\right)
\end{equation*}
with
\begin{equation*}
  \press_\otherNode(\press_\node, \mflux) \in \left(\abs{\mflux}
    \sqrt{\specificGasConstant \temp},
  \frac{1}{\abs{\alpha}}\right) \quad \text{for all } (\press_\node, \mflux) \in
\feasset \define [\lb{\press}_\node, \ub{\press}_\node] \times [\lb{\mflux},
\ub{\mflux}].
\end{equation*}
\end{assumption}
The following lemma guarantees the Lipschitz continuity of
$\press_\otherNode(\press_\node, \mflux)$ on~$\feasset$.
\begin{lemma}
  \label{thm:multi-dim-lipschitz}
  Let $f:\feasset\to \R$ be a partially differentiable function on a
  compact and convex subset $\feasset \subset \R^d$ with $d \in \N$.
  Then, $f$ is Lipschitz continuous on $\feasset$ with Lipschitz constant $L = 1$ \wrt\ the weighted
  $1$-norm
  \begin{equation}
    \label{eq:weighted-norm}
    \norm{\tilde{x}}_{w} \define \sum_{i = 1}^d
    \Abs{\tilde{x}_i} w_i
  \end{equation}
  with positive weights $w \in \R^d$ and $w_i \geq \max_{x \in \feasset}
  \Abs{\dparshort{x_i}{f(x)}}$ for all $i \in [d]$.
\end{lemma}
\begin{proof}
  Let $\tilde{x}, \tilde{y} \in \feasset$.
  We define the auxiliary function $g:[0,1] \to \R$ as $g(\lambda)
  \define f(\lambda \tilde{x} + (1-\lambda) \tilde{y})$.
%  Its derivative reads $g^\prime(\lambda) = p(\lambda)^\top \nabla f(p(\lambda))$.
  Now, we use the fundamental theorem of calculus to prove the claim:
  \begin{align*}
    \abs{f(\tilde{x}) - f(\tilde{y})} & = \abs{g(1) - g(0)}\\
    & = \Abs{\int_0^1 g^\prime(\lambda) \diff \lambda}\\
    & = \Abs{\int_0^1 (\tilde{x} - \tilde{y})^\top \nabla f(\lambda \tilde{x} + (1-\lambda) \tilde{y}) \diff \lambda}\\
    & = \Abs{\int_0^1 \sum_{i = 1}^d (\tilde{x}_i - \tilde{y}_i) \dparshort{x_i}{f(\lambda \tilde{x} + (1-\lambda) \tilde{y})} \diff \lambda}\\
    & \leq \int_0^1 \sum_{i = 1}^d \Abs{(\tilde{x}_i - \tilde{y}_i) \dparshort{x_i}{f(\lambda \tilde{x} + (1-\lambda) \tilde{y})}} \diff \lambda\\
    & \leq \int_0^1 \sum_{i = 1}^d \Abs{\tilde{x}_i - \tilde{y}_i} \max_{x\in \feasset} \Abs{\dparshort{x_i}{f(x)}} \diff \lambda\\
    & \leq \Norm{\tilde{x}_i - \tilde{y}_i}_{\mathrm{w}}.\qedhere
  \end{align*}
\end{proof}
Using this weighted $1$-norm allows to get tighter bounds
in~\eqref{eq:box-bounds-rhs} compared to the usual $1$-norm.
To actually compute this weighted norm one could solve
$\max_{x\in \feasset} \Abs{\dparshort{x_i}{f(x)}}$, which is an NLP
for each $i \in [d]$.
In the case of Algorithm~\ref{alg:mult-method},
the set $\feasset$ will always be a box.
For the function $\press_\otherNode(\press_\node, \mflux)$, we
give the optimal solution or at least a suitable upper bound of these
NLPs \rev{in the following two theorems}.
\begin{theorem}
  \label{thm:max-deriv-p}
  If Assumption~\ref{thm:assumption-gas-bounds} holds, the
  derivative~$\dparshort{\press_\node}{\press_\otherNode}(\press_\node,
  \mflux) > 0$ attains its maximum on a given box $\emptyset \neq
  [\press_\node^-, \press_\node^+] \times [\mflux^-, \mflux^+]
  \subseteq [\lb{\press}_\node, \ub{\press}_\node] \times
  [\lb{\mflux}, \ub{\mflux}]$
  in $(\press_\node^-, \mflux^+)$ if $\mflux^+ \geq 0$, or
  in $(\press_\node^+, \mflux^+)$ else.
\end{theorem}
\begin{theorem}
  \label{thm:max-deriv-q}
  Let $\feasset = [\press_\node^-, \press_\node^+] \times [\mflux^-,
  \mflux^+] \subseteq [\lb{\press}_\node, \ub{\press}_\node] \times
  [\lb{\mflux}, \ub{\mflux}]$ be a given box with $\feasset \neq
  \emptyset$.
  If Assumption~\ref{thm:assumption-gas-bounds} holds, the
  derivative~$\dparshort{\mflux}{\press_\otherNode}(\press_\node,
  \mflux) \leq 0$ has the following properties:
  \begin{align*}
    \min_{(\press_\node, \mflux) \in \feasset \cap \R \times \R_{\geq
    0}} \dparshort{\mflux}{\press_\otherNode}(\press_\node, \mflux)
    & = \dparshort{\mflux}{\press_\otherNode}(\press_\node^-, \mflux^+)
    & \text{if } \mflux^+ \geq 0,\\
    \min_{(\press_\node, \mflux) \in \feasset \cap \R \times \R_{\leq
    0}} \dparshort{\mflux}{\press_\otherNode}(\press_\node, \mflux)
    & > \dparshort{\mflux}{\press_\otherNode}(\press_\otherNode(\press_\node^-, \mflux^-), -\mflux^-)
    & \text{if } \mflux^- < 0.
  \end{align*}
\end{theorem}
\rev{The proofs of these two theorems can be found in the appendix of
  this paper as they are rather long and technical.}

Theorem~\ref{thm:max-deriv-p} and Theorem~\ref{thm:max-deriv-q} can be
used to not only compute the bounds in~\eqref{eq:box-bounds-rhs} for
the initial box $[\lb{\press}_\node, \ub{\press}_\node] \times
[\lb{\mflux}, \ub{\mflux}]$ but also for every new box that is created
in Step~\ref{alg:mult-method-split-box} of
Algorithm~\ref{alg:mult-method}.
This can significantly tighten the bounds in~\eqref{eq:box-bounds-rhs}
as the iteration proceeds.

%%%%%%%%%%%%%%%%%%%%%%%%%%%%%%%%%%%%%%%%%%%%%%%%%%%%%%%%%%%%%%%%%%%%%%%%%%%
\subsection{Numerical Results}
\label{sec:case-study-gas-numerical-results}
Now, we apply Algorithm~\ref{alg:mult-method} to two test problems
from the \sfname{GasLib} library~\cite{Humpola_et_al:2015}, which contains
stationary gas network benchmark instances.
To this end, we implemented our method in
\Python~3.8.10.
%\cite{Van_Rossum_Drake:2009}.
The computations were done on a machine with an \CPU with $4$~cores,
\SIrange{1.8}{4.0}{\giga\hertz}, and \SI{16}{\giga\byte}\,RAM.
The master problems~\eqref{eq:mult-problem-master} and the
subproblems~\eqref{eq:mult-problem-sub} have been modeled using
\GAMS~36.2.0~\cite{Bussieck_Meeraus:2004}.
We used the solver~\CPLEX~20.1.0.1~\cite{cplex2009v12} to solve the
master problems, which are MIPs, and the
solver~\SNOPT~7.7.7~\cite{Gill_et_al:2005} for the subproblems, which
are NLPs.

To improve the performance of our method, we detect boxes
$[\press_\node^-, \press_\node^+] \times [\mflux^-, \mflux^+]$
in~$\boxset_i^k$ that lie outside the feasible set of a pipe.
From \eqref{eq:pressure_function_deriv_p} and
\eqref{eq:pressure_function_deriv_q} we know that this is the case if
$\mflux^- \geq 0$ and $\press_\otherNode(\press_\node^+, \mflux^-) <
\lb{\press}_\otherNode$ holds or if $\mflux^+ \leq 0$ and
$\press_\otherNode(\press_\node^-, \mflux^+) > \ub{\press}_\otherNode$
holds.
Additionally, we can fix the flow in pipes that are not part of a
cycle.
This allows us to reduce the dimension of the corresponding
nonlinearities by one.
To get well-scaled problems, we model all pressure values in \si{\bar}
instead of \si{\pascal} and exclusively use mass flow values (in
\si{\kilogram \per \second}) instead of mass flux values. As a result,
we have to scale all values of
$\dparshort{\mflux}{\press_\otherNode}(\press_\node, \mflux)$ that are
used to compute the bounds in~\eqref{eq:box-bounds-rhs} by a factor of
$10^{-5}/\area$.

\rev{To compute a sufficiently large value $M$ for the master problem constraints \eqref{eq:mult-problem-master-mip-polytope-bigm-lower}--\eqref{eq:mult-problem-master-mip-lipschitz-geq} we use the difference between the values that can occur on the right-hand sides of the inequalities and the bounds of the variables on the left-hand side.
This leads to the formula
\begin{equation*}
	M \define \max\Set{\left(\ub{\press}_\node - \lb{\press}_\node\right), \left(\ub{\mflow} - \lb{\mflow}\right), \left(\frac{1}{\abs{\alpha}} - \lb{\press}_\otherNode\right), \left(\ub{\press}_\otherNode - \frac{\max\Set{\ub{\mflow}, \abs{\lb{\mflow}}}}{\area} \sqrt{\specificGasConstant \temp}\right)}.
\end{equation*}}

\begin{figure}
  \centering
  \input{tikz-imgs/gaslib-11.tikz}
  \caption{Schematic representation of \textsf{GasLib-11}.}
  \label{fig:gaslib-11}
\end{figure}
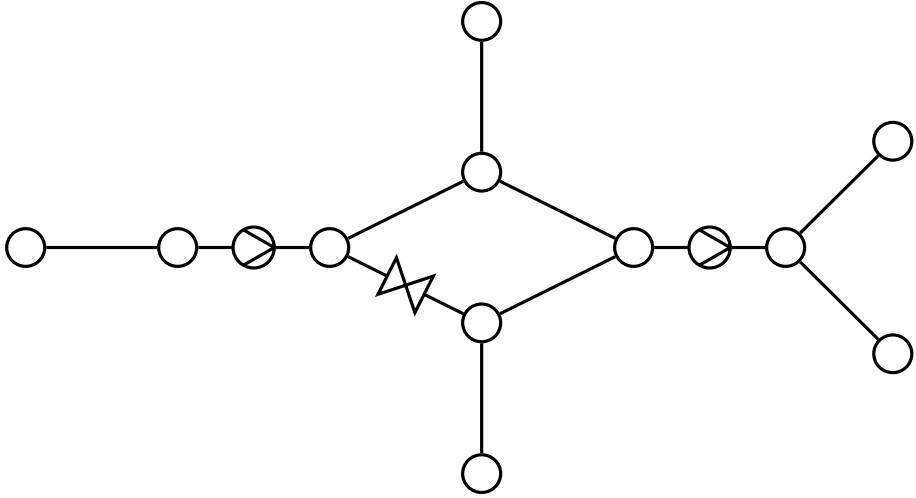

The first instance we solve is \textsf{GasLib-11}, which is shown in
Figure~\ref{fig:gaslib-11}. It contains 11 nodes including 3
entries and 3 exits, 8 pipes, 2 compressor stations, and a single
valve. Because the network has a cycle, not all flows on all arcs are
known a priori.

\begin{figure}
  \centering
  \input{tikz-imgs/gaslib-24.tikz}
  \caption{Schematic representation of \textsf{GasLib-24}.}
  \label{fig:gaslib-24}
\end{figure}
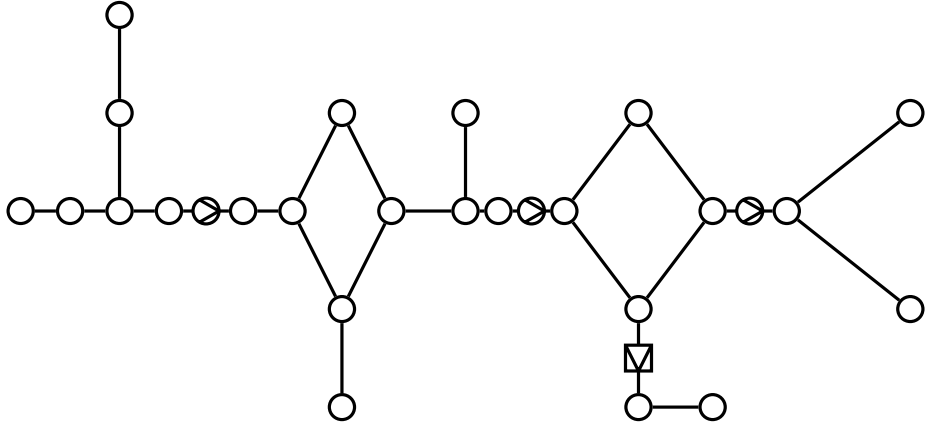

The second instance we solve is \textsf{GasLib-24}; see
Figure~\ref{fig:gaslib-24}.
It consists of $24$ nodes including 3 entries and 5 exits, $19$
pipes, 3 compressor stations, a single control valve, a single one
short pipe, and a single resistor, which we replace by a
short pipe.
There are two cycles in the network.

\begin{table}
  \centering
  \caption{The runtimes and numbers of iterations of
    Algorithm~\ref{alg:mult-method} for different parameters
    $\lambda$.
    If no overall time is given, the time limit of \SI{1000}{\second}
    was reached before an $\myeps$-feasible solution could be found.}
  \label{tab:solution-gas}
  \begin{tabular}{lrrrr}
    \toprule
    Instance & $\lambda$ & \# iterations & Mean iteration time (\si{\second}) & Overall time (\si{\second})\\
    \midrule
    \textsf{GasLib-11} & $0.125$ & $182$ & $1.37$ & $248.76$\\
    \textsf{GasLib-11} & $0.25$ & $89$ & $0.78$ & $69.45$\\
    \textsf{GasLib-11} & $0.375$ & $74$ & $0.64$ & $47.41$\\
    \textsf{GasLib-11} & $0.5$ & $123$ & $0.85$ & $104.24$\\
    \midrule
    \textsf{GasLib-24} & $0.125$ & $218$ & $4.60$ & $-$\\%1001.727407
    \textsf{GasLib-24} & $0.25$ & $95$ & $1.67$ & $158.43$\\
    \textsf{GasLib-24} & $0.375$ & $198$ & $4.53$ & $896.17$\\
    \textsf{GasLib-24} & $0.5$ & $216$ & $4.64$ & $-$\\%1001.519813
    \bottomrule
  \end{tabular}
\end{table}

For both instances we tested the values $0.125$, $0.25$, $0.375$, and
$0.5$ for the parameter $\lambda$.
We chose the value $\myeps = 0.1$ (in \si{\bar}) for the termination
criterion.
The resulting runtimes and numbers of iterations are listed in
Table~\ref{tab:solution-gas}.
In both cases, $\lambda=0.25$ and $\lambda=0.375$ yield better results
than $\lambda=0.125$ or $\lambda=0.5$.
For $\lambda=0.125$ this can be explained by the fact that the boxes
do not shrink fast enough when split because the splitting point can
be quite close to the edge of its box.
Choosing~$\lambda=0.5$, on the other hand, removes the possibility for
the splitting points to be closer to the master problem solution and
therefore potentially to the optimal solution of
Problem~\eqref{eq:mult-problem}.

For \textsf{GasLib-11} the best result was achieved for
$\lambda=0.375$ for which an $\myeps$-feasible solution was found in
\SI{47.41}{\second}.
The method terminated after $74$ iterations with a mean iteration time
of \SI{0.64}{\second}.
The mean time to solve the master problem was \SI{0.35}{\second}
and the mean time to solve the subproblems was \SI{0.14}{\second}.
The remaining \SI{0.16}{\second} were used to (re-)build the model in
each iteration.

For \textsf{GasLib-24} the best result was achieved for $\lambda=0.25$
for which an $\myeps$-feasible solution was found in
\SI{158.43}{\second}.
With $95$ iterations, the mean iteration time was
\SI{1.67}{\second}. This mean includes \SI{1.05}{\second} to solve the
master problem, \SI{0.23}{\second} to solve the subproblems, and
\SI{0.39}{\second} to (re-)build the model in each iteration.

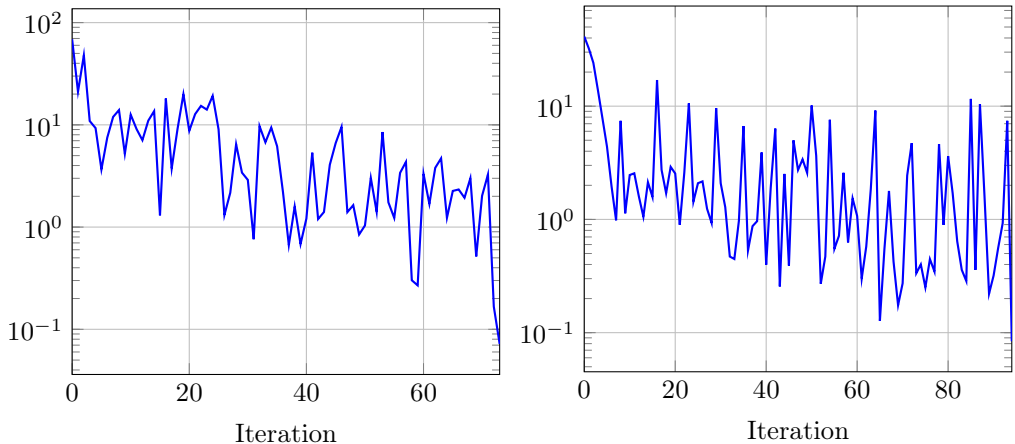
\begin{figure}
  \centering
  \begin{minipage}{.49\textwidth}
    \input{tikz-imgs/max-error-gaslib-11.tikz}
  \end{minipage}
  \begin{minipage}{.49\textwidth}
    \input{tikz-imgs/max-error-gaslib-24.tikz}
  \end{minipage}
  \vspace*{-6mm}
  \caption{The maximal error $\max_{i \in [p]} \abs{f_i(x_{\multind})
      - x_{\rind}}$ in each iteration for \textsf{GasLib-11} with $\lambda=0.375$ (left)
    and \textsf{GasLib-24} with $\lambda=0.25$ (right) on a logarithmic scale.}
  \label{fig:gaslib-iterations}
\end{figure}

Figure~\ref{fig:gaslib-iterations} shows the progress of the
maximal error over the course of the iterations for the best-case of
both test instances.
One can see that it falls rapidly in the first iterations.
After that, it fluctuates strongly with only a slight downward trend
until the threshold $\max_{i \in [p]} \abs{f_i(x_{\multind}) -
  x_{\rind}} \leq \varepsilon$ is reached.
This behavior is typical as the approximation progress from splitting
a box in Step~\ref{alg:mult-method-split-box} of
Algorithm~\ref{alg:mult-method} is reduced as the boxes get smaller.
The fluctuations can be explained by the master problem solution
switching to a bigger box $j_i^k \in J_i^k$ after the previous box has
been refined a number of times in
Step~\ref{alg:mult-method-split-box}.

%%% Local Variables:
%%% mode: latex
%%% TeX-master: "multidim-lipschitz-minlp-preprint"
%%% End:

%% file: tikz-imgs/gaslib-11.tikz
\begin{tikzpicture}[scale=1]
  \def \radius {0.5cm}
  \def \specialRadius {0.5cm}
  \def \lineThickness {very thick}
  \def \locationFactor {0.01}

  \def \nodeLocations {0.0*\locationFactor/0.0*\locationFactor/entry01,
    200.0*\locationFactor/0.0*\locationFactor/entry03,
    600.0*\locationFactor/-300.0*\locationFactor/entry02,
    600.0*\locationFactor/300.0*\locationFactor/exit01,
    1141.0*\locationFactor/141.0*\locationFactor/exit02,
    1141.0*\locationFactor/-141.0*\locationFactor/exit03,
    400.0*\locationFactor/0.0*\locationFactor/N01,
    600.0*\locationFactor/100.0*\locationFactor/N02,
    600.0*\locationFactor/-100.0*\locationFactor/N03,
    800.0*\locationFactor/0.0*\locationFactor/N04,
    1000.0*\locationFactor/0.0*\locationFactor/N05}
  \foreach \x/\y/\name in \nodeLocations
  \node[circle, draw=black, minimum size=\radius, \lineThickness](\name) at (\x,\y){ };

  \def \pipes {entry01/entry03,
    N01/N02,
    entry02/N03,
    N02/exit01,
    N02/N04,
    N03/N04,
    N05/exit02,
    N05/exit03}
  \foreach \source/\dest in \pipes
  \draw [\lineThickness] (\source) -- (\dest);

  \def \compressorStations {entry03/N01/300.0*\locationFactor/0.0*\locationFactor/0.0/CS01entry03N01,
    N04/N05/900.0*\locationFactor/0.0*\locationFactor/0.0/CS02N04N05}
  \foreach \source/\dest/\x/\y/\rotation/\name in \compressorStations
  {
    \node[compressor, draw=black, rotate=\rotation, minimum size=\specialRadius, \lineThickness](\name) at (\x,\y){ };
    \draw [\lineThickness] (\source) -- (\name) -- (\dest);
  }
  \def \valves {N01/N03/500.0*\locationFactor/-50.0*\locationFactor/-26.56505117707799/V01N01N03}
  \foreach \source/\dest/\x/\y/\rotation/\name in \valves
  {
    \node[valve, draw=black, rotate=\rotation, minimum size=\specialRadius, \lineThickness](\name) at (\x,\y){ };
    \draw [\lineThickness] (\source) -- (\name) -- (\dest);
  }
\end{tikzpicture}%
%
%
%%% Local Variables:
%%% mode: latex
%%% TeX-master: "../multidim-lipschitz-minlp-preprint"
%%% End:

%% file: tikz-imgs/gaslib-24.tikz
\begin{tikzpicture}[scale=1]
  \def \radius {0.3cm}
  \def \specialRadius {0.3cm}
  \def \lineThickness {very thick}
  \def \locationFactor {0.013}

  \def \nodeLocations {-100.0*\locationFactor/500.0*\locationFactor/entry03,
    -200.0*\locationFactor/300.0*\locationFactor/entry01,
    -100.0*\locationFactor/400.0*\locationFactor/entry02,
    250.0*\locationFactor/400.0*\locationFactor/exit01,
    500.0*\locationFactor/100.0*\locationFactor/exit02,
    700.0*\locationFactor/400.0*\locationFactor/exit03,
    700.0*\locationFactor/200.0*\locationFactor/exit04,
    125.0*\locationFactor/100.0*\locationFactor/exit05,
    -150.0*\locationFactor/300.0*\locationFactor/N101,
    -100.0*\locationFactor/300.0*\locationFactor/N01,
    -50.0*\locationFactor/300.0*\locationFactor/N04,
    25.0*\locationFactor/300.0*\locationFactor/N05,
    75.0*\locationFactor/300.0*\locationFactor/N05a,
    175.0*\locationFactor/300.0*\locationFactor/N05b,
    125.0*\locationFactor/400.0*\locationFactor/N05c,
    125.0*\locationFactor/200.0*\locationFactor/N05d,
    250.0*\locationFactor/300.0*\locationFactor/N06,
    283.3*\locationFactor/300.0*\locationFactor/N08,
    350.0*\locationFactor/300.0*\locationFactor/N09,
    425.0*\locationFactor/400.0*\locationFactor/N10,
    425.0*\locationFactor/200.0*\locationFactor/N11,
    425.0*\locationFactor/100.0*\locationFactor/N12,
    500.0*\locationFactor/300.0*\locationFactor/N13,
    575.0*\locationFactor/300.0*\locationFactor/N16}
  \foreach \x/\y/\name in \nodeLocations
  \node[circle, draw=black, minimum size=\radius, \lineThickness](\name) at (\x,\y){ };

  \def \pipes {entry03/entry02,
    N05d/exit05,
    entry01/N101,
    N01/N04,
    N05/N05a,
    N05a/N05c,
    N05c/N05b,
    N05a/N05d,
    N05d/N05b,
    N05b/N06,
    N06/exit01,
    N06/N08,
    N09/N10,
    N10/N13,
    N09/N11,
    N11/N13,
    N12/exit02,
    N16/exit03,
    N16/exit04}
  \foreach \source/\dest in \pipes
  \draw [\lineThickness] (\source) -- (\dest);

  \def \shortPipes {entry02/N01,
    N101/N01}
  \foreach \source/\dest in \shortPipes
  \draw [\lineThickness] (\source) -- (\dest);

  \def \compressorStations {N04/N05/-12.5*\locationFactor/300.0*\locationFactor/0.0/CS1,
    N08/N09/316.7*\locationFactor/300.0*\locationFactor/0.0/CS2,
    N13/N16/537.5*\locationFactor/300.0*\locationFactor/0.0/CS3}
  \foreach \source/\dest/\x/\y/\rotation/\name in \compressorStations
  {
    \node[compressor, draw=black, rotate=\rotation, minimum size=\specialRadius, \lineThickness](\name) at (\x,\y){ };
    \draw [\lineThickness] (\source) -- (\name) -- (\dest);
  }
  \def \controlValves {N11/N12/425.0*\locationFactor/150.0*\locationFactor/-90.0/CV01}
  \foreach \source/\dest/\x/\y/\rotation/\name in \controlValves
  {
    \node[controlvalve, draw=black, rotate=\rotation, minimum size=\specialRadius, \lineThickness](\name) at (\x,\y){ };
    \draw [\lineThickness] (\source) -- (\name) -- (\dest);
  }
\end{tikzpicture}%
%
%
%%% Local Variables:
%%% mode: latex
%%% TeX-master: "../multidim-lipschitz-minlp-preprint"
%%% End:

%% file: tikz-imgs/max-error-gaslib-11.tikz
\begin{tikzpicture}
  \def \thickness {thick}
  \begin{axis}[
    scale only axis,
    ymode=log,
    xmin=0, xmax=73,
    width=0.85\textwidth,
    xlabel=Iteration,
    % ylabel=Max error,
    ylabel near ticks,
    grid=major,
    legend pos=north east]
    \addplot [\thickness, blue] table[x=iteration, y=max_error] {tikz-imgs/gaslib-11.dat};
    % \addlegendentry{Max error}
  \end{axis}
\end{tikzpicture}

%% file: tikz-imgs/max-error-gaslib-24.tikz
    \begin{tikzpicture}
      \def \thickness {thick}
      \begin{axis}[
        scale only axis,
        ymode=log,
        xmin=0, xmax=94,
        width=0.85\textwidth,
        xlabel=Iteration,
        % ylabel=Max error,
        ylabel near ticks,
        grid=major,
        legend pos=north east]
        \addplot [\thickness, blue] table[x=iteration, y=max_error] {tikz-imgs/gaslib-24.dat};
        % \addlegendentry{Max error}
      \end{axis}
    \end{tikzpicture}

%% file: conclusion.tex
\section{Conclusion}
\label{sec:conclusion}

In this paper we developed a successive linear relaxation method for
solving MINLPs with nonlinearities that are not given in closed form.
Instead, we only assume that these multivariate nonlinearities can be
evaluated and that we know their global Lipschitz constants.
We illustrate the flexibility of this class of models and of our
method by showing that it can be applied to bilevel optimization
models with nonlinear and nonconvex lower-level problems as well as to
nonconvex MINLPs for gas transport problems that are constrained by
differential equations.
Moreover, we proved finite termination of the method \rev{(at
  approximate global optimal solutions)} and derived a
worst-case iteration bound.

Finally, let us sketch \rev{three} important topics for future research that
are out of scope of this paper.
First, one can weaken the assumptions made in this paper as it is done
in~\cite{Schmidt_et_al:2019}.
In particular, the assumption of exact function evaluations is rather
strong in some applications such as, \eg, in the case of PDE
constraints.
\rev{Second, one could also incorporate general-purpose local solvers
  to compute feasible points that lead to upper bounds.
  Together with lower bounds that we directly get from solving the
  master problem in every iteration, this could be used to obtain a
  gap and thus a further termination criterion besides the one based
  on the approximation accuracy that we used in this paper.}
\rev{Third}, it is obvious (and also not expected) that the proposed method
is not competitive with MINLP solution methods that explicitly use the
structural properties of the nonlinearities that are given in closed
form.
However, there are many possible ways of improving the general method
proposed in this paper.
For instance, the incorporation of presolving and dimension-reduction
techniques would help a lot to improve the performance of the method.

%%% Local Variables:
%%% mode: latex
%%% TeX-master: "multidim-lipschitz-minlp-preprint"
%%% End:

%% file: acknowledgements.tex
\section*{Acknowledgments}
\label{sec:acknowledgements}

The first author gratefully acknowledges the use of the
services and facilities of the Energie Campus Nürnberg 
and financial support of the state of Bavaria.
All authors thank the DFG for their support within projects A05,
B08, and C08 in CRC TRR 154.

%%% Local Variables:
%%% mode: latex
%%% TeX-master: "multidim-lipschitz-minlp-preprint"
%%% End:

%% file: appendix.tex
\section{Proofs of Theorem~\ref{thm:max-deriv-p} and
  Theorem~\ref{thm:max-deriv-q}}
\label{sec:appendix}

\rev{In this appendix, we present the proof for
  Theorem~\ref{thm:max-deriv-p} and Theorem~\ref{thm:max-deriv-q}. To
  this end, we first need to analyze the first- and second-order
  derivatives of the pressure function~$\press(x, \press_\node,
  \mflux)$. The first-order derivatives are considered in
  Lemma~\ref{thm:press_first_derivatives}, the second-order
  derivatives $\partial_{\press_\node}^2\press$ and
  $\dparshort{\mflux}{\dparshort{\press_\node}{\press}}$ are examined
  in Lemma~\ref{thm:pressure_function_deriv_pp}
  and~\ref{thm:pressure_function_deriv_pq_sign}, and, finally, the
  second-order derivative $\partial_{\mflux}^2\press$ is investigated
  in Lemma~\ref{thm:pressure_function_deriv_qq},
  \ref{thm:deriv-q-pos-flow}, and~\ref{thm:deriv-q-pos-flow-deriv-q}.}

\rev{We start with the following properties of the first-order
  derivatives from \cite{Gugat_et_al:2018c}.}
\begin{lemma}
	\label{thm:press_first_derivatives}
	The function $\press$ as defined in \eqref{eq:pressure_function} is
	differentiable for $\press \in (\abs{\mflux} \sqrt{\specificGasConstant
		\temp}, 1 / \abs{\alpha})$ with
	\begin{align}
		\label{eq:pressure_function_deriv_p}
		\dpar{\press_\node}{\press}(x, \press_\node, \mflux)
		&= \frac{F'(\press_\node)}{F'(\press(x, \press_\node, \mflux))} > 0,\\
		\label{eq:pressure_function_deriv_q}
		\dpar{\mflux}{\press}(x, \press_\node, \mflux)
		&= \frac{2\mflux \specificGasConstant \temp \ln\left(\frac{(\pubu)
				\press(x, \press_\node, \mflux)}{(\pubaux{\press(x,
					\press_\node, \mflux)}) \press_\node}\right)
			- \specificGasConstant \temp \abs{\mflux} \nfric x}{F'(\press(x,
			\press_\node, \mflux))}
		\begin{cases}
			= 0, & \text{for } \mflux = 0,\\
			< 0, & \text{for } \mflux \neq 0,
		\end{cases}
	\end{align}
	and
	\begin{equation}
		\label{eq:pressure_function_deriv_x}
		\sign\left(\dpar{x}{\press}(x, \press_\node, \mflux)\right)
		= - \sign(\mflux).
	\end{equation}
\end{lemma}
The sign condition in~\eqref{eq:pressure_function_deriv_x} implies
\begin{align*}
	\press_\node & > \press(x, \press_\node, \mflux) \quad \text{if } \mflux > 0,\\
	\press_\node & < \press(x, \press_\node, \mflux) \quad \text{if } \mflux < 0,\\
	\press_\node & = \press(x, \press_\node, \mflux) \quad \text{if } \mflux = 0
\end{align*}
for $x \in (0, \length]$.
This can be written as
\begin{equation*}
	\sign(\mflux) = \sign(\press_\node - \press(x, \press_\node, \mflux)).
\end{equation*}

From \cite{Gugat_et_al:2018b} we know that the second derivative \wrt
$\press_\node$ of $\press$ is given by
\begin{equation}
	\label{eq:pressure_function_deriv_pp}
	\begin{aligned}
		\frac{\partial^2 \press}{\partial \press_\node^2}
		={} & \frac{(\press_\node^2 + \specificGasConstant \temp
			\mflux^2(1 + 2\alpha \press_\node)) (\press^2 -
			\specificGasConstant \temp \mflux^2)^2}{F'(\press)
			\press_\node^2 (1 + \alpha \press_\node)^2 (\press^2 - \mflux^2
			\specificGasConstant \temp)^2}
		\\
		& \quad - \frac{(\press^2 + \specificGasConstant \temp \mflux^2(1
			+ 2 \alpha \press)) (\press_\node^2 - \specificGasConstant \temp
			\mflux^2)^2}{F'(\press) \press_\node^2 (1 + \alpha
			\press_\node)^2 (\press^2 - \mflux^2 \specificGasConstant
			\temp)^2}.
	\end{aligned}
\end{equation}
The sign of this derivative is given by the following lemma.
\begin{lemma}
	\label{thm:pressure_function_deriv_pp}
	It holds
	\begin{equation*}
		\label{eq:pressure_function_deriv_pp_sign}
		\sign\left(\frac{\partial^2 \press}{\partial \press_\node^2}\right)
		= - \sign(\mflux)
	\end{equation*}
	for $\press \in (\abs{\mflux} \sqrt{\specificGasConstant \temp},
	1 / \abs{\alpha})$ and $x \in (0, \length]$.
\end{lemma}
\begin{proof}
	Because of
	\eqref{eq:press_upper_bound}, \eqref{eq:press_lower_bound}, and
	\eqref{eq:pressure_function_F_deriv_p}, the denominator in
	\eqref{eq:pressure_function_deriv_pp} is positive.
	Thus, only the numerator
	\begin{equation}
		\label{eq:pressure_function_deriv_pp_numerator}
		(\press_\node^2 + \specificGasConstant \temp \mflux^2(1 + 2\alpha
		\press_\node)) (\press^2 - \specificGasConstant \temp \mflux^2)^2
		- (\press^2 + \specificGasConstant \temp \mflux^2(1 + 2 \alpha
		\press)) (\press_\node^2 - \specificGasConstant \temp \mflux^2)^2
	\end{equation}
	determines the sign of $\frac{\partial^2 \press}{\partial
		\press_\node^2}$.
	We note that $(\press_\node^2 + \specificGasConstant \temp
	\mflux^2(1 + 2\alpha \press_\node))$ as well as $(\press^2 +
	\specificGasConstant \temp \mflux^2(1 + 2 \alpha \press))$ are
	positive:
	\begin{subequations}
		\label{eq:two_alpha_press}
		\begin{align}
			\label{eq:two_alpha_press_0}
			\press_\node^2 + \specificGasConstant \temp \mflux^2(1 + 2\alpha
			\press_\node)
			& = (\press_\node^2 - \specificGasConstant \temp \mflux^2) + 2
			\specificGasConstant \temp \mflux^2 (1 + \alpha  \press_\node) >
			0,
			\\
			\label{eq:two_alpha_press_1}
			\press^2 + \specificGasConstant \temp \mflux^2 (1 + 2 \alpha
			\press)
			& = (\press^2 - \specificGasConstant \temp \mflux^2) + 2
			\specificGasConstant \temp \mflux^2 (1 +  \alpha \press) > 0.
		\end{align}
	\end{subequations}
	It follows that both the first and the second term
	in~\eqref{eq:pressure_function_deriv_pp_numerator} are positive.
	We rewrite them as
	\begin{align*}
		(\press_\node^2 + \specificGasConstant \temp \mflux^2 (1 + 2\alpha
		\press_\node)) (\press^2 - \specificGasConstant \temp \mflux^2)^2
		& = \press_\node^2 \press^4 \left(1 + \specificGasConstant \temp
		\mflux^2 \frac{1 + 2\alpha \press_\node}{\press_\node^2}\right)
		\left(1 - \mach\right)^2,\\
		(\press^2 + \specificGasConstant \temp \mflux^2(1 + 2 \alpha
		\press)) (\press_\node^2 - \specificGasConstant \temp \mflux^2)^2
		& = \press_\node^4 \press^2 \left(1 + \specificGasConstant \temp
		\mflux^2 \frac{1 + 2\alpha \press}{\press^2}\right) \left(1 -
		\mach_\node\right)^2
	\end{align*}
	using the squared mach numbers
	\begin{equation*}
		\mach = \specificGasConstant \temp \frac{\mflux^2}{\press^2},
		\quad \mach_\node = \specificGasConstant \temp
		\frac{\mflux^2}{\press_\node^2}.
	\end{equation*}
	Now, if the sign of $\mflux$ is given, all terms are easily
	comparable except for
	\begin{equation*}
		1 + \specificGasConstant \temp \mflux^2
		\frac{1 + 2\alpha \press_\node}{\press_\node^2}
		\quad \text{and} \quad
		1 + \specificGasConstant \temp \mflux^2 \frac{1 + 2\alpha
			\press}{\press^2}.
	\end{equation*}
	We take the derivative \wrt\ the pressure and get
	\begin{align*}
		\dpar{\press}{} \left(1 + \specificGasConstant \temp \mflux^2
		\frac{1 + 2\alpha \press}{\press^2}\right)
		& = \specificGasConstant \temp \mflux^2 \dpar{\press}{}\left(\frac{1
			+ 2 \alpha \press}{\press^2}\right)
		= \specificGasConstant \temp \mflux^2 \frac{- 2 \press - 2 \alpha
			\press^2}{\press^4}
		\\
		& = - 2 \specificGasConstant \temp \mflux^2 \frac{\press (1 + \alpha
			\press)}{\press^4}
		< 0.
	\end{align*}
	We can now determine the sign of $\frac{\partial^2 \press}{\partial
		\press_\node^2}$ in dependence of $\mflux$.
	If $\mflux = 0$, then $\press = \press_\node$ and, therefore,
	\begin{equation*}
		\frac{\partial^2 \press}{\partial \press_\node^2} = 0
	\end{equation*}
	holds.
	If $\mflux > 0$, then $\press_\node > \press$ and $\mach >
	\mach_\node$.
	% Due to $\alpha < 0$.
	We thus have
	\begin{equation*}
		\frac{\partial^2 \press}{\partial \press_\node^2} < 0.
	\end{equation*}
	If $\mflux < 0$, we analogously get
	\begin{equation*}
		\frac{\partial^2 \press}{\partial \press_\node^2} > 0. \qedhere
	\end{equation*}
\end{proof}

Next, we analyze the mixed second-order derivative of $\press$.
\begin{lemma}
	\label{thm:pressure_function_deriv_pq_sign}
	It holds
	\begin{equation*}
		%\label{eq:pressure_function_deriv_pq_sign}
		\frac{\partial^2 \press}{\partial \mflux \partial \press_\node}
		\begin{cases}
			= 0, & \text{for } \mflux = 0,\\
			> 0, & \text{for } \mflux \neq 0
		\end{cases}
	\end{equation*}
	for $\press \in (\abs{\mflux} \sqrt{\specificGasConstant \temp},
	1 / \abs{\alpha})$ and $x \in (0, \length]$.
\end{lemma}
\begin{proof}
	It holds
	\begingroup
	\allowdisplaybreaks
	\begin{align*}
		&\!\!\!\!\!\!\!\!\!\! \frac{\partial^2 \press}{\partial \mflux \partial \press_\node}
		=
		\dpar{\mflux}{}\left(\frac{F'(\press_\node)}{F'(\press)}\right)
		= \dpar{\mflux}{}\left(\frac{(\press_\node^2 - \mflux^2
			\specificGasConstant \temp) \press (1 + \alpha \press)}{(\press^2 -
			\mflux^2 \specificGasConstant \temp) \press_\node (1 + \alpha
			\press_\node)}\right)
		\\
		= \ & \frac{\left(-2 \mflux \specificGasConstant \temp \press (1 +
			\alpha \press) + (\press_\node^2 - \mflux^2
			\specificGasConstant \temp) \left((1 + \alpha \press)
			\dpar{\mflux}{\press} + \press \alpha
			\dpar{\mflux}{\press}\right)\right) (\press^2 - \mflux^2
			\specificGasConstant \temp) \press_\node (1 + \alpha
			\press_\node)}{(\press^2 - \mflux^2 \specificGasConstant
			\temp)^2 \press_\node^2 (1 + \alpha \press_\node)^2}
		\\
		& \quad -\, \frac{(\plbu) \press (\pub) \left(2 \press
			\dpar{\mflux}{\press} - 2 \mflux \RsT\right) \press_\node
			(\pubu)}{(\press^2 - \mflux^2 \specificGasConstant \temp)^2
			\press_\node^2 (1 + \alpha \press_\node)^2}
		\\
		= \ & \frac{\left(-2 \mflux \specificGasConstant \temp \press (1 +
			\alpha \press) + (\press_\node^2 - \mflux^2
			\specificGasConstant \temp) (1 + 2 \alpha \press)
			\dpar{\mflux}{\press}\right) (\press^2 - \mflux^2
			\specificGasConstant \temp)}{(\press^2 - \mflux^2
			\specificGasConstant \temp)^2 \press_\node (1 + \alpha
			\press_\node)}
		\\
		& \quad -\, \frac{2 (\plbu) \press^2 (\pub) \dpar{\mflux}{\press} -
			2 \mflux \RsT (\plbu) \press (\pub)}{(\press^2 - \mflux^2
			\specificGasConstant \temp)^2 \press_\node (1 + \alpha
			\press_\node)}
		\\
		= \ & \frac{2 \mflux \RsT \press (\pub) \left((\plbu) -
			(\plb)\right)}{(\press^2 - \mflux^2 \specificGasConstant
			\temp)^2 \press_\node (1 + \alpha \press_\node)}
		\\
		& \quad + \frac{(\plbu) \left((1 + 2 \alpha \press) (\plb) - 2
			\press^2 (\pub)\right) \dpar{\mflux}{\press}}{(\press^2 - \mflux^2
			\specificGasConstant \temp)^2 \press_\node (1 + \alpha
			\press_\node)}
		\\
		= \ & \frac{2 \mflux \specificGasConstant \temp \press (1 + \alpha
			\press) (\press_\node^2 - \press^2) - (\press_\node^2 -
			\mflux^2 \specificGasConstant \temp) (\press^2 + \mflux^2
			\specificGasConstant \temp (1 + 2 \alpha \press))
			\dpar{\mflux}{\press}}{(\press^2 - \mflux^2 \specificGasConstant
			\temp)^2 \press_\node (1 + \alpha \press_\node)}.
	\end{align*}
	\endgroup
	We note that $(\press^2 + \mflux^2 \specificGasConstant \temp (1 + 2
	\alpha \press)) > 0$ holds because of \eqref{eq:two_alpha_press_1}.
	Now, if $\mflux = 0$, then $\press_\node = \press$ and
	$\dpar{\mflux}{\press} = 0$ holds.
	It follows
	\begin{equation*}
		\frac{\partial^2 \press}{\partial \mflux \partial \press_\node} =
		0.
	\end{equation*}
	If $\mflux \neq 0$, then $\mflux (\press_\node^2 - \press^2) > 0$
	as well as $\dpar{\mflux}{\press} < 0$ holds and we get
	\begin{equation*}
		\frac{\partial^2 \press}{\partial \mflux \partial \press_\node} >
		0. \qedhere
	\end{equation*}
\end{proof}
From $\dpar{\mflux}{\press} = 0$ for $\mflux = 0$ and by using
Schwarz's theorem for $\mflux \neq 0$, it follows
\begin{equation}
	\label{eq:pressure_function_deriv_qp_sign}
	\frac{\partial^2 \press}{\partial \press_\node \partial \mflux}
	= \frac{\partial^2 \press}{\partial \mflux \partial \press_\node}
	\begin{cases}
		= 0, & \text{for } \mflux = 0,\\
		> 0, & \text{for } \mflux \neq 0.
	\end{cases}
\end{equation}

For the second-order derivative \wrt\ $\mflux$ of $\press$,
we can only give the sign for positive mass flux~$\mflux$.
\begin{lemma}
	\label{thm:pressure_function_deriv_qq}
	It holds\\
	\begin{equation*}
		%\label{eq:pressure_function_deriv_qq_sign}
		\frac{\partial^2 \press}{\partial \mflux^2} < 0
	\end{equation*}
	for $\press \in (\abs{\mflux} \sqrt{\specificGasConstant \temp},
	1 / \abs{\alpha})$, $x \in (0, \length]$, and $\mflux > 0$.
\end{lemma}
\begin{proof}
	We introduce the auxiliary function
	\begin{equation}
		\label{eq:massflux-aux}
		\mfluxaux \define \mfluxaux(x, \press_\node, \mflux) \define
		\specificGasConstant \temp \ln\left(\frac{(1 + \alpha \press)
			\press_\node}{(1 + \alpha \press_\node) \press}\right) +
		\frac{1}{2} \specificGasConstant \temp \sign(\mflux) \nfric x
	\end{equation}
	to rewrite the derivative of $\press$ \wrt\ $\mflux$ as
	\begin{equation*}
		\dpar{\mflux}{\press} = \frac{-2 \mflux \mfluxaux \press (1 + \alpha
			\press)}{\press^2 - \mflux^2 \specificGasConstant \temp}.
	\end{equation*}
	It holds
	\begin{equation*}
		\sign(\mfluxaux) = \sign(\press_\node - \press) = \sign(\mflux).
	\end{equation*}
	We can take the derivative of $\mfluxaux$ \wrt\ $\mflux$ for $\mflux
	\neq 0$ and get
	\begin{align*}
		\dpar{\mflux}{\mfluxaux}(x, \press_\node, \mflux)
		& = \specificGasConstant \temp \frac{(1 + \alpha \press_\node)
			\press}{(1 + \alpha \press) \press_\node} \frac{\alpha
			\press_\node \dpar{\mflux}{\press} (1 + \alpha \press_\node)
			\press - (1 + \alpha \press) \press_\node (1 + \alpha
			\press_\node) \dpar{\mflux}{\press}}{(1 + \alpha \press_\node)^2
			\press^2}\\
		& = \frac{- \specificGasConstant \temp}{(1 + \alpha \press) \press} \dpar{\mflux}{\press}
		= \frac{2 \mflux \specificGasConstant \temp \mfluxaux(x,
			\press_\node, \mflux)}{\press^2 - \mflux^2 \specificGasConstant
			\temp}
		> 0.
	\end{align*}
	Now, we can take the second derivative of $\press$ \wrt $\mflux$ and get
	\begin{align*}
		&\!\!\!\!\!\!\!\!\!\!\frac{\partial^2 \press}{\partial \mflux^2}
		= \dpar{\mflux}{}\left(\frac{-2 \mflux \mfluxaux \press (1 +
			\alpha \press)}{\press^2 - \mflux^2 \specificGasConstant
			\temp}\right)
		\\
		= \ & \frac{-2 \left(\mfluxaux \press (1 + \alpha \press) + \mflux
			\press (1 + \alpha \press) \dpar{\mflux}{\mfluxaux} + \mflux
			\mfluxaux (1 + 2 \alpha \press) \dpar{\mflux}{\press}\right)
			(\press^2 - \mflux^2 \specificGasConstant \temp)}{(\press^2 -
			\mflux^2 \specificGasConstant \temp)^2}
		\\
		& \quad + \frac{2 \mflux \mfluxaux \press (1 + \alpha \press) (2
			\press \dpar{\mflux}{\press} - 2 \mflux \specificGasConstant
			\temp)}{(\press^2 - \mflux^2 \specificGasConstant \temp)^2}
		\\
		= \ & \frac{-2 \mfluxaux \left(\left(\press (1 + \alpha \press) +
			\mflux (1 + 2 \alpha \press) \dpar{\mflux}{\press}\right)
			(\press^2 - \mflux^2 \specificGasConstant \temp) - \mflux
			\press (1 + \alpha \press) (2 \press \dpar{\mflux}{\press} - 2
			\mflux \specificGasConstant \temp)\right)}{(\press^2 -
			\mflux^2 \specificGasConstant \temp)^2}
		\\
		& \quad -\, \frac{2 \mflux \press (1 + \alpha \press)
			\dpar{\mflux}{\mfluxaux} (\press^2 - \mflux^2 \specificGasConstant
			\temp)}{(\press^2 - \mflux^2 \specificGasConstant \temp)^2}
		\\
		= \ & \frac{-2 \mfluxaux \left(\press (1 + \alpha \press) (\press^2
			+ \mflux^2 \specificGasConstant \temp) - \mflux (\press^2 +
			\mflux^2 \RsT (1 + 2 \alpha \press))
			\dpar{\mflux}{\press}\right)}{(\press^2 - \mflux^2
			\specificGasConstant \temp)^2}
		\\
		& \quad -\, \frac{2 \mflux \press (1 + \alpha \press)
			\dpar{\mflux}{\mfluxaux}}{\press^2 - \mflux^2 \specificGasConstant
			\temp}.
	\end{align*}
	It holds $\sign(\mfluxaux) = \sign(\mflux)$ and we have $(\press^2 +
	\mflux^2 \specificGasConstant \temp (1 + 2 \alpha \press)) > 0$
	because of \eqref{eq:two_alpha_press_1}.
	For $\mflux > 0$, we get $\dpar{\mflux}{\press} < 0$ and
	$\dpar{\mflux}{\mfluxaux} > 0$.
	The claim then follows.
\end{proof}
To compensate the fact that the last result only holds for $\mflux >
0$, we estimate the value of $\dpar{\mflux}{\press}$ for $\mflux < 0$
with values for $\mflux > 0$.
\begin{lemma}
	\label{thm:deriv-q-pos-flow}
	It holds
	\begin{equation*}
		\dpar{\mflux}{\press}(x, \press_\node, \mflux) >
		\dpar{\mflux}{\press}(x, \press(x, \press_\node, \mflux), -\mflux)
	\end{equation*}
	for $\press \in (\abs{\mflux} \sqrt{\specificGasConstant \temp},
	1/\abs{\alpha})$, $x \in (0, \length]$, and $\mflux < 0$.
\end{lemma}
\begin{proof}
	Equation \eqref{eq:pressure_equation}, \ie,
	\begin{equation*}
		F(\press) = F(\press_\node) - \frac{1}{2}
		\specificGasConstant \temp \mflux \abs{\mflux} \nfric x,
	\end{equation*}
	is equivalent to
	\begin{equation*}
		F(\press_\node) = F(\press) - \frac{1}{2}
		\specificGasConstant \temp (-\mflux) \abs{-\mflux} \nfric x.
	\end{equation*}
	Therefore, if $(\press_\node, \mflux, \press)$ fulfills
	\eqref{eq:pressure_equation}, then $(\press, -\mflux, \press_\node)$
	fulfills \eqref{eq:pressure_equation} as well.
	This means, we can write the pressure $\press_\node$ as a function of
	$\press$ and $\mflux$:
	\begin{equation*}
		\label{eq:pressure_function_u}
		\press_\node(x, \press, \mflux) = \press(x, \press, -\mflux).
	\end{equation*}
	Now, we have
	\begin{equation*}
		\label{eq:pressure_uv_deriv_q}
		\begin{aligned}
			\dpar{\mflux}{\press}(x, \press_\node, \mflux)
			& = \frac{2\mflux \specificGasConstant \temp
				\ln\left(\frac{(\pubu) \press}{(\pub) \press_\node}\right) -
				\specificGasConstant \temp \abs{\mflux} \nfric x}{F'(\press)}
			\\
			& = \frac{2(-\mflux) \specificGasConstant \temp
				\ln\left(\frac{(\pub) \press_\node}{(\pubu) \press}\right) -
				\specificGasConstant \temp \abs{-\mflux} \nfric x}{F'(\press)}
			\\
			& = \frac{F'(\press_\node)}{F'(\press)}
			\dpar{\mflux}{\press_\node}(x, \press, -\mflux)
			\\
			& > \dpar{\mflux}{\press_\node}(x, \press, -\mflux)
		\end{aligned}
	\end{equation*}
	because $\press_\node < \press$ holds for $\mflux < 0$, $F'$ is strictly
	increasing in $\press$, and $\dpar{\mflux}{\press_\node}(x, \press,
	-\mflux) < 0$ for $\mflux \neq 0$.
\end{proof}
Now, for $\dpar{\mflux}{\press}(x, \press(x, \press_\node, \mflux),
-\mflux)$, we can determine the sign of its derivative \wrt~$\mflux$
for $\mflux < 0$.
\begin{lemma}
	\label{thm:deriv-q-pos-flow-deriv-q}
	It holds
	\begin{equation*}
		\frac{\partial^2 \press}{\partial \mflux^2}(x, \press(x,
		\press_\node, \mflux), -\mflux) > 0
	\end{equation*}
	for $\press \in (\abs{\mflux} \sqrt{\specificGasConstant \temp},
	1 / \abs{\alpha})$, $x \in (0, \length]$, and $\mflux < 0$.
\end{lemma}
\begin{proof}
	We know that
	\begin{equation*}
		\press_\node = \press(x, \press(x, \press_\node, \mflux), -\mflux)
	\end{equation*}
	holds.
	Therefore, we can write
	\begin{align*}
		\dpar{\mflux}{\press}(x, \press(x, \press_\node, \mflux), -\mflux)
		& = \frac{2 (- \mflux) \RsT \ln\left(\frac{(\pubaux{\press(x,
					\press_\node, \mflux)}) \press_\node}{(\pubu) \press(x,
				\press_\node, \mflux)}\right) - \RsT \abs{-\mflux} \nfric
			x}{F'(\press_\node)}\\
		& = \frac{-2 \mflux \mfluxaux \press_\node (\pubu)}{\plbu}
	\end{align*}
	by using the definition~\eqref{eq:massflux-aux} of $\mfluxaux$.
	Then, the second-order derivative is given by
	\begin{equation*}
		\frac{\partial^2 \press}{\partial \mflux^2}(x, \press(x,
		\press_\node, \mflux), -\mflux)
		= \press_\node (\pubu) \frac{\left(-2 \mfluxaux - 2 \mflux
			\dpar{\mflux}{\mfluxaux}\right) (\plbu) - (4 \mflux^2
			\mfluxaux \RsT)}{(\plbu)^2}.
	\end{equation*}
	Because of $\sign(\mfluxaux) = \sign(\mflux) = -1$ and
	$\dpar{\mflux}{\mfluxaux} > 0$, the claim follows.
\end{proof}

Now, we gathered all tools that are needed to find the optimal
solution or at least a suitable upper bound for the NLPs $\max_{x \in
	\feasset} \Abs{\dparshort{\press_\node}{\press_\otherNode}}$ and
$\max_{x \in   \feasset} \Abs{\dparshort{\mflux}{\press_\otherNode}}$
that are needed to compute the weighted norm~\eqref{eq:weighted-norm}
for~\eqref{eq:box-bounds-rhs}.
\begin{proof}[Proof of Theorem~\ref{thm:max-deriv-p}]
	The claim is a direct consequence of
	Lemma~\ref{thm:pressure_function_deriv_pp} and
	Lemma~\ref{thm:pressure_function_deriv_pq_sign}.
\end{proof}
\begin{proof}[Proof of Theorem~\ref{thm:max-deriv-q}]
	The first property is a direct consequence of
	\eqref{eq:pressure_function_deriv_qp_sign} and
	Lemma~\ref{thm:pressure_function_deriv_qq}.

	Now, we prove the second property.
	From~\eqref{eq:pressure_function_deriv_qp_sign} it follows that
	$\press_\node = \press_\node^-$ must hold in the minimum.
	We know from Lemma~\ref{thm:deriv-q-pos-flow} that
	\begin{equation*}
		\dparshort{\mflux}{\press_\otherNode}(\press_\node^-, \mflux)
		>
		\dparshort{\mflux}{\press_\otherNode}(\press_\otherNode(\press_\node^-,
		\mflux), -\mflux)
	\end{equation*}
	holds for all $\mflux \in [\mflux^-, \mflux^+] \cap \R_{\leq 0}$.
	Since we are dealing with a compact set, this also holds for the
	minima:
	\begin{equation*}
		\min_{\mflux \in [\mflux^-, \mflux^+] \cap \R_{\leq 0}}
		\dparshort{\mflux}{\press_\otherNode}(\press_\node^-, \mflux)
		> \min_{\mflux \in [\mflux^-, \mflux^+] \cap \R_{\leq 0}}
		\dparshort{\mflux}{\press_\otherNode}(\press_\otherNode(\press_\node^-,
		\mflux), -\mflux).
	\end{equation*}
	Hence, the claim follows from Lemma~\ref{thm:deriv-q-pos-flow-deriv-q}.
\end{proof}

%% file: multidim-lipschitz-minlp.bib
@book{Dempe:2002,
  author    = {Dempe, Stephan},
  publisher = {Springer},
  year      = {2002},
  doi       = {10.1007/b101970},
  title     = {Foundations of bilevel programming},
}

@article{Bonami2008,
  author       = {Bonami, P. and Biegler, L. T. and Conn, A. R. and Cornuéjols, G. and Grossmann, I. E. and Laird, C. D. and Lee, J. and Lodi, A. and Margot, F. and Sawaya, N. and Wächter, A.},
  year         = {2008},
  doi          = {10.1016/j.disopt.2006.10.011},
  journal      = {Discrete Optimization},
  pages        = {186--204},
  title        = {An algorithmic framework for convex mixed integer nonlinear programs},
  volume       = {15},
}

@inproceedings{Kannan_Monma:1978,
  author    = {Kannan, R. and Monma, C. L.},
  editor    = {Henn, R. and Korte, B. and Oettli, W.},
  location  = {Berlin, Heidelberg},
  publisher = {Springer Berlin Heidelberg},
  booktitle = {Optimization and Operations Research: Proceedings of a Workshop Held at the University of Bonn, October 2--8, 1977},
  year      = {1978},
  doi       = {10.1007/978-3-642-95322-4_17},
  isbn      = {978-3-642-95322-4},
  pages     = {161--172},
  title     = {On the Computational Complexity of Integer Programming Problems},
}

@book{Garey_Johnson:1979,
  author    = {Garey, M. R. and Johnson, D. S.},
  location  = {New York, NY, USA},
  publisher = {W. H. Freeman \& Co.},
  year      = {1979},
  isbn      = {0716710447},
  title     = {Computers and Intractability: A Guide to the Theory of NP-Completeness},
}

@book{Tawarmalani_Sahinidis:2002,
  author    = {Tawarmalani, M. and Sahinidis, N. V.},
  publisher = {Springer Science \& Business Media},
  year      = {2002},
  doi       = {10.1007/978-1-4757-3532-1},
  subtitle  = {Theory, Algorithms, Software, and Applications},
  title     = {Convexification and Global Optimization in Continuous and Mixed-Integer Nonlinear Programming},
  volume    = {65},
}

@article{Horst1987,
  author       = {Horst, R. and Tuy, H.},
  year         = {1987},
  doi          = {10.1007/BF00939434},
  issn         = {1573-2878},
  journal      = {Journal of Optimization Theory and Applications},
  number       = {2},
  pages        = {253--271},
  title        = {On the convergence of global methods in multiextremal optimization},
  volume       = {54},
}

@article{Horst1988,
  author       = {Horst, R.},
  year         = {1988},
  doi          = {10.1007/BF00939768},
  issn         = {1573-2878},
  journal      = {Journal of Optimization Theory and Applications},
  number       = {1},
  pages        = {11--37},
  title        = {Deterministic global optimization with partition sets whose feasibility is not known: Application to concave minimization, reverse convex constraints, DC-programming, and Lipschitzian optimization},
  volume       = {58},
}

@article{Pinter1988,
  author       = {Pintér, J.},
  year         = {1988},
  doi          = {10.1080/02331938808843322},
  journal      = {Optimization},
  number       = {1},
  pages        = {101--110},
  title        = {Branch- and bound algorithms for solving global optimization problems with Lipschitzian structure},
  volume       = {19},
}

@article{Gugat_et_al:2018b,
  title={Towards Simulation Based Mixed-Integer Optimization with Differential Equations},
  author={Gugat, Martin and Leugering, G{\"u}nter and Martin, Alexander and
  Schmidt, Martin and Sirvent, Mathias and Wintergerst, David},
  journal={Networks},
  volume={72},
  number={1},
  pages={60--83},
  year={2018},
  publisher={Wiley Online Library},
  doi={10.1002/net.21812}
}

@book{Koch_et_al:2015,
  author    = {Koch, T. and Hiller, B. and Pfetsch, M. E. and Schewe, L.},
  publisher = {SIAM},
  booktitle = {Evaluating Gas Network Capacities},
  year      = {2015},
  doi       = {10.1137/1.9781611973693},
  isbn      = {978-1-611973-68-6},
  series    = {SIAM-MOS series on Optimization},
  title     = {Evaluating Gas Network Capacities},
}

@article{Pfetsch_et_al:2015,
  author       = {M. E. Pfetsch and A. Fügenschuh and B. Geißler and N. Geißler and R. Gollmer and B. Hiller and J. Humpola and T. Koch and T. Lehmann and A. Martin and A. Morsi and J. Rövekamp and L. Schewe and M. Schmidt and R. Schultz and R. Schwarz and J. Schweiger and C. Stangl and M. C. Steinbach and S. Vigerske and B. M. Willert},
  doi          = {10.1080/10556788.2014.888426},
  journaltitle = {Optimization Methods and Software},
  number       = {1},
  pages        = {15--53},
  title        = {Validation of nominations in gas network optimization: models, methods, and solutions},
  volume       = {30},
  year         = {2015},
}

@article{Humpola_et_al:2015,
  author       = {Schmidt, M. and Aßmann, D. and Burlacu, R. and Humpola, J. and Joormann, I. and Kanelakis, N. and Koch, T. and Oucherif, D. and Pfetsch, M. E. and Schewe, L. and Schwarz, R. and Sirvent, M.},
  doi          = {10.3390/data2040040},
  journaltitle = {Data},
  number       = {4},
  title        = {{G}as{L}ib---{A} {L}ibrary of {G}as {N}etwork {I}nstances},
  volume       = {2},
  year         = {2017},
}

@book{Kralik_et_al:1988,
  author    = {Králik, J. and Stiegler, P. and Vostrý, Z. and Záworka, J.},
  location  = {New York},
  publisher = {Elsevier Sci. Publ.},
  year      = {1988},
  series    = {Studies in Automation and Control},
  title     = {Dynamic Modeling of Large-Scale Networks with Application to Gas Distribution},
  volume    = {6},
}

@article{Gugat_et_al:2018c,
  title={Networks of pipelines for gas with nonconstant compressibility factor:
  stationary states},
  author={Gugat, Martin and Schultz, R{\"u}diger and Wintergerst, David},
  journal={Computational and Applied Mathematics},
  volume={37},
  number={2},
  pages={1066--1097},
  year={2018},
  publisher={Springer},
  doi={10.1007/s40314-016-0383-z}
}

@report{Still:2018,
  title = {Lectures on Parametric Optimization: An Introduction},
  author = {Still, Georg},
  url = {http://www.optimization-online.org/DB_HTML/2018/04/6587.html},
  year = {2018},
  institution = {University of Twente, The Netherlands}
}

@article{Hoffman:1952,
  title = {On approximate solutions of systems of linear inequalities},
  author = {Hoffman, Alan J.},
  journal = {Journal of Research of the National Bureau of Standards},
  volume = {49},
  number = {4},
  pages = {263--265},
  year = {1952}
}

@article{Li:1993,
  title = {The sharp Lipschitz constants for feasible and optimal solutions of a perturbed linear program},
  author = {Li, Wu},
  journal = {Linear Algebra and its Applications},
  volume = {187},
  pages = {15--40},
  year = {1993},
  doi = {10.1016/0024-3795(93)90125-8}
}

@misc{Pena:2018,
      title = {An algorithm to compute the Hoffman constant of a system of linear constraints},
      author = {Javier Peña and Juan Vera and Luis F. Zuluaga},
      year = {2018},
      eprint = {1804.08418},
      archivePrefix = {arXiv},
      primaryClass = {math.OC}
}

@incollection{Bussieck_Meeraus:2004,
  title={General algebraic modeling system (GAMS)},
  author={Bussieck, Michael R and Meeraus, Alex},
  booktitle={Modeling languages in mathematical optimization},
  pages={137--157},
  year={2004},
  publisher={Springer},
  doi={10.1007/978-1-4613-0215-5_8}
}

@manual{cplex2009v12,
  title={V12. 1: User’s Manual for CPLEX},
  author={{CPLEX, IBM ILOG}},
  journal={International Business Machines Corporation},
  volume={46},
  number={53},
  pages={157},
  year={2009}
}

@article{Gill_et_al:2005,
  author = {Gill, Philip E. and Murray, Walter and Saunders, Michael A.},
  title = {{SNOPT}: {A}n {SQP} {A}lgorithm for {L}arge-{S}cale {C}onstrained {O}ptimization},
  journal = {SIAM Review},
  volume = {47},
  number = {1},
  pages = {99--131},
  year = {2005},
  doi = {10.1137/S0036144504446096},
}

@book{hart2017pyomo,
  title={Pyomo---Optimization Modeling in Python},
  author={Hart, William E. and Laird, Carl D. and Watson, Jean-Paul and Woodruff, David L. and Hackebeil, Gabriel A. and Nicholson, Bethany L. and Siirola, John D.},
  volume={67},
  year={2017},
  publisher={Springer},
  doi={10.1007/978-3-319-58821-6}
}

@manual{gurobi,
  author = {{Gurobi Optimization, LLC}},
  url    = {http://www.gurobi.com},
  date   = 2022,
  title  = {Gurobi Optimizer Reference Manual},
}

@article{BASBLib,
  author       = {Remigijus Paulavicius and Claire S. Adjiman},
  title        = {BASBLib - a library of bilevel test problems},
  year         = {2017},
  journal      = {Zenodo},
  publisher    = {Zenodo},
  version      = {v2.2},
  doi          = {10.5281/zenodo.897966},
}

@article{Schmidt_et_al:2021b,
  author   = {Schmidt, Martin and Sirvent, Mathias and Wollner, Winnifried},
  title    = {The Cost of Not Knowing Enough: Mixed-Integer Optimization with Implicit Lipschitz Nonlinearities},
  number   = {5},
  pages    = {1355--1372},
  year     = {2022},
  volume   = {16},
  isbn     = {1862-4480},
  doi      = {10.1007/s11590-021-01827-9},
  url-opto = {http://www.optimization-online.org/DB_HTML/2018/04/6594.html},
  url-opus = {https://opus4.kobv.de/opus4-trr154/frontdoor/index/index/docId/235},
  journal  = {Optimization Letters},
}

@article{Schmidt_et_al:2019,
  author   = {Martin Schmidt and Mathias Sirvent and Winnifried Wollner},
  title    = {A Decomposition Method for {MINLP}s with {L}ipschitz Continuous Nonlinearities},
  journal  = {Mathematical Programming},
  year     = {2019},
  volume   = {178},
  number   = {1},
  pages    = {449--483},
  issn     = {1436-4646},
  doi      = {10.1007/s10107-018-1309-x},
  url-opto = {http://www.optimization-online.org/DB_HTML/2017/07/6130.html},
  url-opus = {https://opus4.kobv.de/opus4-trr154/frontdoor/index/index/docId/145},
}

@book{Pinter1996,
  author    = {Pintér, J. D.},
  publisher = {Springer {V}erlag},
  year      = {1996},
  doi       = {10.1007/978-1-4757-2502-5},
  isbn      = {978-1-4757-2502-5},
  title     = {Global {O}ptimization in {A}ction ({C}ontinuous and Lipschitz Optimization: {A}lgorithms, {I}mplementations and {A}pplications)},
}

@article{Pinter1986c,
  author       = {Pintér, J.},
  year         = {1986},
  doi          = {10.1080/02331938608843118},
  journal      = {Optimization},
  number       = {2},
  pages        = {187--202},
  title        = {Globally convergent methods for $n$-dimensional multiextremal optimization},
  volume       = {17},
}

@article{Tuy1988,
  author       = {Tuy, H. and Horst, R.},
  year         = {1988},
  doi          = {10.1007/BF01580762},
  issn         = {1436-4646},
  journal      = {Mathematical Programming},
  number       = {1},
  pages        = {161--183},
  title        = {Convergence and restart in branch-and-bound algorithms for global optimization. Application to concave minimization and D.C. Optimization problems},
  volume       = {41},
}

@book{Horst1996,
  author    = {Horst, R. and Tuy, H.},
  location  = {Berlin},
  publisher = {Springer {V}erlag},
  year      = {1996},
  doi       = {10.1007/978-3-662-03199-5},
  edition   = {2nd},
  isbn      = {978-3-540-61038-0},
  title     = {Global {O}ptimization},
}

@article{Belotti2013U,
  author       = {Belotti, P. and Kirches, C. and Leyffer, S. and Linderoth, J. and Luedtke, J. and Mahajan, A.},
  year         = {2013},
  doi          = {10.1017/S0962492913000032},
  issn         = {1474-0508},
  journal      = {Acta Numerica},
  pages        = {1--131},
  title        = {Mixed-integer nonlinear optimization},
  volume       = {22},
}

@techreport{Beck_et_al:2022,
  title    = {A Survey on Bilevel Optimization Under Uncertainty},
  author   = {Yasmine Beck and Ivana Ljubi\'c and Martin Schmidt},
  url      = {http://www.optimization-online.org/DB_HTML/2022/06/8963.html},
  year     = {2022},
}

@article{Kleinert_et_al:2021c,
  title    = {A Survey on Mixed-Integer Programming Techniques in Bilevel Optimization},
  author   = {Thomas Kleinert and Martine Labbé and Ivana Ljubi\'c and Martin Schmidt},
  year     = {2021},
  journal  = {EURO Journal on Computational Optimization},
  issn     = {2192-4406},
  volume   = {9},
  pages    = {100007},
  doi      = {10.1016/j.ejco.2021.100007},
  url-opto = {http://www.optimization-online.org/DB_HTML/2021/01/8187.html},
  url-opus = {https://opus4.kobv.de/opus4-trr154/frontdoor/index/index/docId/361},
}

@inbook{Dempe:2020,
  author    = {Dempe, Stephan},
  editor    = {Dempe, Stephan and Zemkoho, Alain},
  publisher = {Springer International Publishing},
  booktitle = {Bilevel Optimization: Advances and Next Challenges},
  date      = {2020},
  doi       = {10.1007/978-3-030-52119-6_20},
  pages     = {581--672},
  title     = {Bilevel Optimization: Theory, Algorithms, Applications and a Bibliography},
}

@article{Duran1986,
  author       = {Duran, M. A. and Grossmann, I. E.},
  year         = {1986},
  doi          = {10.1007/BF02592064},
  issn         = {1436-4646},
  journal      = {Mathematical Programming},
  number       = {3},
  pages        = {307--339},
  title        = {An outer-approximation algorithm for a class of mixed-integer nonlinear programs},
  volume       = {36},
}

@article{Fletcher1994,
  author       = {Fletcher, R. and Leyffer, S.},
  year         = {1994},
  doi          = {10.1007/BF01581153},
  issn         = {1436-4646},
  journal      = {Mathematical Programming},
  number       = {1},
  pages        = {327--349},
  title        = {Solving mixed integer nonlinear programs by outer approximation},
  volume       = {66},
}

@article{KRONQVIST2019105,
title = {A center-cut algorithm for quickly obtaining feasible solutions and solving convex MINLP problems},
journal = {Computers \& Chemical Engineering},
volume = {122},
pages = {105--113},
year = {2019},
note = {2017 Edition of the European Symposium on Computer Aided Process Engineering (ESCAPE-27)},
issn = {0098-1354},
doi = {10.1016/j.compchemeng.2018.06.019},
author = {J. Kronqvist and D.E. Bernal and A. Lundell and T. Westerlund},
}

@article{Kronqvist_et_al:2019,
  author = {Kronqvist, Jan and Bernal, David E. and Lundell, Andreas and Grossmann, Ignacio E.},
  journal = {Optimization and Engineering},
  title = {A review and comparison of solvers for convex MINLP},
  doi = {10.1007/s11081-018-9411-8},
  year = {2019},
  volumne = {20},
  number = {2},
  pages = {397--455}
}

@article{Al-Khayyal_Sherali:2000,
  author       = {Al-Khayyal, F. A. and Sherali, H. D.},
  year         = {2000},
  doi          = {10.1137/S105262349935178X},
  journal      = {SIAM Journal on Optimization},
  number       = {4},
  pages        = {1049--1057},
  title        = {On Finitely Terminating Branch-and-Bound Algorithms for Some Global Optimization Problems},
  volume       = {10},
}

@article{Smith_Pantelides:1997,
  author       = {Smith, E. M. B. and Pantelides, C. C.},
  year         = {1997},
  doi          = {10.1016/S0098-1354(97)87599-0},
  issn         = {0098-1354},
  journal      = {Computers \& Chemical Engineering},
  pages        = {S791--S796},
  title        = {Global optimisation of nonconvex {MINLP}s},
  volume       = {21},
}

@article{HorstTuy1988,
  author       = {Horst, R. and Thoai, Ng V.},
  year         = {1988},
  doi          = {10.1007/BF00939776},
  issn         = {1573-2878},
  journal      = {Journal of Optimization Theory and Applications},
  number       = {1},
  pages        = {139--145},
  title        = {Branch-and-bound methods for solving systems of Lipschitzian equations and inequalities},
  volume       = {58},
}

@article{Rose_et_al:2016,
  author       = {Daniel Rose and Martin Schmidt and Marc C. Steinbach and Bernhard M. Willert},
  title        = {Computational optimization of gas compressor stations: {MINLP} models versus continuous reformulations},
  year         = {2016},
  doi          = {10.1007/s00186-016-0533-5},
  pages        = {409--444},
  issn         = {1432-2994},
  volume       = {83},
  number       = {3},
  journal      = {Mathematical Methods of Operations Research},
  preprint-url = {http://www.optimization-online.org/DB_HTML/2h015/02/4793.html},
}

@article{Buchheim_et_al:2018,
  author  = {Buchheim, Christoph and Kuhlmann, Renke and Meyer, Christian},
  doi     = {10.1007/s10589-018-9993-2},
  isbn    = {1573-2894},
  journal = {Computational Optimization and Applications},
  number  = {3},
  pages   = {641--675},
  title   = {Combinatorial optimal control of semilinear elliptic {PDEs}},
  volume  = {70},
  year    = {2018},
}

@article{Kleinert_et_al:2021b,
  title    = {Outer Approximation for Global Optimization of Mixed-Integer Quadratic Bilevel Problems},
  author   = {Thomas Kleinert and Veronika Grimm and Martin Schmidt},
  year     = {2021},
  journal  = {Mathematical Programming (Series B)},
  doi      = {10.1007/s10107-020-01601-2},
  isbn     = {1436-4646},
}

@Article{BeckerMeidnerVexler:2007,
  Title                    = {Efficient numerical solution of parabolic optimization problems by finite element methods},
  Author                   = {Becker, R. and Meidner, D. and Vexler, B.},
  Journal                  = {Optimization Methods and Software},
  Year                     = {2007},
  Number                   = {5},
  Pages                    = {813--833},
  Volume                   = {22},
  Doi                      = {10.1080/10556780701228532},
}

@Book{HinzePinnauUlbrichUlbrich:2009,
  Title                    = {Optimization with PDE Constraints},
  Author                   = {Michael Hinze and Rene Pinnau and Michael Ulbrich and Stefan Ulbrich},
  Publisher                = {Springer},
  Year                     = {2009},
  Series                   = {Mathematical Modelling: Theory and Applications},
  Volume                   = {23},
  Doi = {10.1007/978-1-4020-8839-1}
}

@Book{Troeltzsch:2010,
  Title                    = {Optimal control of partial differential equations},
  Author                   = {Tr\"oltzsch, Fredi},
  Publisher                = {American Mathematical Society, Providence, RI},
  Year                     = {2010},
  Series                   = {Graduate Studies in Mathematics},
  Volume                   = {112},
  Doi                      = {10.1090/gsm/112}
}

@incollection{Bajaj_et_al:20221,
  title     = {Black-Box Optimization: Methods and Applications},
  author    = {Bajaj, Ishan and Arora, Akhil and Hasan, M. M.},
  booktitle = {Black Box Optimization, Machine Learning, and No-Free Lunch Theorems},
  pages     = {35--65},
  year      = {2021},
  publisher = {Springer},
  doi       = {10.1007/978-3-030-66515-9_2},
}

@book{Conn_et_al:2009,
  title     = {Introduction to derivative-free optimization},
  author    = {Conn, Andrew R and Scheinberg, Katya and Vicente, Luis N},
  year      = {2009},
  publisher = {SIAM},
  doi       = {10.1137/1.9780898718768},
  series    = {MOS-SIAM Series on Optimization},
}

@book{Dempe-et-al:2015,
  author    = {Dempe, Stephan and Kalashnikov, Vyacheslav and Pérez-Valdés, Gerardo A. and Kalashnykova, Nataliya},
  publisher = {Springer},
  date      = {2015},
  doi       = {10.1007/978-3-662-45827-3},
  title     = {Bilevel Programming Problems},
}

@techreport{Beck_et_al:2022c,
  title     = {On a Computationally Ill-Behaved Bilevel Problem with a Continuous and Nonconvex Lower Level},
  author    = {Yasmine Beck and Martin Schmidt and Johannes Thürauf and Daniel Bienstock},
  year      = {2022},
  url       = {https://arxiv.org/abs/2202.01033},
}

@techreport{Leyffer_et_al:2008,
  title   = {Branch-and-refine for mixed-integer nonconvex global optimization},
  author  = {Leyffer, Sven and Sartenaer, Annick and Wanufelle, Emilie},
  note    = {Preprint ANL/MCS-P1547-0908, Mathematics and Computer Science Division, Argonne National Laboratory},
  year    = {2008},
  url     = {https://wiki.mcs.anl.gov/leyffer/images/1/15/SOS-OA-ANL.pdf},
}

@article{Paulavicius_et_al:2020,
title = {BASBL: Branch-And-Sandwich BiLevel solver. Implementation and computational study with the BASBLib test set},
journal = {Computers \& Chemical Engineering},
volume = {132},
year = {2020},
doi = {10.1016/j.compchemeng.2019.106609},
author = {R. Paulavi\v{c}ius and J. Gao and P.-M. Kleniati and C. S. Adjiman},
}
